\documentclass[10pt]{article}
\usepackage{amsmath,amssymb,amsthm, euscript, bm}

\usepackage{mathrsfs}

\usepackage{amscd}
\usepackage[all]{xy}

\usepackage{mdframed,xcolor}

\newtheorem{theorem}{Theorem}[section]
\newtheorem{proposition}[theorem]{Proposition}
\newtheorem{observation}[theorem]{Observation}
\newtheorem{lemma}[theorem]{Lemma}
\newtheorem{corollary}[theorem]{Corollary}
\theoremstyle{definition}
\newtheorem{definition}[theorem]{Definition}

\newtheorem{remark}[theorem]{Remark}

\newcommand{\N}{\mathbb{N}}
\newcommand{\Z}{\mathbb{Z}}
\newcommand{\Q}{\mathbb{Q}}
\newcommand{\R}{\mathbb{R}}
\newcommand{\C}{\mathbb{C}}

\newcommand{\D}{\mathbb{D}}

\newcommand{\AF}{{\it A}\hspace{-.3em} \raisebox{+.0235em}{\sf F}}

\newcommand{\bara}{{\overline{a}}}
\newcommand{\barE}{{\overline{E}}}
\newcommand{\barN}{{\overline{N}}}
\newcommand{\barU}{{\overline{U}}}
\newcommand{\barW}{{\overline{W}}}

\newcommand{\barx}{{\overline{x}}}
\newcommand{\bardelta}{{\overline{\delta}}}
\newcommand{\bareta}{{\overline{\eta}}}
\newcommand{\bariota}{{\overline{\iota}}}
\newcommand{\barphi}{{\overline{\varphi}}}
\newcommand{\barsigma}{{\overline{\sigma}}}

\newcommand{\barw}{{\overline{w}}}
\newcommand{\bary}{{\overline{y}}}
\newcommand{\barepsilon}{\overline{\epsilon}}
\newcommand{\barkappa}{\overline{\kappa}}

\newcommand{\hatkappa}{\widehat{\kappa}}
\newcommand{\hatiota}{\widehat{\iota}}
\newcommand{\hatW}{\widehat{W}}

\newcommand{\Js}{\mathcal{Z}}

\newcommand{\bfk}{{\bf k}}

\newcommand{\T}{\ensuremath{\mathbb{T}}}
\newcommand{\bh}{\overline{h}}
\newcommand{\br}{\overline{r}}
\newcommand{\Cs}{$\mathrm{C}^*$-al\-ge\-bra}
\newcommand{\Css}{$\mathrm{C}^*$-sub\-al\-ge\-bra}
\newcommand{\cA}{\mathscr{A}}

\newcommand{\cG}{\mathcal{G}}
\newcommand{\cK}{\mathcal{K}}
\newcommand{\cI}{\mathcal{I}}
\newcommand{\cT}{\mathcal{T}}
\newcommand{\cU}{\mathcal{U}}
\newcommand{\cR}{\mathcal{R}}

\newcommand{\cZ}{{\mathcal{Z}}}

\newcommand{\oneLnplus}{{1_{M_{2^{L_{n+1}}}\otimes M_{3^{L_{n+1}}}}}}

\newcommand{\PhiT}{{\Phi_{\mathcal{ T}}}}

\newcommand{\subsetepsilon}{\subset_{\varepsilon}}
\newcommand{\subsetdelta}{\subset_{\delta}}

\newcommand{\tcR}{\widetilde{\mathcal{R}}}
\newcommand{\tcU}{\widetilde{\mathcal{U}}}

\newcommand{\tdelta}{{\widetilde{\delta}}}

\newcommand{\tepsilon}{{\tilde{\epsilon}}}
\newcommand{\tf}{{\tilde{f}}}
\newcommand{\tg}{{\tilde{g}}}

\newcommand{\tL}{{\widetilde{L}}}
\newcommand{\tpsi}{{\widetilde{\psi}}}
\newcommand{\tphi}{{\tilde{\varphi}}}
\newcommand{\tkappa}{{\widetilde{\kappa}}}
\newcommand{\tp}{{\widetilde{p}}}
\newcommand{\tq}{{\tilde{q}}}
\newcommand{\ts}{{\widetilde{s}}}
\newcommand{\tU}{{\widetilde{U}}}
\newcommand{\tW}{{\widetilde{W}}}

\newcommand{\ty}{{\tilde{y}}}

\newcommand{\tr}{{\mathrm{tr}}}

\newcommand{\tildevarphi}{{\widetilde{\varphi}}}

\DeclareMathOperator{\Image}{Im}
\DeclareMathOperator{\id}{id}
\DeclareMathOperator{\Ad}{Ad}
\DeclareMathOperator{\Aff}{Aff}
\DeclareMathOperator{\Aut}{Aut}
\DeclareMathOperator{\ev}{ev}
\DeclareMathOperator{\Ev}{Ev}

\DeclareMathOperator{\Lip}{Lip}
\DeclareMathOperator{\supp}{supp}
\DeclareMathOperator{\Ped}{Ped}
\DeclareMathOperator{\rank}{rank}

\begin{document}
\title{Rationally AF algebras and KMS states \\ of $\mathcal{Z}$-absorbing \Cs s}
\author{
George A. Elliott\thanks{The first named author was supported by a Natural Sciences and Engineering Research Council of Canada Discovery Grant. } 
\ \ and\ \ \
  Yasuhiko Sato\thanks{The second author was supported by JSPS (the Grant-in-Aid for Research Activity Start-up
25887031) and the Department of Mathematical Sciences,  Kyushu University.}\\ 
}
\date{}

\maketitle
\begin{abstract} 
In order to realize all possible KMS-bundles on the Jiang-Su algebra, we introduce a class of \Cs s which we call rationally approximately finite dimensional (RAF). Using these, we show  that for a given proper simplex bundle $(S, \pi)$ with a singleton $\pi^{-1}(\{0\})$ and a unital separable monotracial \Cs{} $A$ absorbing the Jiang-Su algebra tensorially (for instance, the irrational rotation algebra),  there exists a  flow on $A$ whose KMS-bundle is isomorphic to $(S, \pi)$. 
\end{abstract}

\section{Introduction}\label{Sec1}

Approximately finite-dimensional separable \Cs s (AF algebras) were classified, roughly fifty years ago, in \cite{Gl}, \cite{Dix}, \cite{Br}, and \cite{Ell01}, in terms of the Murray-von Neumann semigroup of equivalence classes of projections---equivalently, the universal enveloping ordered group that has come to be known (since Murray and von Neumann) as $K_0$ (and in the setting of AF algebras, the dimension group).

A very general class of simple separable \Cs s, assumed to be well behaved but satisfying very simple abstract axioms, has now been classified by a generalization of this invariant; see \cite{K1}, \cite{K2}, \cite{KP}, \cite{P}, \cite{GLN1}, \cite{GLN2}, \cite{EN1}, \cite{EGLN1}, \cite{TWW}, \cite{CETWW}, \cite{EGLN2}, \cite{GL1}, and \cite{GL2}, all of which follow on substantial  earlier work over the last thirty to forty years. The axioms are amenability,  absorption of the Jiang-Su algebra $\mathcal{Z}$, and the Universal Coefficient Theorem (conceivably redundant in the amenable case). Recall that the tensor product of any \Cs{} with the Jiang-Su algebra absorbs the Jiang-Su algebra tensorially. In many examples (see \cite{EN2}), this property already holds. The invariant, in the stable case, is $K_*=K_0\oplus K_1$ together with the natural pairing with traces, where both $K_0$ and $K_1$ are arbitrary countable abelian groups, and the traces constitute an arbitrary simplicial cone paired in an arbitrary way with $K_0$ (the natural order in which is determined by this pairing). 

At the same time, considerable progress has been made in the last thirty years in the classification of non-simple \Cs s beyond the case of AF algebras. In the infinite ($\mathcal{O}_\infty$-stable)
 case an ideal-related $KK$-equivalence based isomorphism theorem was outlined in \cite{K2}; a different proof of this was given in \cite{Ga}. In the finite (non-simple) case, after earlier results, a $K$-theoretical classification of inductive limits of finite direct sums of matrix algebras over commutative \Cs s (AH, or approximately homogeneous, algebras) with no dimension growth in the spectra and with the ideal property (every closed two-sided ideal generated by projections) was given in \cite{GJL}. The next step in this direction would be to classify inductive limits of sequences of \Css s of  matrix algebras over commutative \Cs s (ASH, or approximately subhomogeneous, algebras), with no dimension growth and with the ideal property. 
 
 Unexpectedly, in an investigation of KMS-state behavior of one-parameter automorphism groups of a \Cs, along the lines of \cite{BEK1} and \cite{BEK2}, a class of ASH algebras arose which it was possible to classify. 
 This is the basis of the present paper, which uses this new classification theorem to construct specified KMS-state phenomena in the Jiang-Su \Cs, and hence in the tensor product of this algebra with any other monotracial \Cs.

 The class of algebras in question (see Section 4), to be named rationally AF algebras (following the terminology of \cite{GLN1} and \cite{GLN2}), or RAF algebras, can be characterized as those \Cs s such that the tensor product with every infinite-dimensional UHF (Glimm uniformly hyperfinite) algebra is AF. 
 
 By \cite{GJLP}, every \Cs{} which is both RAF and AH with no dimension growth is AF. 
 (To see this, one may assume that the algebra is stable, and then as shown in the proof of Corollary \ref{Cor5.4}, below, it has the ideal property. By \cite{GJLP} it is AT, and since by Lemma \ref{Lem4.2}, it has zero $K_1$ group, by \cite{Thomsen-1} it is AI. By Corollary 1.3 of \cite{BEll} any AH algebra with no dimension growth has real rank zero if it has small eigenvalue variation, which holds if the tensor product with a UHF algebra has real rank zero. It is well known that an AI algebra of real rank zero is AF.) 
 
 The present class, thus, projects directly into the unknown territory of the non-simple ASH class.
 
  Interestingly (see Theorem \ref{Thm5.3}), the classifying invariant for this class---or, rather, the Jiang-Su stable members of the class---, at least up to stable isomorphism, is exactly the same as for the class of AF algebras, namely, the ordered $K_0$-group, which can be an arbitrary countable ordered abelian group whose tensor product with every dense subgroup of the rational numbers is a   dimension group. (We use the terminology rational dimension group, following \cite{LN2} which deals with the simple case. See Definition \ref{Def3.2}, Lemma \ref{Lem3.1} (v), and Theorem \ref{Thm4.4}.) 
  
  The classification up to isomorphism for non-stable algebras is also almost the same as for AF algebras, and in the case that there is an approximate unit consisting of projections (for instance if the algebras are stable or unital), it is exactly the same (namely, the dimension range).
  The general case (appearing already with the non-unital algebras stably isomorphic to the Jiang-Su algebra) is slightly more subtle. See Corollaries \ref{Cor4.8} and \ref{Cor5.4}, which introduce what we shall call the matrix dimension range.

  The proof of Theorem \ref{Thm5.3}, the $\mathcal{Z}$-absorbing RAF-algebra classification theorem, consists of an application of the Winter deformation technique \cite[Proposition 4.5]{W1},
  extended in a simple way beyond the unital case in which it is at present couched.

  Our result concerning KMS-state structure, which follows from the Jiang-Su absorbing RAF classification (Corollaries \ref{Cor4.8} and \ref{Cor5.4}) is Theorem \ref{ThmMain}. (A precise statement is already given in the Abstract.) The method is similar to that of \cite{BEK1} and \cite{BEK2}, and to that of \cite{ETh} and \cite{EST}, and there is some overlap in the results. 
 In \cite{BEK1} and \cite{BEK2} the \Cs s on which the actions are constructed are not precisely identified. In \cite{EST}, the present setting but with only the case of a compact bundle, i.e., bounded set of admissible inverse temperatures, is dealt with---this uses the known simple \Cs{} classification referred to above.  
 The case of an infinite Kirchberg algebra is also dealt with in \cite{EST}, without the compactness restriction. In \cite{ETh}, the only overlap is the case of a monotracial simple unital AF algebra. 
 
As a straightforward application of the present paper, we obtain in particular the following result.
\begin{corollary}\label{Cor1.1}
For any irrational rotation algebra, there exists a flow which realizes any given proper simplex bundle with singleton fibre over $0$ as a KMS-bundle on it. 
\end{corollary}
 
\section{Preliminaries}\label{Sec2}
Let us start with some basic terminology and a fact for ordered abelian groups with the Riesz interpolation property. 

In this paper, a \emph{ partially ordered abelian group} ($G$, $G^+$) will mean an abelian group $G$ with a positive cone $G^+$ satisfying $G^+ + G^+\subset G^+$ and $G^+\cap -G^+=\{0\}$. If just the first condition is satisfied, $(G, G^+)$ will be called a \emph{pre-ordered abelian group} (see \cite{Goodearl}). If, in addition to both conditions, $G=G^+-G^+$, then $(G, G^+)$ will be called an \emph{ordered abelian group} (see \cite[Definition 5.1.3]{RLL}). For a partially ordered (or pre-ordered) abelian group $(G, G^+)$, we shall denote by $\leq$ the order relation defined by $g\leq h$ if $h-g\in G^+$.

A partially ordered abelian group $(G, G^+)$ will be said to have the {\it Riesz (or Birkhoff-Riesz) interpolation property} (RIP) if for $g_0$, $g_1$, $h_0$, $h_1\in G$ satisfying $g_i\leq h_j$ for all $i, j\in\{0, 1\}$ there exists $a\in G$ such that $g_i\leq a\leq h_j$ for all $i, j\in\{0, 1\}$.
(See \cite{Bir}.) 

For a supernatural number ${\frak n}$ (i.e., a formal infinite product of finite or infinite powers of prime numbers; see \cite{Dix}, \cite{Ror0}),   denote by $\D_{\frak n}$ the set of all rational numbers $p/q$ given by $p\in\Z$ and $q\in\N$ with $q | {\frak n}$, and fix the positive cone $\D_{\frak n}^+=\{d\in \D_{\frak n} \ : \ d\geq 0\}$.  Denote by $\D_{\frak n}[e^x, e^{-x}]$  the group of Laurent polynomials on $\R$ with coefficients in $\D_{\frak n}$. For a closed subset $F\subset \R$, equip $\D_{\frak n}[e^x, e^{-x}]$ with the strict pointwise order, making it an ordered abelian group: 
\[\D_{\frak n}[e^x, e^{-x}]_F^+=\{0\}\cup\{f\in \D_{\frak n}[e^x, e^{-x}] \ :\  f(x)>0 \text{ for any } x\in F\}.\] 
We will have occasion to use the following fundamental lemma in the sequel (in Lemma \ref{Lem3.4}). 
\begin{lemma}\label{Lem2.1}
Let $F\subset \R$ be a closed subset of $\R$ and ${\frak n}$ an infinite supernatural number. Then the ordered abelian group $(\D_{\frak n}[e^x, e^{-x}],  \D_{\frak n}[e^x, e^{-x}]_F^+)$ has the  RIP if and only if $F$ is semi-bounded (i.e., either bounded below or bounded above).
\end{lemma}
\begin{proof}
The ``if'' part of the statement is a variant of the argument of  \cite[Section 5]{Thomsen1}. We include a proof for the reader's convenience. Suppose that $F$ is semi-bounded, and set $-F=\{x\in\R\ : \ -x\in F\}$. In the case that $F$ is bounded, the Stone-Weierstrass theorem can be applied directly to show the RIP. So we may assume that $F$ is unbounded. To simplify notation, set $G=\D_{\frak n}[e^x, e^{-x}]$ and $G_F^+=\D[e^x, e^{-x}]_F^+$. It is immediate that $(G,  G_F^+)$ has the RIP if and only if $(G, G_{-F}^+)$ has. Thus, we may assume that $F$ is bounded below.

Let $p_i$, $q_j\in G$, $i, j\in\{0, 1\}$, be such that $p_i\leq q_j$ for all $i, j\in \{0, 1\}$. If $p_{i_0}=q_{j_0}$ for some $i_0, j_0\in \{0, 1\}$, then $a=p_{i_0}=q_{j_0}$ satisfies $p_i\leq a \leq q_j$ for all $i, j\in\{0, 1\}$. Therefore, we may suppose that $p_i(x) < q_j(x)$ for all $i, j\in\{0, 1\}$ and $x\in F$. As the first step in the argument which follows, note that for sufficiently  large $N\in\N$ and $R>0$,  there exists $a_0\in G$ such that 
\[p_i(x) + e^{-Nx}< a_0(x) <q_j(x) -e^{-Nx},\]
for all $i, j\in\{0, 1\}$ and $x\in F\cap [R, \infty)$. 

Let $\D_{\frak n}^{\infty}$ denote the infinite direct sum $\bigoplus_{n\in\Z} \D_{\frak n}$ of copies of $\D_{\frak n}$ over $\Z$, and equip $\D_{\frak n}^{\infty}$ with the reverse lexicographic order $\leq_{\rm lex}$, which means that 
for $g=(g_n)_{n\in\Z}$ and $h=(h_n)_{n\in\Z}\in\D_{\frak n}^{\infty}$, $g\leq_{\rm lex} h$ if $g_n= h_n$ for all $n\in\Z$ or $g_{n_0}< h_{n_0}$ for $n_0=\max\{n\in\Z\ : \ g_n\neq h_n\}$. If $g\leq_{\rm lex} h$ and $g\neq h$,  write $g<_{\rm lex} h$. Since $F$ is unbounded (above), it follows that $(g_n)_{n\in\Z}<_{\rm lex} (h_n)_{n\in\Z}$ if, and only if, there exists $r>0$ such that $\sum_{n\in\Z} g_ne^{nx}<\sum_{n\in\Z}h_ne^{nx}$ for any $x\in F\cap [r, \infty)$. Let $p_i^{\infty}=(p_{i, n})_{n\in\Z}$, $q_j^{\infty}=(q_{j, n})_{n\in\Z}\in \D_{\frak n}^{\infty}$ denote the coefficient sequences of $p_i(x)=\sum_{n\in\Z} p_{i, n}e^{nx}$ and $q_j(x)=\sum_{n\in\Z} q_{j, n}e^{nx}$, $i, j\in\{0, 1\}$. Since the reverse lexicographic order is a total order, reindexing, we may assume that $p_0^{\infty}\leq_{\rm lex} p_1^{\infty} <_{\rm lex} q_0^{\infty} \leq_{\rm lex} q_1^{\infty}$. Set $n_0=\max\{n\in \Z\ : \ p_{1, n}\neq q_{0, n}\}$ and note that $p_{1, n_0}< q_{0, n_0}$. 
For $k\in\Z$, set $\delta_k=(\delta_{k, n})_{n\in\Z}\in\D_{\frak n}^{\infty}$ where $\delta_{k, n}$ is the Kronecker delta, and set $a_0^{\infty}=p_1^{\infty}+\delta_{n_0-1}$. Then we have  $p_1^{\infty} <_{\rm lex} a_0^{\infty} <_{\rm lex} q_0^{\infty}$. Let $N\in\N$ chosen above be chosen sufficiently large that also $p_0^{\infty} +\delta_{-N}\leq_{\rm lex}p_1^{\infty} +\delta_{-N} <_{\rm lex} a_0^{\infty} <_{\rm lex} q_0^{\infty} -\delta_{-N} \leq q_1^{\infty}-\delta_{-N}$, which is the analogue in $\D_{\frak n}^{\infty}$ of the required condition. 

Since by the choice of $a_0$, $e^{Nx}(p_i(x)-a_0(x))< -1<1< e^{Nx}(q_j(x) -a_0(x))$ for all $i, j\in\{0, 1\}$ and $x\in F\cap [R, \infty)$, there exist $d>0$ and a continuous function $f :\R\rightarrow \R$ with compact support such that $e^{Nx}(p_i(x)-a_0(x)) +d < f(x)< e^{Nx}(q_j(x)-a_0(x))-d$ for any $x\in F$. By the Stone-Weierstrass theorem (applied to the compact space $F\cup\{\infty\}$ and the subalgebra of Laurent polynomials bounded on this space), we obtain $a_1\in G$ such that $\sup_{x\in F}|f(x)-a_1(x)|< d$. Defining $a\in G$ by $a(x)=e^{-Nx}a_1(x) + a_0(x)$, we have $p_i\leq a\leq q_j$ for all $i, j\in\{0, 1\}$.

To show the ``only if '' part,  assume that $F$ is not semi-bounded. Let $N\in\N$ and $\beta\in F$ be such that $N<e^{\beta}< 2N$. Define $p_i$, $q_j\in G$, $i, j\in\{0, 1\}$, by $p_0(x)=-1$, $p_1(x)=(N-e^x)(e^x-2N)$, $q_0(x)=e^{2x}$, and $q_1(x)= N^2$ for $x\in\R$. Then it follows that 
$p_i \leq q_j$ for all $i, j\in\{0, 1\}$. If there existed $a\in G$ with $p_i\leq a\leq q_j$ for all $i, j\in\{0, 1\}$, then, as $p_0\leq a \leq q_1$, $a$ would be a bounded function on $F$. Since $F$ is not semi-bounded, $a$ must be a constant function, in contradiction with $p_1\leq a \leq q_0$.
\end{proof}

\section{Rational dimension groups and simplex bundles}\label{Sec3}
The notion of rationally Riesz group was introduced in \cite{LN2}, and it was shown that a weakly unperforated simple ordered abelian group $G$ is a rationally Riesz group if and only if $G\otimes \D_{\frak p}$ has the Riesz interpolation property for every infinite supernatural number ${\frak p}$.  (See \cite{Ror0} for the definition of the supernatural numbers, which extend the natural numbers, which we shall refer to as finite.) In this section, we shall consider a modified definition for non-simple ordered abelian groups (Definition \ref{Def3.2}), and construct an example which gives rise  in a natural way to a given proper simplex bundle (see below), in Proposition \ref{Prop3.5}. In the sequel (see Theorem \ref{ThmMain}), proper simplex bundles will be shown to give rise to 
KMS-bundles on the Jiang-Su algebra.

Recall that the positive cone $(G\otimes H)^+$ of the tensor product of  two partially ordered (or pre-ordered) abelian groups $(G, G^+)$ and $(H, H^+)$ is defined as the set of all finite sums of $\{g\otimes h\ :\ g\in G^+, h\in H^+\}$, and note that $(G\otimes H, (G\otimes H)^+)$ is then a partially ordered (or pre-ordered) abelian group; see \cite{GH2}. For convenience, we shall just write ${G\otimes H}^+$. For an abelian group $G$ we shall denote by $G\otimes 1_{\bfk}$ the subgroup of $G\otimes \D_{\bfk}$ consisting of $\{g\otimes 1_{\bfk}\ :\ g\in G\}$, where $1_{\bfk}$ denotes the number $1\in \D_{\bfk}$.

\begin{lemma}\label{Lem3.1}
Let ${\frak p}$ and ${\frak q}$ be relatively prime supernatural numbers. For an abelian group $G$, the following fundamental statements hold.
\begin{itemize}
\item[ \rm (i)] If $G\otimes \D_{\frak p}$ and $G\otimes \D_{\frak q}$ are torsion-free abelian groups, then so also is $G$.
\item[\rm (ii)] If $G$ is torsion free, then the subgroup $(G\otimes \D_{\frak p}\otimes 1_{\frak q})\cap (G\otimes 1_{\frak p}\otimes \D_{\frak q})$ of $G\otimes \D_{\frak p}\otimes \D_{\frak q}$ is isomorphic to $G$ as an abelian group.
\item[\rm (iii)] If a partially ordered abelian group $(G, G^+)$ is such that $(G\otimes \D_{\frak p}, {G\otimes \D_{\frak p}}^+)$ and $(G\otimes \D_{\frak q}, {G\otimes \D_{\frak q}}^+)$ are torsion-free ordered abelian groups, then $(G, G^+)$ is an ordered abelian group (i.e., $G=G^+-G^+$).
\item[\rm (iv)]
If $(G, G^+)$ is an unperforated ordered abelian group (i.e., $ng\geq 0$ for some $n\in\N$ implies $g\geq 0$; see \cite[Definition 6.7.1]{Bl2}), then so also is $(G\otimes\D_{\frak n}, {G\otimes \D_{\frak n}}^+)$ for any supernatural number ${\frak n}$. 
\item[\rm (v)] If $(G, G^+)$ is an unperforated partially ordered abelian group such that $(G\otimes \D_{\frak p}, {G\otimes\D_{\frak p}}^+)$ and $(G\otimes \D_{\frak q}, {G\otimes \D_{\frak q}}^+)$ are dimension groups, then $(G, G^+)$ is an ordered abelian group and $(G\otimes \D_{\frak n}, {G\otimes\D_{\frak n}}^+)$ is a dimension group for any infinite supernatural number ${\frak n}$. 
\end{itemize}
\end{lemma}
\begin{proof}
\noindent {\rm (i)}. Assume that $g\in G$ is a torsion element and $n\in\N$ is the first number such that $ng=0$. Since $G\otimes \D_{\frak p}$ is torsion free, we have $g \otimes 1_{\frak p}=0$ in $\langle g \rangle \otimes \D_{\frak p}$. Since $\langle g \rangle\otimes \D_{\frak p} \cong (\Z/n\Z)\otimes \D_{\frak p} \cong \D_{\frak p} / n\D_{\frak p}$, it follows that $1\in n\D_{\frak p}$.  Similarly, we have  $1\in n\D_{\frak q}$. 
Because ${\frak p}$ and $\frak{q}$ are relatively prime, this implies that $n=1$ and so $g=ng=0$. 

\noindent {\rm (ii)}. Since $G$ is torsion free, the canonical embedding $\varphi : G\rightarrow G\otimes 1_{\frak p}\otimes 1_{\frak q}$ defined by $\varphi(g) = g\otimes 1_{\frak p}\otimes 1_{\frak q}$ for $g\in G$ is a group isomorphism. Therefore, it suffices to show that $(G\otimes \D_{\frak p}\otimes 1_{\frak q})\cap (G\otimes 1_{\frak p}\otimes \D_{\frak q}) =G\otimes 1_{\frak p}\otimes 1_{\frak q}$. If $x\in (G\otimes \D_{\frak p}\otimes 1_{\frak q})\cap( G\otimes 1_{\frak p}\otimes \D_{\frak q})$,  then there exist $k$, $l\in \N$ such that $1/k\in \D_{\frak p}$, $1/l\in \D_{\frak q}$, and $kx, lx \in G\otimes 1_{\frak p}\otimes 1_{\frak q}$. Since $k$ and $l$ are relatively prime, there exist $a$, $b\in\Z$ such that $1=ak + bl$. It follows that $x=akx + blx\in G\otimes 1_{\frak p}\otimes 1_{\frak q}$. The converse inclusion $(G\otimes \D_{\frak p}\otimes 1_{\frak q}) \cap (G\otimes 1_{\frak p}\otimes \D_{\frak q}) \supset G\otimes 1_{\frak p}\otimes 1_{\frak q}$ is trivial.

\noindent {\rm (iii)}. Fix $x\in G$. For ${\frak i}={\frak p}$, ${\frak q}$, since $(G\otimes \D_{\frak i}, {G\otimes \D_{\frak i}}^+)$ is an ordered group, there exist $a_{\frak i}$, $b_{\frak i}\in {G\otimes \D_{\frak i}}^+$ such that $x\otimes 1_{\frak i}=a_{\frak i}-b_{\frak i}$. Let $n_{\frak i}$ be such that $1/n_{\frak i}\in \D_{\frak i}$ and $n_{\frak i} a_{\frak i}$, $n_{\frak i} b_{\frak i} \in G\otimes 1_{\frak i}$ for ${\frak i}= {\frak p}$, ${\frak q}$. Since $a_{\frak i}$, $b_{\frak i}\in {G\otimes \D_{\frak i}}^+$, there exist $g_{a_{\frak i}}$, $g_{b_{\frak i}}\in G^+$,  $\frak{i}=\frak{p}, \frak{q}$,  such that $n_{\frak i} a_{\frak i}=g_{a_{\frak i}}\otimes 1_{\frak i}$ and $n_{\frak i} b_{\frak i}=g_{b_{\frak i}} \otimes 1_{\frak i}$. By {\rm (i)}, $G$ is torsion free, and so from $n_{\frak i} x \otimes 1_{\frak i}=(g_{a_{\frak i}}-g_{b_{\frak i}})\otimes 1_{\frak i}$ it follows that $n_{\frak i} x= g_{a_{\frak i}}-g_{b_{\frak i}}$. Since $n_{\frak p}$ and $n_{\frak q}$ are relatively prime, there exist $c, d\in \N$ such that $1=cn_{\frak p} -dn_{\frak q}$. Then we have
\[ x= (c g_{a_{\frak p}} +d g_{b_{\frak q}})-(cg_{b_{\frak p}}+ d g_{a_{\frak q}})\in G^+-G^+.\] 

\noindent {\rm (iv)}.  Since $(G, G^+)$ and $(\D_{\frak n}, {\D_{\frak n}}^+)$ are torsion-free ordered abelian groups, it is straightforward to check that $(G\otimes \D_{\frak n}, {G\otimes \D_{\frak n}}^+)$ is also a torsion-free ordered abelian group. Therefore it suffices to show that there exists $d\in \N$ such that $1/d\in \D_{\frak n}$ and $d x\in {G\otimes \D_{\frak n}}^+$ if $x\in G\otimes \D_{\frak n}$ and $n\in\N$ satisfy $nx\in {G\otimes\D_{\frak n}}^+$. Let $g_j\in G$ and $d_j\in \D_{\frak n}$, $j=1, 2, ..., N$, be such that $\sum_{j=1}^N g_j\otimes d_j=x$. Since $nx\in {G\otimes \D_{\frak n}}^+$, there exist $h_m\in G_+$, and $e_m\in {\D_{\frak n}}^+$, $m=1,2,..., M$, such that $ nx=\sum_{m=1}^M h_m\otimes e_m$. Choose $d\in\N$ such that $1/d\in \D_{\frak n}$ and $dd_j$, $de_m\in \Z$ for all $j=1,2,..., N$ and $m=1,2,..., M$. Then it follows that $\left(\sum_{j=1}^N ndd_j g_j\right)\otimes 1_{\frak n} =\left(\sum_{m=1}^M d e_m h_m\right)\otimes 1_{\frak n}$. 
Since $G$ is torsion free, we have $ \sum_{j=1}^N ndd_j g_j=\sum_{m=1}^M de_m h_m \in G^+$. Finally, since $G$ is unperforated, we conclude that $ dx =\left(\sum_{j=1}^N d d_j g_j\right)\otimes 1_{\frak n}\in (G\otimes \D_{\frak n})^+$. 

\noindent {\rm (v)}. By {\rm (i) } and {\rm (iii)}, we see that $(G, G^+)$ is an unperforated ordered abelian group, and by {\rm (iv)} so also is $(G\otimes \D_{\frak n}, {G\otimes \D_{\frak n}}^+)$. Because the Riesz interpolation property in an ordered abelian group is equivalent to the Riesz decomposition property of \cite{Rie} (see \cite{Bir}; see also \cite[Theorem IV 6.2]{Dav}), it suffices to show that for $x$, $a_0$, $a_1\in {G\otimes \D_{\frak n}}^+$ with $x\leq a_0+a_1$ there exist two elements $x_i \in {G\otimes \D_{\frak n}}^+$, $i=0, 1$, such that $x= x_0 + x_1$ and $x_i\leq a_i$ for both $i=0, 1$.

Choose $N\in\N$ such that $1/N \in \D_{\frak n}$ and $Nx$, $Na_i\in G\otimes 1_{\frak n}$, and denote by $g_x$, $g_{a_i}\in G^+$, $i=0, 1$, the elements such that $g_x\otimes 1_{\frak n}=Nx$ and $g_{a_i}\otimes 1_{\frak n}= N a_i$ for $i=0, 1$. Since $g_x\otimes 1_{\frak n} \leq(g_{a_0} + g_{a_1})\otimes 1_{\frak n}$, we have $g_x\leq g_{a_0}+ g_{a_1}$, as $G$ is unperforated. By the Riesz decomposition property of $G\otimes \D_{\frak p}$ and $G\otimes \D_{\frak q}$, for ${\frak j}={\frak p}, {\frak q}$ there are $y_i^{({\frak j})}\in {G\otimes \D_{\frak j}}^+$, $i=0, 1$, such that $g_x\otimes 1_{\frak j} = y_0^{({\frak j})}+y_1^{({\frak j})}$ and $y_i^{({\frak j})}\leq g_{a_i}\otimes 1_{\frak j}$ for both $i=0, 1$. 
Choose $N_{\frak j}\in \N$, ${\frak j}={\frak p}, {\frak q}$,  such that $1/N_{\frak j}\in \D_{\frak j}$ and $N_{\frak j} y_i^{({\frak j})} \in G\otimes 1_{\frak j}$ for both $i=0, 1$, and denote by $g_i^{({\frak j})}\in G^+$ the element such that $N_{\frak j} y_i^{({\frak j})}=g_i^{({\frak j})}\otimes 1_{\frak j}$. Thus we have $N_{\frak j} g_x=g_0^{({\frak j})}+g_1^{(\frak j)}$ for ${\frak j}= {\frak p}, {\frak q}$. 
Since $N_{\frak p}$ and $N_{\frak q}$ are relatively prime, if ${\frak n}$ is infinite, there exist natural numbers $c$, $d$, and $M$ $(> N_{\frak p}N_{\frak q})$ such that $1/(MN)\in\D_{\frak n}$ and $M=cN_{\frak p} + dN_{\frak q}$. The elements $x_i =(cg_i^{({\frak p})}+d g_i^{({\frak q})})\otimes 1/(MN) \in {G\otimes \D_{\frak n}}^+$, $i=0, 1$, then satisfy the desired conditions.
\end{proof}

\begin{definition}\label{Def3.2}
We shall call a partially ordered abelian group $(G, G^+)$ a \emph{rational dimension group} if $(G, G^+)$ is unperforated and, further, $(G\otimes\D_{\frak p}, {G\otimes\D_{\frak p}}^+)$ and $(G\otimes\D_{\frak q}, {G\otimes\D_{\frak q}}^+)$ are dimension groups for some pair of relatively prime supernatural numbers ${\frak p}$ and ${\frak q}$. Note that a partially ordered abelian group is not necessarily unperforated even if $G\otimes\D_{\frak n}$ is a dimension group for every infinite supernatural number ${\frak n}$. For example, the ordered abelian group $(\Z, S_n)$ with $S_n=\{0, n, n+1,...\}$ is not unperforated for $n\in\N\setminus\{1\}$; however, $(\Z\otimes\D_{\frak n}, {S_n\otimes \D_{\frak n}}^+)$ is isomorphic to $(\D_{\frak n}, {\D_{\frak n}}^+)$ as an ordered abelian group for any infinite supernatural number ${\frak n}$.
\end{definition}

\begin{remark}\label{Rem3.3}
For simple ordered abelian groups, it is enough for the tensor product of an unperforated ordered group with $\D_{\frak n}$ for a single supernatural number ${\frak n}$ to be a dimension group in order for it to be a rational dimension group (see \cite[Proposition 5.7]{LN2}). However, for non-simple ordered abelian groups, the following example indicates the necessity of the two tensor products in  the definition above.

Let $p$ and $q$ be two relatively prime numbers, and let ${\frak u}$ denote the universal supernatural number, with $\D_{\frak u}=\Q$. Set $H=\Q^2$, $H^+=\{(x, y)\in H\ :\ x\geq 0, y\geq 0\}$, $a_0=(0, 0)$, $b_0=(1, 0)$, $a_1=(1/p, 1/p-1)$, and $b_1=(1/p, 1/p)$ in $H$. Define $(G, G^+)$ as the ordered subgroup of $(H, H^+)$ generated by $\{a_1, b_1\}$, which is unperforated and non-simple. Note that, for any supernatural number ${\frak n}$, the ordered abelian group $(G\otimes \D_{\frak n}, {G\otimes \D_{\frak n}}^+)$ is isomorphic to 
$(\D_{\frak n}a_1+\D_{\frak n}b_1, (\D_{\frak n}a_1+\D_{\frak n}b_1)\cap H^+)$, and that $a_0$, $b_0\in G$. 
Then the tensor product $(G\otimes \D_{\frak u}, {G\otimes \D_{\frak u}}^+)$ $\cong (H, H^+)$ is a dimension group. However, if there exists $x\in G\otimes \D_{q^{\infty}}$ such that $a_i\otimes 1_{q^{\infty}}\leq x\leq b_j\otimes 1_{q^{\infty}}$, for $i, j\in\{0, 1\}$, then $x$ must correspond to $(1/p, 0)$ via the identification of $G\otimes \D_{q^{\infty}}$ with $\D_{q^{\infty}}a_1+\D_{q^{\infty}}b_1$. Since $p$ and $q$ are relatively prime,  $x$ cannot be contained in $G\otimes \D_{q^{\infty}}$. This shows that $(G\otimes \D_{q^{\infty}}, {G\otimes\D_{q^{\infty}}}^+)$ does not have the RIP.
\end{remark}

\medskip

Let us recall a few notions concerning  simplex bundles, introduced in \cite{BEK1} and \cite{BEK2} (see also \cite{ETh} and \cite{EST}). A pair $(S, \pi)$ consisting of a second countable locally compact Hausdorff space $S$ and a continuous map $\pi : S\rightarrow \R$ is called a \emph{simplex bundle} if $\pi^{-1}(\{\beta\})$ is compact   and has a structure of Choquet simplex for every $\beta\in \R$. To simplify notation, we shall write $\pi^{-1}(\beta)$ for $\pi^{-1}(\{\beta\})$.  Two simplex bundles $(S, \pi)$ and $(S', \pi')$ are said to be \emph{isomorphic} if there exists a homeomorphism $\Phi : S\rightarrow S'$ such that $\pi'\circ \Phi=\pi$ and the restriction of $\Phi$ to $\pi^{-1}(\beta)$ is affine for all $\beta\in \R$. 

For a simplex bundle $(S, \pi)$, we shall denote by $A(S)$ the set of all continuous functions $f$ from $S$ to $\R$ such that each restriction of $f$ to $\pi^{-1}(\beta)$ is affine for any $\beta\in\R$. We also set 
\begin{align*}
A_{0}(S)&=\{f\in A(S) \ :\ f \text{ vanishes at infinity}\}, \\
A_{00}(S)&=\{f\in A(S)\ : \supp(f) \text{ is compact and } \supp(f)\cap \pi^{-1}(0)=\O\}.
\end{align*} 

A simplex bundle $(S, \pi)$ will be called \emph{proper} if $\pi$ is a proper map (i.e. $\pi^{-1}(K)$ is compact for any compact set $K\subset \R$), and $A(S)$ separates points of $S$. 

\begin{lemma}\label{Lem3.3}
Let $(S, \pi)$ be a simplex bundle, $S'$ a locally compact  second countable Hausdorff space, and $\pi' : S'\rightarrow \R$ a continuous map. Suppose that $\pi$ is proper and $\Phi$ is a bijective continuous map from $S$ to $S'$ such that $\pi=\pi'\circ\Phi$. Then $\Phi$ is a homeomorphism. In particular, if $(S, \pi)$ is a proper simplex bundle and the restriction of $\Phi$ to $\pi^{-1}(\beta)$ is affine for any $\beta\in\R$, then $(S', \pi')$ is also a proper simplex bundle.
\end{lemma}
\begin{proof}
For any compact subset $K$ of $S'$, since $\pi'$ is continuous, there exists a compact interval $I\subset \R$ such that $\pi'(K)\subset I$, which implies that $\Phi^{-1}(K)\subset\pi^{-1}(I)$. Since $\pi^{-1}(I)$ is compact, so also is $\Phi^{-1}(K)$. This shows that $\Phi$ is a proper map. Let $F$ be a closed subset of $S$. To show $\Phi^{-1}$ is continuous, it suffices to show that $\Phi(F)$ is closed. For any compact subset $K$ of $S'$, it follows that $\Phi(F)\cap K=\Phi(F\cap \Phi^{-1}(K))$ is compact. Since $S'$ is locally compact and Hausdorff, this shows that $\Phi(F)$ is closed. 

As $\pi=\pi'\circ \Phi$, it follows that ${\pi'}^{-1}(B)=\Phi(\pi^{-1}(B))$ for any subset $B$ of $\R$. Then $\pi'$ is also a proper map. If each restriction of $\Phi$ to $\pi^{-1}(\beta)$ is affine, then we see that $A(S')=\{f\circ\Phi^{-1}\ : \ f\in A(S)\}$ separates points of $S'$. 
\end{proof}

The concept of a proper simplex bundle was presented as an abstract characterization of the KMS-bundle of a \Cs{} in \cite{BEK1}, \cite{BEK2}, \cite{ETh}, and \cite{EST}. On the one hand, it was shown in \cite{BEK1} and \cite{BEK2} (see also \cite[Section 5.3]{BR}) that a KMS-bundle for a unital separable \Cs{} is a proper simplex bundle, and on the other hand, in \cite{ETh}, it was shown that any proper simplex bundle can be realized as a KMS-bundle on a given unital simple infinite-dimensional AF-algebra. The construction in this paper is a variant of \cite[Subsection 4.2]{ETh} for rational dimension groups (instead of just dimension groups). In the rest of this section, we shall consider a proper simplex bundle $(S, \pi)$ such that $\pi^{-1}(0)$ is a singleton $\{\tau_S\}$. 

Let $\cG$ be a countable additive subgroup of $\R$ with $1\in \cG$ and let $\cG^{\infty}$ denote the infinite direct sum $\bigoplus_{n\in\Z} \cG$ of copies of $\cG$ over $\Z$. Given a sequence $g=(g_n)_{n\in\Z}$ in $\cG^{\infty}$, define a continuous function $L(g)\in A(S)$ by 
\[ L(g) (s)= \sum_{n\in\Z} g_n e^{n\pi(s)}\quad \text{ for } s\in S.\]
From now on, we shall denote by $e^{-\pi}$ the function $e^{-\pi(\cdot)}\in A(S)$, so that $L(g)=\sum_{n\in\Z}g_ne^{n\pi}$.  Let $\sigma_S$ denote the automorphism of $A(S)$ defined by $\sigma_S(f) =e^{-\pi} f$ for $f\in A(S)$, and $\sigma_{\cG^{\infty}}$ the automorphism of $\cG^{\infty}$ defined by $\sigma_{\cG^{\infty}}((g_n)_{n\in\Z})=(g_{n+1})_{n\in\Z}$ for $(g_n)_{n\in\Z}\in \cG^{\infty}$. Note that 
\[\sigma_S\circ L= L\circ \sigma_{\cG^{\infty}}.\] 

For $k\in\N$, choose positive continuous functions $\psi_k$, $\psi_{k +}$, $\psi_{k-}$ from $\R$ into $[0, 1]$ such that $(-\infty, -1/k]\subset \psi_{k-}^{-1}(\{1\})$, $[-1/2k, 1/2k]\subset \psi_k^{-1}(\{1\})$, $[1/k, \infty)\subset\psi_{k +}^{-1}(\{1\})$, and $\psi_{k -}+\psi_k + \psi_{k +}=1_{C(\R)}$. 
Denote by $\cG[e^{-\pi}, (1-e^{-\pi})]$ the additive subgroup of $A(S)$ generated by $\{ge^{n\pi}(1-e^{-\pi})^m\ : \ g\in\cG,\ m, n\in\Z\}$. Consider the countable subgroup of $A(S)$
\[\cG_k^{\infty} =\cG[e^{-\pi}, (1-e^{-\pi})]\psi_{k -}\circ\pi + L(\cG^{\infty})\psi_k\circ\pi +\cG[e^{-\pi}, (1-e^{-\pi})]\psi_{k+}\circ\pi.\]
In the same way as in \cite[Lemma 2.2]{BEK2} and  \cite[Property 4.5]{ETh}, choose a countable subgroup $\cG_{00}$ of $A_{00}(S)$ satisfying the following conditions: 
\begin{itemize}
\item[(1)] for any $f\in A_{00}(S)$, $\varepsilon>0$, and $N\in\N$ with $\supp(f)\subset\pi^{-1}((-N, N)\setminus\{0\})$, there exists $g\in \cG_{00}$ such that $\sup_{s\in S} |f(x)-g(x)|<\varepsilon$ and $\supp(g) \subset \pi^{-1}((-N, N)\setminus\{ 0\})$,
\item[(2)] $\cG_k^{\infty} +\cG_{00}\subset \cG_{k+1}^{\infty} +\cG_{00}$ for any $k\in\N$, and
\item[(3)] $\sigma_S(\cG_{00}) =(\id_{A(S)}-\sigma_S)(\cG_{00})=\cG_{00}$.
\end{itemize}
Consider the countable subgroup of $A(S)$ 
\[ \cR(\cG)=\bigcup_{k=1}^{\infty} \cG_k^{\infty} +\cG_{00}.\]
Define positive cones $A(S)^+$ and $\cR(\cG)^+$ of $A(S)$ and $\cR(\cG)$ by
\begin{align*}
A(S)^+&=\{0\}\cup\{f\in A(S)\ :\ f(x)>0 \text{ for any } x\in S\},\\
\cR(\cG)^+&=\cR(\cG)\cap A(S)^+.
\end{align*}

\begin{lemma}\label{Lem3.4}
With $\cG$ and $\cR(\cG)$ as above, the following statements hold. 
\begin{itemize}
\item[\rm (i)] $(\cR(\cG), \cR(\cG)^+)$ is an unperforated ordered group.
\item[\rm (ii)] In the case $\cG=\Z$, $(\cR(\Z), \cR(\Z)^+)$ is a rational dimension group.
\end{itemize}
\end{lemma}
\begin{proof}
\noindent{\rm (i)}. It is straightforward to see that $\cR(\cG)^+\cap -\cR(\cG)^+=0$ and that the partially ordered abelian group $(\cR(\cG), \cR(\cG)^+)$ is unperforated. We must show that $\cR(\cG)=\cR(\cG)^+-\cR(\cG)^+$. For $k\in\N$,  let $g\in \cG_k^{\infty} +\cG_{00}\subset \cR(\cG)$ be given, so that
\[ g=h_-\psi_{k-}\circ\pi +h_0\psi_k\circ\pi + h_+\psi_{k+}\circ\pi +g_0,\]
with  $h_+$, $h_-\in \cG[e^{-\pi}, (1-e^{-\pi})]$, $h_0\in L(\cG^{\infty})$, and $g_0\in \cG_{00}$. For $N\in\N$, let us fix the notation $-[N, \infty):=(-\infty, -N]$ and $+[N, \infty):=[N, \infty)$. With  $M$ a large enough natural number, we have $h_{\pm}(x)\leq e^{\pm M\pi(x)}$ for any $x\in \pi^{-1}(\pm [M, \infty))$. Since the supports of $h_0\psi_k\circ\pi$ and $g_0$ are compact, and $\pi$ is proper, there exists $c>0$ such that the function $f=c(e^{-M\pi}\psi_{k-}\circ\pi +\psi_k\circ\pi + e^{M\pi}\psi_{k+}\circ\pi)\in\cR(\cG)^+$ satisfies $f-g\in \cR(\cG)^+$. Thus, $g=f- (f-g)\in \cR(\cG)^+-\cR(\cG)^+$. 

\noindent{\rm (ii)}. Let ${\frak n}$ be an infinite supernatural number and denote by $\tcR(\D_{\frak n})$ the  subgroup of $A(S)$ defined by
\[\tcR(\D_{\frak n})=\left\{ \sum_{i=1}^N d_i g_i\in A(S)\ : \ N\in \N,\ d_i\in \D_{\frak n},\ g_i\in\cR(\Z)\right\}.\]
We define a positive cone $\tcR(\D_{\frak n})^+$ by $\tcR(\D_{\frak n})^+=\tcR(\D_{\frak n})\cap A(S)^+$. First we show that the partially ordered abelian group $(\tcR(\D_{\frak n}), \tcR(\D_{\frak n})^+)$ is isomorphic to $(\cR(\Z)\otimes \D_{\frak n}, {\cR(\Z)\otimes \D_{\frak n}}^+)$ (as a partially ordered abelian group). Note that, by Lemma \ref{Lem3.1} (iv), 
 $(\cR(\Z)\otimes\D_{\frak n}, {\cR(\Z)\otimes\D_{\frak n}}^+)$ is an unperforated ordered abelian group.
Let $\Phi : \cR(\Z)\otimes\D_{\frak n} \rightarrow\tcR(\D_{\frak n})$ denote the group homomorphism determined by $\Phi(g\otimes d)= dg$ for $g\in\cR(\Z)$ and $d\in \D_{\frak n}$, which is obviously surjective. 
To show the injectivity of $\Phi$, let $g_i$, $g_i'\in \cR(\Z)$ and $d_i$, $d_i'\in \D_{\frak n}$, $i=1, 2, ..., N$, be such that $\sum_{i=1}^Nd_i g_i =\sum_{i=1}^N d_i' g_i'$ in $\tcR(\D_{\frak n})$. 
Choose $d\in\N$ such that $d d_i$, $dd_i'\in \Z$ for all $i=1,2,..., N$ and $1/d\in \D_{\frak n}$. Then it follows that $\sum_{i=1}^N dd_ig_i=\sum_{i=1}^N d d_i'g_i'$ in $\cR(\Z)$. Thus we have $d(\sum_{i=1}^Ng_i\otimes d_i)=d(\sum_{i=1}^N g_i'\otimes d_i')$ in $\cR(\Z)\otimes \D_{\frak n}$. Since $\cR(\Z)\otimes \D_{\frak n}$ is torsion free, we have $\sum_{i=1}^N g_i\otimes d_i =\sum_{i=1}^N g_i'\otimes d_i'$. It is trivial to see that $\Phi({\cR(\Z)\otimes \D_{\frak n}}^+)\subset \tcR(\D_{\frak n})^+$. Since $(\cR(\Z)\otimes\D_{\frak n}, {\cR(\Z)\otimes\D_{\frak n}}^+)$ is unperforated, we also have the converse inclusion.  

What remains to be shown is the RIP of $(\tcR(\D_{\frak n}), \tcR(\D_{\frak n})^+)$. The following argument is essentially same as that in the proof of \cite[Lemma 4.6]{ETh}.
 Let $f_i$, $g_j\in \tcR(\D_{\frak n})$, $i$, $j\in\{0, 1\}$, be such that $f_i(x) < g_j (x)$ for any $x\in S$ and $i, j\in\{0, 1\}$. At $\tau_S\in\pi^{-1}(0)$, since $\D_{\frak n}$ is dense in $\R$, we obtain $d\in\D_{\frak n}$ such that $f_i(\tau_S) < d <g_j(\tau_S)$ for all $i, j\in\{0, 1\}$. By the definition of $\tcR(\D_{\frak n})$, there exist $N\in\N$ and $\tf_{i+}$, $\tf_{i-}$, $\tg_{j +}$, $\tg_{j -}\in \D_{\frak n}[e^{-\pi}, (1-e^{-\pi})]$, $i, j\in\{0, 1\}$ such that $\tf_{i\pm}(x)=f_i(x)$ and $\tg_{j\pm}(x)=g_j(x)$ for any $i, j\in\{0, 1\}$ and $x\in \pi^{-1}(\pm [N, \infty))$. Note that $\pi(\pi^{-1}(\pm[N, \infty))$ are semi-bounded closed subsets of $\R$, since $\pi$ is a proper map. 
By Lemma \ref{Lem2.1}, we obtain $e_-$, $e_+\in\D_{\frak n}[e^x, e^{-x}]$ such that 
\[ f_i(x)<e_{\pm}(\pi(x))<g_j(x)\quad\text{ for } s\in\pi^{-1}(\pm[N, \infty)).\]
Set 
\[\bh(x)=(e_-\psi_{k-}+ d\psi + e_+\psi_{k+})\circ\pi(x)\quad\text{ for }x\in S;\] 
then it follows that $f_i(x)<\bh(x)< g_j(x)$ for any $i, j\in\{0, 1\}$ and $x\in \pi^{-1}((-\infty, -N]$ $\cup\{0\}\cup[N, \infty))$. By \cite[Lemma 2.3]{BEK2} and \cite[Lemma 4.4]{ETh}, there exists $\widetilde{h}\in A(S)$ such that $\widetilde{h}(x)=\bh(x)$ for $x\in\pi^{-1}((-\infty, -N]\cup\{0\}\cup [N, \infty))$, and $f_i < \widetilde{h} < g_j$ on $S$ for all $i, j\in\{0, 1\}$. Because $\pi^{-1}([-N, N])$ is compact, there exists $\varepsilon >0$ such that 
\[f_i(x) + \varepsilon <\widetilde{h}(x) < g_j(x)-\varepsilon,\]
for any $i, j\in\{0, 1\}$ and $x\in \pi^{-1}([-N, N])$. Set $\br(x)=\widetilde{h}(x)-\bh(x)$ for $x\in S$.  Then it follows that $\supp(\br)\subset\pi^{-1}((-N, N)\setminus\{0\})$. By {\rm (1)} of the conditions for $\cG_{00}$, we have $r\in\cG_{00}$ such that $r(x)\approx_{\varepsilon}\br(x)$ for $x\in S$ and $\supp(r)\subset \pi^{-1}((-N, N)\setminus \{0 \})$. The element $h= \bh +r\in\tcR(\D_{\frak n})$  satisfies $f_i < h < g_j$ for all $i, j\in\{0, 1\}$.
\end{proof}

 For an infinite supernatural number ${\frak n}$, consider the ordered abelian group  $(\tcR(\D_{\frak n}),$ $\tcR(\D_{\frak n})^+)$ introduced in the proof above. Define an ordered group $G_{\D_{\frak n}}$ and a subgroup $G_{\Z}\subset G_{\D_{\frak n}}$ by 
\begin{align*}
G_{\D_{\frak n}}&=\{\xi\oplus g\in(\D_{\frak n}^{\infty} \oplus \tcR(\D_{\frak n}))\ :\ \exists\varepsilon >0, \text{ for any } x\in \pi^{-1}((-\varepsilon, \varepsilon)),  L(\xi)(x)=g(x)\},\\
G_{\Z}&=(\Z^{\infty}\oplus \cR(\Z))\cap G_{\D_{\frak n}}.
\end{align*}
Define positive cones for $G_{\Z}$ and $G_{\D_{\frak n}}$ by 
\begin{align*}
G_{\D_{\frak n}}^+&=\{0\}\cup\{\xi\oplus g\in G_{\D_{\frak n}}\ :\ g\in\tcR(\D_{\frak n})^+\setminus\{0\}\},\\
G_{\Z}^+&=G_{\Z}\cap G_{\D_{\frak n}}^+.
\end{align*}
Note that $(G_{\Z}, G_{\Z}^+)$ is independent of the choice of infinite supernatural number ${\frak n}$. 
It is straightforward to check that $(G_{\D_{\frak n}}, G_{\D_{\frak n}}^+)$ and $(G_{\Z}, G_{\Z}^+)$ are torsion-free ordered abelian groups such that $(G_{\D_{\frak n}}, G_{\D_{\frak n}}^+)$ is isomorphic to $(G_{\Z}\otimes \D_{\frak n}, G_{\Z}\otimes \D_{\frak n}^+)$ (as ordered abelian groups). Since $\tcR(\D_{\frak n})$ is a dimension group (see proof of Lemma \ref{Lem3.4}), so also is $G_{\D_{\frak n}}$, for any infinite supernatural number ${\frak n}$, which means that $(G_{\Z}, G_{\Z}^+)$ is a rational dimension group. 

Define an automorphism $\sigma$ of $(G_{\D_{\frak n}}, G_{\D_{\frak n}}^+)$ by 
\[ \sigma(\xi\oplus g)=(\sigma_{\D_{\frak n}^{\infty}}(\xi)\oplus \sigma_S(g)), \quad\text{for }\xi\oplus g\in G_{\D_{\frak n}}.\]
From $\sigma_S(\cR(\Z)^+)=\cR(\Z)^+$, it follows that $\sigma(G_{\Z}^+)=G_{\Z}^+$. Then, using the same symbol, we may regard $\sigma$ as an automorphism of $(G_{\Z}, G_{\Z}^+)$.

Set $1_0=(\delta_{0, n})_{n\in\Z}\in \Z^{\infty}$, using the Kronecker delta $\delta_{m n}$ for $m, n\in\Z$, and set $u=(1_0\oplus 1)\in G_{\Z}$. Although $u$ is not necessarily an order unit, we 
may consider a kind of state space $S_u(H)$ for an ordered abelian subgroup $(H, H^+)$ of $(G_{\D_{\frak n}}, G_{\D_{\frak n}}^+)$ with $u\in H^+$ defined by 
\[S_u(H)=\{\varphi \ : \ H\rightarrow\R\ :\ \text{ a positive group homomorphism with }\varphi(u)=1\}.\]

Consider $S_u(H)$ with the relative topology from the Cartesian product $\R^H$ (the topology of pointwise convergence).
Assuming $\sigma(H^+)=H^+$, for $\alpha=e^{-\beta}\in\R$, consider the subset $S_{\sigma}^{\alpha}(H)$ of $S_u(H)$ and the bundle $S_{\sigma}(H)$ over $\R$ defined by 
\begin{align*}
S_{\sigma}^{\alpha}(H)&=\{\varphi\in S_u(H)\ : \ \varphi\circ\sigma=\alpha\varphi\},\\
S_{\sigma}(H)&=\{(\varphi, \beta)\in S_u(H)\times\R\ :\ \varphi\in S_{\sigma}^{\alpha}(H)\}.
\end{align*}
Although $\R^H$ is not locally compact, we shall show that $S_{\sigma}(G_{\Q})$ is locally compact,
 in Proposition \ref{Prop3.5} {\rm (i)} below.
Define a projection $\pi_{\sigma H} : S_{\sigma}(H)\rightarrow \R$ by $\pi_{\sigma H}(\varphi, \beta)=\beta$. 
To simplify notation, let us write $\pi_{\Z}$ and $\pi_{\D_{\frak n}}$ for $\pi_{\sigma G_{\Z}}$ and $\pi_{\sigma G_{\D_{\frak n}}}$.

For a supernatural number ${\frak n}$, let $\iota_{\frak n}$ denote  the canonical embedding of $G_{\Z}$ into $G_{\D_{\frak n}}$, and  $\iota_{\frak n}^*$ the induced map from $S_u(G_{\D_{\frak n}})$ to $S_u(G_{\Z})$ defined by $\iota_{\frak n}^*(\varphi)=\varphi\circ\iota_{\frak n}$ for $\varphi\in S_u(G_{\D_{\frak n}})$. Define a continuous map $\iota_{{\frak n} \sigma}^*$ from  $S_{\sigma}(G_{\D_{\frak n}})$ to $S_{\sigma}(G_{\Z})$ by $\iota_{{\frak n} \sigma}^*(\varphi, \beta)=(\iota_{\frak n}^*(\varphi), \beta)$ for $(\varphi, \beta)\in S_{\sigma}(G_{\D_{\frak n}})$. For $s\in S$ and $g\in G_{\D_{\frak n}}$ with $g=\xi\oplus f$ for some $\xi\in\D_{\frak n}^{\infty}$ and $f\in \tcR(\D_{\frak n})$,  set $\widehat{s}(g)=f(s)$ and regard $\widehat{s}$ as an element of $S_{\sigma}^{\alpha}(G_{\D_{\frak n}})$ for $\alpha=e^{-\pi(s)}$. Define a continuous map $\Psi_{\D_{\frak n}} $ from $S$ to $S_{\sigma}(G_{\D_{\frak n}})$  by $\Psi_{\D_{\frak n}}(s)=(\widehat{s}, \pi(s))$ for $s\in S$. 

\begin{proposition}\label{Prop3.5}
As in Remark \ref{Rem3.3}, let ${\frak u}$ denote the supernatural number corresponding to the universal UHF algebra, i.e., $\D_{\frak u}=\Q$.
\begin{itemize}
\item[\rm (i)] The pair $(S_{\sigma}(G_{\Q}), \pi_{ \Q})$ is a proper simplex bundle.
\item[\rm (ii)] The induced map $\iota_{{\frak u}\sigma}^* : S_{\sigma}(G_{\Q})\rightarrow S_{\sigma}(G_{\Z})$ is a homeomorphism such that $\pi_{ \Z}\circ \iota_{{\frak u} \sigma}^* =\pi_{\Q}$ and the restriction $\iota_{{\frak u} \sigma}^*|_{\pi_{ \Q}^{-1}(\beta)}$ is affine for each $\beta\in\R$. Therefore the pair $(S_{\sigma}(G_{\Z}), \pi_{\Z})$ is also a proper simplex bundle that is isomorphic to $(S_{\sigma}(G_{\Q}), \pi_{\Q})$.
\item[\rm (iii)] $\Psi_{\Z} : S\rightarrow S_{\sigma}(G_{\Z})$ is an isomorphism of simplex bundles.
\end{itemize}
\end{proposition}
To prove the proposition, we shall consider the subgroup of $G_{\Q}$ for which $u$ is an order unit, and its elements of compact support. Let $\rho : G_{\Q}\rightarrow\tcR(\Q)$ denote the standard projection defined by $\rho(\xi\oplus f)=f$. Set $G_{\Q}^{u+}=\{g\in G_{\Q}^+ : \text{ there exists } N\in\N \text{ such that } x\leq Nu \}$, $G_{\Q}^{c+}=\{g\in G_{\Q}^+ : \supp(\rho(g)) \text{ is compact} \}$, $G_{\Q}^u=G_{\Q}^{u+}-G_{\Q}^{u+}$, and $G_{\Q}^{c}=G_{\Q}^{c+}-G_{\Q}^{c+}$.  It is straightforward to check that $(G_{\Q}^u, G_{\Q}^{u+})$ is an ordered abelian group with order unit $u$.  Since $(G_{\Q}, G_{\Q}^+)$ is a dimension group, we also see that $(G_{\Q}^u, G_{\Q}^{u+})$ has the RIP. Let $S(G_{\Q}^u)$ denote the state space of $G_{\Q}^u$,  a Choquet simplex (see for example \cite[Theorem 10.17]{Goodearl}). 

To simplify notation, set $Q_0=\Q[e^x, e^{-x}]$ and $Q=\Q[e^{-x}, (1-e^{-x})]$. For $\xi=(\xi_n)_{n\in\Z}\in \Q^{\infty}$,  define $\Sigma \xi\in Q_0$ by $\Sigma \xi(x)=\sum_{n\in\Z} \xi_n e^{nx}$. Let $\varphi\in S(G_{\Q}^u)$. Note that, since $\varphi(Ng)=N\varphi(g)$ for any $g\in G_{\Q}^u$ and $N\in \N$,  we have $\varphi(\lambda g)=\lambda\varphi(g)$ for $\lambda\in\Q$ and $g\in G_{\Q}^u$. 

\begin{lemma}\label{Lem3.6}
Suppose that $\varphi$, $\psi\in S(G_{\Q}^u)$, $\lambda>0$, and $\beta\in \R$ satisfy $\psi(g)\leq \lambda\varphi(g)$ for any $g\in G_{\Q}^{u+}$, and $\varphi(\sigma(g))=e^{-\beta}\varphi(g)$ for any $g\in G_{\Q}^c$. Then it follows that $\psi(\sigma(g))=e^{-\beta}\psi(g)$ for any $g\in G_{\Q}^c$.
\end{lemma}

\begin{proof}
For $\eta=(\eta_n)_{n\in\Z}$ and $\xi=(\xi_n)_{n\in\Z}\in \Q^{\infty}$,  let $\eta *\xi$ denote the convolution product of $\eta$ and $\xi$ (i.e., $\eta*\xi =(\sum_{k\in\Z} \eta_k\xi_{n-k})_n$).  For $q\in Q_0$ and $g=(\xi\oplus f)\in G_{\Q}^c$, setting $q^{\infty}=(q_n)_n\in \Q^{\infty}$ such that 
$\Sigma q^{\infty} =q$, we define $q\cdot g\in G_{\Q}^c$ by $q\cdot g=(q^{\infty}*\xi\oplus q\circ\pi f)$. Note that, since $\varphi(e^{-x}\cdot g)=e^{-\beta}\varphi(g)$ for any $g\in G_{\Q}^c$, it follows that $\varphi(q\cdot g)=q(\beta)\varphi(g)$ for any $q\in Q_0$ and $g\in G_{\Q}^c$. 

Fix $g\in G_{\Q}^{c+}$ and define $\psi_g(q)=\psi(q\cdot g)$ for $q\in Q_0$. Thus one obtains the Cauchy-Schwartz inequality 
\[\psi_g(qr)^2\leq \psi_g(q^2)\psi_g(r^2)\]
for any $q, r\in Q_0$. Choose a sequence $\alpha_n\in\Q$, $n\in\N$,  such that $\lim\limits_{n\to\infty}\alpha_n=e^{-\beta}$, and set $\tq_n=e^{-x}-\alpha_n\in Q_0$. Then we have  
\begin{align*}
\psi_g(\tq_n)^2&\leq\psi_g(\tq_n^2)\psi_g(1)\leq\lambda^{-1}\varphi(\tq_n^2\cdot g)\psi_g(1) \\
&=\lambda^{-1}\tq_n^2(\beta)\varphi(g)\psi(g)\rightarrow 0\quad (n\to \infty),
\end{align*}
which implies that $\lim_{n\to\infty}\psi(e^{-x}\cdot g) -\alpha_n\psi(g)=0$. Since $G_{\Q}^c=G_{\Q}^{c+}-G_{\Q}^{c+}$, we have $\psi(e^{-x}\cdot g)=e^{-\beta}\psi(g)$ for any $g\in G_{\Q}^c$.
\end{proof}

\begin{proof}[Proof of Proposition \ref{Prop3.5}]
{\rm (i)}. First we show that $\pi_{ \Q} : S_{\sigma}(G_{\Q})\rightarrow \R$ is a proper map. Let $K\subset \R$ be a compact set, and $s_{\lambda}=(\varphi_{\lambda}, \beta_{\lambda})\in \pi_{\Q}^{-1}(K)$, $\lambda\in \Lambda$ be a universal net (see \cite{Ped2} for the definition).  For $g\in G_{\Q}$ such that $\rho(g)$ is bounded on $S$, there exists $c\in \N$ such that $|\varphi_{\lambda}(g)|<c$ for any $\lambda \in \Lambda$. Then the net $\varphi_{\lambda}(g)$, $\lambda\in\Lambda$, converges to a point  in $[-c, c]$. For a general $g\in G_{\Q}$, there are $g_0$, $g_+$, $g_-\in G_{\Q}$ such that $g=g_-+ g_0 + g_+$, and $\rho(g_0)$, $\rho(\sigma^N(g_-))$, and $\rho(\sigma^{-N}(g_+))$ are bounded on $S$. By $(\varphi_{\lambda}, \beta_{\lambda})\in \pi_{\Q}^{-1}(K)$ and $\varphi_{\lambda}(g_{\pm})=e^{\pm N\beta_{\lambda}}\varphi_{\lambda}(\sigma^{\mp N}(g_{\pm}))$, it follows that $\varphi_{\lambda}(g)$, $\lambda\in\Lambda$, converges in $\R$. Set $\displaystyle \varphi(g)=\lim_{\lambda \to \infty}\varphi_{\lambda}(g)$.  Setting $\displaystyle \beta=\lim_{\lambda\to \infty}\beta_{\lambda}\in K$, we have $\varphi(\sigma(g))=e^{-\beta}\varphi(g)$ for any $g\in G_{\Q}$, which implies that $s_{\lambda}$ converges to $(\varphi, \beta)\in \pi_{\Q}^{-1}(K)$. Thus, $\pi_{\Q}^{-1}(K)$ is compact.

For $s=(\varphi, \beta)\in S_{\sigma}(G_{\Q})$, we see that $\pi_{\Q}^{-1}([\beta- \varepsilon, \beta+\varepsilon])$ for $\varepsilon>0$ is a compact neighbourhood of $S$, which means that $S_{\sigma}(G_{\Q})$ is locally compact. 

It remains to show that $\pi_{\Q}^{-1}(\beta)$ is a Choquet simplex for any $\beta\in\R$. Fix $\beta\in\R$.  Denote by $\iota_u$ the canonical embedding of $G_{\Q}^u$ into $G_{\Q}$ and define the induced map $\iota_u^* : \pi_{\Q}^{-1}(\beta)\rightarrow S(G_{\Q}^u)$ by $\iota_u^*(s)=\varphi\circ\iota_u$ for $s=(\varphi, \beta)\in \pi_{\Q}^{-1}(\beta)$. From the definition, it follows that $\iota_u^*$ is affine and continuous in the topology of pointwise convergence. We see that points of $\pi_{\Q}^{-1}(\beta)$ are separated by $G_{\Q}^u$, because any $g\in G_{\Q}$ can be decomposed as $g=\sigma^{-N}(g_-)+ g_0 +\sigma^N(g_+)$ for some $g_0$, $g_-$, $g_+\in G_{\Q}^u$ and $N\in\N$. Thus $\iota_u^*$ is injective. Since a compact face of a Choquet simplex is also a Choquet simplex (see \cite[Theorem 10.9]{Goodearl}), it suffices to show that $\Image(\iota_u^*)$ is a face of $S(G_{\Q}^u)$. 

Let $\varphi\in \Image(\iota_u^*)$, $\psi$, $\eta\in S(G_{\Q}^u)$, and $\lambda\in (0, 1)$ satisfy $\varphi=\lambda\psi + (1-\lambda)\eta$.  In the following argument, we first consider a specific $g\in G_{\Q}^u$ with the property $\rho(g)=\psi_{k+}\circ \pi$, $\psi_{k-}\circ\pi$, $(\psi_{k+}-\psi_{l+})\circ\pi$, or $(\psi_{k-}-\psi_{l-})\circ\pi$ for $k$, $l\in\N$ with $k> l$. Set $f_g=\rho(g)$, $Q_g=\{q\in Q\ : \ 0\oplus (q\circ\pi) f_g\in G_{\Q}^u\}$, and $\psi_g(q)=\psi(0 \oplus (q\circ\pi) f_g)$ for $q\in Q_g$. Thus we  obtain the  Cauchy-Schwartz inequality, 
$\psi_g(qr)^2\leq\psi_g(q^2)\psi_g(r^2)$,
for any $q$, $r\in Q_g$ satisfying $qr$, $q^2$, $r^2\in Q_g$.  By an argument  similar to that in the proof of Lemma \ref{Lem3.6},
% Note that, since $\supp(f_g)$ is semi-bounded, for any $q\in Q_g$ there exists a sequence $q_n\in \Q[e^x, e^{-x}]$, $n\in\N$ which uniformly converges to $q$ on $\pi(\supp(f_g))$. Then it follows that 
%\[\varphi(0\oplus (q\circ\pi) f_g)=\lim_{n\to\infty}\varphi(0\oplus (q_n\circ\pi) f_g)=\lim_{n\to\infty}q_n(\beta)\varphi(0\oplus f_g) =q(\beta)\varphi(0\oplus f_g).\]
%For $q\in Q_g$, we set a sequence $\alpha_n\in \Q$, $n\in\N$ which converges to $q(\beta)$ and set $\tq_n=q_n-\alpha_n\in \Q[e^x, e^{-x}]$. If $q\in Q_g$, then it follows that $\tq_n$, $\tq_n^2\in Q_g$ and 
%\begin{align*}
%\psi_g(\tq_n)^2&\leq\psi_g(\tq_n^2)\psi_g(1)\leq\lambda^{-1}\varphi(0\oplus (\tq_n^2\circ\pi) f_g)\psi_g(1) \\
%&=\lambda^{-1}\tq_n^2(\beta)\varphi(0\oplus f_g)\psi_g(1)\rightarrow 0,\quad (n\to \infty).
%\end{align*}
we have 
\begin{align}
\psi(0\oplus (q\circ\pi) f_g)=q(\beta)\psi(0\oplus f_g)\quad  \text{ for any } q\in Q_g.
\end{align} For general $g=\xi\oplus f\in G_{\Q}$ represented 
as $f=(f_-\psi_{k -})\circ\psi + L(\xi)\psi_{k}\circ \pi + (f_+\psi_{k +})\circ\pi + f_0$ with $f_+$, $f_-\in Q$, and $f_0\in \cG_{00}$,  define 
\[\tpsi(g)=f_-(\beta)\psi(0\oplus \psi_{k -}\circ\pi)+ \psi(\xi\oplus L(\xi)\psi_{k}\circ\pi)+ f_+(\beta)\psi(0\oplus\psi_{k+}\circ\pi) +\psi(0\oplus f_0).\]
Then $\tpsi : G_\Q\rightarrow\R$ is well defined (single-valued). Indeed, if two elements $g_1$, $g_2\in G_{\Q}$ are written as 
\[ g_i= \xi^{(i)}\oplus\left((f_-^{(i)}\psi_{k_i-})\circ\pi + L(\xi^{(i)})(\psi_{k_i}\circ\pi)+(f_+^{(i)}\psi_{k_i +})\circ\pi + f_0^{(i)}\right),  \quad i=1, 2,\]
for some $f_-^{(i)}$, $f_+^{(i)}\in Q$, and $f_0^{(i)}\in G_{00}$, and if $g_1=g_2$, then we may assume that $k_1\leq k_2$ and $f_{\pm}^{(1)}=f_{\pm}^{(2)}$ without loss of generality. Applying (1)  to $f_g=(\psi_{k_2 \pm} -\psi_{k_1 \pm})\circ\pi$ we have 
\[\psi(0\oplus f_{\pm}(\psi_{k_2\pm}-\psi_{k_1\pm})\circ\pi)=f_{\pm}(\beta)\psi(0\oplus(\psi_{k_2\pm}-\psi_{k_1\pm})\circ\pi),\] which implies  $\tpsi(g_2)-\tpsi(g_1)=\psi(g_2-g_1) =0$. Similarly, applying (1) to $f_g=\psi_{k\pm}\circ\pi$ we see that $\tpsi(g)=\psi(g)$ for any $g\in G_{\Q}^u$.  It is not hard to see that $\tpsi : G_{\Q}\rightarrow \R$ is a group homomorphism such that $\tpsi(u)=1$. In the following paragraph we shall show that $\tpsi$ is positive. 

Let $g=\xi\oplus f\in G_{\Q}$ be such that $f>0$ and let $k\in\N$, $f_+$, $f_-\in Q$, and $f_0\in\cG_{00}$ satisfy $f=(f_-\psi_{k-})\circ\pi + L(\xi)(\psi_k\circ\pi) +(f_+\psi_{k +}\circ\pi) + f_0$. Let $N\in\N$ and continuous functions $e_+$, $e_-$ on $\R$ satisfy $\supp(f_0)\subset [-N, N]$, $0\leq e_{\pm}\leq 1$, $\supp(e_{\pm})\subset\pm[N, \infty)$, and $e_{\pm}|_{\pm [N+1, \infty)}=1$. Set 
\[ z=((1-e_-)f_-\psi_{k -})\circ\pi+ L(\xi)\psi_{k}\circ\pi +((1-e_+)f_+\psi_{k+})\circ\pi + f_0\in A(S).\]
Since $f>0$, we may assume that $N\in\N$, $e_+$, and $e_-$ also satisfy $f_{\pm}e_{\pm}\geq 0$, and that $z\geq 0$. Since there exists (a large) $M\in\N$ such that $e^{\mp Mx}f_{\pm}$ are bounded on $\pm [0, \infty)$, by approximating $e^{\pm Mx}(1-e_{\pm})$ on $\pm[0, \infty)$ uniformly, we obtain sequences $d_{+ n}$, $d_{- n}\in Q$ such that 
\[\sup_{x\in \pm [0, \infty)}|d_{\pm n}f_{\pm}(x)-(1-e_{\pm})f_{\pm}(x)|< 1/n.\]
Set 
\begin{align*}
y_n&= (f_-(1-d_{-n})\psi_{k-})\circ\pi +(f_+(1-d_{+n})\psi_{k \pm})\circ\pi\in\tcR(\Q),\\
z_n&=(f_-d_{- n}\psi_{k -})\circ\pi +L(\xi)\psi_{k}\circ\pi + (f_+d_{+ n}\psi_{k+})\circ\pi + f_0\in\tcR(\Q).
\end{align*}
Then it follows that $0\oplus y_n$, $\xi \oplus z_n\in G_{\Q}$ and $g=(0\oplus y_n)+(\xi\oplus z_n)$ for any $n\in\N$. Since $f_{\pm}(1-d_{\pm})$ converge uniformly to $e_{\pm}f_{\pm}$ on $\pm[0, \infty)$, we have that $\displaystyle \liminf_{n\to\infty}\tpsi(0\oplus y_n)$
\[=\liminf_{n\to \infty}f_-(1-d_{-n})(\beta)\psi(0\oplus \psi_{k -}\circ\pi)+f_+(1-d_{+ n})(\beta)\psi(0\oplus \psi_{k +}\circ\pi)\geq 0.\]
Since $z_n$ converges uniformly to $z\geq 0$ on $S$, and $\supp(z)$ is compact, we have  $\xi\oplus z_n\in G_{\Q}^u$ and $\displaystyle \liminf_{n\to\infty}\tpsi(\xi\oplus z_n)=\liminf_{n\to\infty} \psi(\xi\oplus z_n)\geq 0$. Thus we conclude that $\tpsi(g)\geq 0$.

 By Lemma \ref{Lem3.6} and $\psi(g)\leq\lambda^{-1}\varphi(g)$ for $g\in G_{\Q}^{u+}$, we see that $\tpsi(\sigma(g))=e^{-\beta}\tpsi(g)$ for any $g\in G_{\Q}$. 
 From $\tpsi(g)=\psi(g)$ for any $g\in G_{\Q}^u$, it follows that $\psi=\iota_u^*(\tpsi)\in \Image(\iota_u^*)$. By the same argument, we also see that $\eta\in \Image(\iota_u^*)$, which means that $\Image(\iota_u^*)$ is a face of $S(G_{\Q}^u)$, as asserted, and the proof of (i) is complete. 

\rm{(ii)}. For any $g\in G_{\Q}$, consider the smallest number $N_g\in \N$ such that $N_gg\in G_{\Z}$. To show injectivity of $\iota_{{\frak u} \sigma}^*$,  let $s=(\varphi, \pi_{\Q}(\varphi))$, $t=(\psi, \pi_{\Q}(\psi))\in S_{\sigma}(G_{\Q})$ satisfy $\iota_{{\frak u} \sigma}^*(s)=\iota_{{\frak u} \sigma}^*(t)$. Then, 
\[N_g\varphi(g)=\varphi\circ\iota_{\frak u}(N_g g) =\psi\circ\iota_{\frak u}(N_gg) = N_g\psi(g),\]
for any $g\in G_{\Q}$, and  it follows that $s=t$. To prove surjectivity of $\iota_{{\frak u} \sigma}^*$,  let $s=(\varphi, \beta)\in S_{\sigma}(G_{\Z})$ and define a map $\tphi :G_{\Q}\rightarrow\R$ by $\tphi(g)=\frac{1}{N_g}\varphi(N_gg)$ for $g\in G_{\Q}$. From
$\tphi(g +h)=\frac{1}{N_{g+h}N_gN_h}\varphi(N_{g+h}N_gN_h(g+h))=\tphi(g)+\tphi(h)$ for any $g$, $h\in G_{\Q}$, it follows that $\tphi$ is a group homomorphism. It is straightforward to see that $\tphi$ is a positive group homomorphism satisfying $\tphi(u)=1$. By $\tphi(\sigma(g))=\frac{1}{N_gN_{\sigma(g)}}\varphi(N_gN_{\sigma(g)}\sigma(g))=\frac{1}{N_gN_{\sigma(g)}}e^{-\beta}\varphi(N_gN_{\sigma(g)}g)= e^{-\beta}\tphi(g)$ for any $g\in G_{\Q}$, we have $(\tphi, \beta)\in S_{\sigma}(G_{\Q})$ and $\iota_{{\frak u}\sigma}^*((\tphi, \beta))=(\varphi, \beta)$. 

By the same argument as in {\rm (i)}, we see that $S_{\sigma}(G_{\Z})$ is also locally compact. From the definition of $\iota_{\frak u}^*$, it follows that $\iota_{{\frak u}\sigma}^*$ is a bijective continuous map such that $\pi_{\Z}\circ\iota_{{\frak u}\sigma}^*=\pi_{\Q}$ and $\iota_{{\frak u}\sigma}^*|_{\pi_{\Q}^{-1}(\beta)}$ is affine for any $\beta\in\R$. Because of Lemma \ref{Lem3.3}, $\iota_{{\frak u} \sigma}^*$ is an isomorphism of simplex bundles and $(S_{\sigma}(G_{\Z}), \pi_{\Z})$ is also proper. 

{\rm (iii)}.
As $\Psi_{\Z}=\iota_{{\frak u} \sigma}^*\circ\Psi_{\Q}$, it suffices to show that $\Psi_{\Q} : S\rightarrow S_{\sigma}(G_{\Q})$ is an isomorphism of simplex bundles. From the definition, it is straightforward to check that $\Psi_{\Q}$ is injective and continuous, and the restriction of $\Psi_{\Q}$ to $\pi^{-1}(\beta)$ is affine for any $\beta\in\R$. By Lemma \ref{Lem3.3}, it is enough to show that $\Psi_{\Q}$ is surjective. However, this is done in the proof of \cite[Lemma 4.12]{ETh}. We sketch a proof for the convenience of the reader. 

Let $(\varphi, \beta)$ be in $S_{\sigma}(G_{\Q})$. If $g=\xi\oplus f\in G_{\Q}$ satisfies $f=0\in \tcR(\Q)$, then it follows that $ng\leq u$ for any $n\in \N$, which implies $\varphi(g)=0$. Thus we obtain a positive group homomorphism $\varphi' : \tcR(\Q)\rightarrow \R$ such that $\varphi'\circ \rho =\varphi$. Denote by $A_{\R}(S)$ the real Banach space, with the supremum norm, consisting of functions in $A(S)$ which have a limit at infinity. Since $\rho(G_{\Q}^u)\cap A_{\R}(S)$ is uniformly dense in $A_{\R}(S)$ (see \cite[Lemma 4.11]{ETh}), we can extend $\varphi'$ to $A_{\R}(S)$ in such a way that $ |\varphi'(f)|\leq \sup_{s\in S}|f(s)|$ for any $f\in A_{\R}(S)$. By the Hahn-Banach theorem, we can further extend $\varphi'$ to the Banach space $ C^u(S)=\{f\in C(S, \R) : \text{ there exists }\lim_{s\to \pm\infty} f(s)\}$. Since $\varphi'(1)=1$, the extension of $\varphi'$ is also a positive bounded linear functional on $C^u(S)$. Thus we obtain a Borel measure $m_{\varphi}$ on $S$ such that 
\[\varphi'(f) =\int_S f(x) \ dm_{\varphi}(x)\quad\text{for any }f\in C_0(S).\]

For $y\in \Aff(\pi^{-1}(\beta))$ such that $y>0$, by \cite[Lemma 4.4 (2) and Lemma 4.11]{ETh}, we obtain a sequence $g_{y n}=(\xi_{y n}\oplus f_{y n})\in G_{\Q}$, $n\in\N$, such that $\supp(f_{y n})$ is compact, $f_{y n}\geq 0$ for any $n\in\N$, $\sup_{n\in\N}\sup_{s\in S}|f_{y n}(s)|<\infty$, and $f_{y n}|_{\pi^{-1}(\beta)}$ converges to $y$ uniformly. Let $\omega$ be a free ultrafilter on $\N$. For $y\in \Aff(\pi^{-1}(\beta))$ with $y>0$, we define $\displaystyle \barphi(y)=\lim_{n\to\omega}\varphi'(f_{yn})$. Then $\barphi$ is independent of the choice of $\xi_{yn}$ and $f_{yn}$. Indeed, if we suppose that $g_{yn}'=(\xi_{yn}'\oplus f_{yn}')\in G_{\Q}$ is another choice, then for $\varepsilon >0$ there exists $N\in\N$ such that $\sup_{s\in \pi^{-1}(\beta)}|f_{yn}(s) - f_{yn}'(s)|< \varepsilon$, for any $n>N$. Set $X_n=\supp(f_{yn})\cup \supp(f_{yn}')$, which is compact for any $n\in \N$, and let $q_m\in \Q[e^x, e^{-x}]$, $m\in\N$, be a sequence such that $q_m |_{X_n}$ converges to the characteristic function $\chi_{\beta}$
 at $\{\beta\}$ and such that $\supp_{s\in X_n}|q_m(s)|$, $m\in\N$, is a bounded sequence. 
Then, by the Lebesgue dominated convergence theorem, 
\begin{align*}
\lim_{m\to\infty} \varphi'((q_m\circ\pi) f_{yn}) &=\int_S(\chi_{\beta}\circ\pi )f_{yn}\ d m_{\varphi}\approx_{\varepsilon} \int_S (\chi_{\beta}\circ\pi) f_{yn}'\ dm_{\varphi} =\lim_{m\to\infty}\varphi'((q_m\circ\pi)f_{yn}'), 
\end{align*}
 for any $n\geq N$.
It follows that 
\[|\varphi'(f_{y n}- f_{y n}')|\leq \limsup_{m\to \infty}\left(\left|\varphi'(q_m\circ\pi(f_{yn}-f_{yn}'))\right| +\left|(1-q_m(\beta))\varphi'(f_{yn}- f_{yn}')\right|\right)<\varepsilon. \]
Since $\varepsilon >0$ is arbitrary, we have $\lim_{n\to\omega} \varphi'(f_{yn})=\lim_{n\to\omega}\varphi'(f_{yn}')$. 

Since $\barphi$ is well defined, we see that $\barphi(x+y)=\barphi(x)+\barphi(y)$ for any $x, y\in\Aff(\pi^{-1}(\beta))$ with $x, y >0$. Because $\Aff(\pi^{-1}(\beta))$ is an ordered abelian group in the strict ordering, we can extend $\barphi$ to a positive group homomorphism from $\Aff(\pi^{-1}(\beta))$ to $\R$. Thus there exists $s\in \pi^{-1}(\beta)$ and $\lambda\in (0, \infty)$ such that $\barphi(y)=\lambda y(s)$ for any $y\in\Aff(\pi^{-1}(\beta))$; see \cite[Theorem 7.1]{Goodearl}, for example. From the definition of $\barphi$ it follows that 
\[\lambda f(s) = \barphi(f|_{\pi^{-1}(\beta)})=\varphi(\xi\oplus f),\]
for any $g=\xi\oplus f\in G_{\Q}$ such that $\supp(f)$ is compact. For $g=\xi\oplus f\in G_{\Q}$ with $f\in A_0(S)$ and $\xi=0$, by the property (1) of $\cG_{00}$, there exists a sequence $f_n\in\cG_{00}$, $n\in\N$, which converges to $f$ uniformly. Then it follows that $\varphi(0\oplus f)=\lim_{n\to\infty}\varphi(0\oplus f_n)=\lambda f(s)$. For $f\in Q$, there exists a large $N\in\N$ such that $(fe^{\mp Nx}\psi_{k\pm})\circ \pi\in A_0(S)$ for any $k\in\N$. Then we have that $\varphi(0\oplus (f\psi_{k\pm})\circ\pi)=e^{\pm N\beta}\varphi(0\oplus (e^{\mp Nx}f\psi_{k\pm})\circ\pi)=\lambda f\psi_{k\pm}(\beta)$, which implies that $\varphi(g)=\lambda f(s)$ for any $g=\xi\oplus f\in G_{\Q}$. Since $\varphi(u)=1$, it is clear that $\lambda=1$, and thus we conclude that $\Psi_{\Q}(s)=(\varphi, \beta)$.
\end{proof}

\section{Rationally AF algebras}\label{Sec4}

In order to realize a given rational dimension group at the level of $K_0$-groups for operator algebras, it seems plausible that a corresponding \Cs{} should become an AF-algebra after tensoring 
with a UHF-algebra. This tensoring procedure is called rationalization in \cite{Bl3}, \cite{Win3}. While the rationalization was determined by a single UHF-algebra in these previous works, we would like to use the term ``rationally'' to refer to the tensor products with two UHF-algebras, that are relatively prime in the natural sense. (cf. Remark \ref{Rem3.3}.)

We prepare some facts and notation concerning \Cs s.
For a \Cs\ $C$, we let $C^{\sim}$ denote the unitization of $C$, and $C^1$ denote the closed unit ball of $C$. We denote by $\id_C$ the identity automorphism of $C$.  For a subset $F\subset C$, let $P(F)$ denote the set of all projections in $F$ and $F^+$ the set of all positive elements in $F$. When $C$ is a unital \Cs, the symbol $1_C$ means the unit of $C$. For two elements $x$, $y\in C$, and $\varepsilon >0$,  we use the notation $x\approx_{\varepsilon} y$ if $\|x -y \|< \varepsilon$. For subsets $F$, $G$ in $C$, and $\varepsilon >0$, the notation $F \subset_{\varepsilon} G$ means  that  for any $x\in F$ there exists $y\in G$ with $y\approx_{\varepsilon} x$. 

For a supernatural number ${\frak n}$, let $M_{\frak n}$ denote the uniformly hyperfinite (UHF) algebra of type ${\frak n}$ (i.e., $(K_0(M_{\frak n}), K_0(M_{\frak n})^+, [1_{M_{\frak n}}]_0)\cong (\D_{\frak n}, \D_{\frak n}^+, 1_{\D_{\frak n}})$), and $\cK$ the \Cs{} of all compact operators on a separable Hilbert space. For a natural number $n\in\N$, we let $\{e_{i j}^{(n)}\}_{i, j=1}^n$ or $\{e_{i j}\}_{i, j=1}^n$ denote the set of canonical matrix units of $M_n$. 

For two \Cs s $C$, $D$, and a $*$-homomorphism $\varphi$ from $C$ to $D$, denote the induced map from $K_0(C)$ to $K_0(D)$, a homomorphism of pre-ordered abelian groups (the positive cone of $K_0$ being the canonical image of the Murray-von Neumann semigroup---see \cite{RLL}), by $\varphi_*$.

\begin{definition}\label{Def4.1}
Given a \Cs{} $A$, we shall say that $A$ is \emph{rationally AF} (RAF) if $A\otimes M_{\frak p}$ and $A\otimes M_{\frak q}$ are AF-algebras for some relatively prime  pair of supernatural numbers ${\frak p}$ and ${\frak q}$. 
\end{definition}

\begin{lemma}\label{Lem4.2}
{ }\
\begin{itemize}
\item[\rm (i)] If $A$ is a RAF-algebra, then $A\otimes M_{\frak n}$ is approximately finite dimensional for every infinite supernatural number ${\frak n}$. 
\item[\rm (ii)] Any RAF-algebra has  trivial $K_1$-group.
\end{itemize}
\end{lemma}
\begin{proof}
\noindent {\rm (i)}. Since $A\otimes M_{ n}$ is also a RAF-algebra for any $n\in\N$, to show $A\otimes M_{\frak n}$ is approximately finite dimensional, it suffices to show that for any finite subset $F$ of $A$ and $\varepsilon>0$ there exists a finite dimensional $\mathrm{C}^*$-subalgebra $B$ of $A\otimes M_{\frak n}$ such that $F\otimes 1_{M_{\frak n}}\subsetepsilon B$.

Taking large natural numbers $m_{\frak p}$ and $m_{\frak q}$ with $m_{\frak p} |{\frak p}$ and $m_{\frak q}|{\frak q}$, we have finite dimensional $\mathrm{C}^*$-subalgebras $A_{\frak i}$ of $A\otimes M_{\frak i}$ for ${\frak i}={\frak p}, {\frak q}$ such that $F\otimes 1_{M_{\frak i}}\subsetepsilon A_{\frak i}$ for both $\frak i={\frak p}, {\frak q}$. Since $m_{\frak p}$ and $m_{\frak q}$ are relatively prime, there exist natural numbers $a$, $b$, and $N$ such that $am_{\frak p} + bm_{\frak q}=N$ and $N|{\frak n}$ (cf. proof of Lemma \ref{Lem3.1}).  Therefore there exists a unital embedding $\Phi$ of $M_{m_{\frak p}}\oplus M_{m_{\frak q}}$ into $M_{\frak n}$. Identifying $A\otimes (M_{m_{\frak p}}\oplus M_{m_{\frak q}})$ with $(A\otimes M_{m_{\frak p}})\oplus (A\otimes M_{m_{\frak q}})$, we obtain a finite dimensional $\mathrm{C}^*$-subalgebra $B =(\id_A\otimes\Phi)(A_{\frak p}\oplus A_{\frak q})$ of $A\otimes M_{\frak n}$ which satisfies the above condition. 

\noindent{\rm (ii)}. Because of the K\"{u}nneth theorem \cite[Theorem 2.14]{Sch} (see also \cite[Theorem 23.1.3]{Bl2}), we see that $K_1(A)\otimes K_0(M_{\frak i})\cong K_1(A\otimes M_{\frak i})=0$ for both ${\frak i}={\frak p}, {\frak q}$. From {\rm (i)} of Lemma \ref{Lem3.1}, it follows that $K_1(A)$ is torsion free. Therefore the canonical embedding of $K_1(A)$ into $K_1(A)\otimes \D_{\frak p}=0$ is injective, and so $K_1(A)=0$.
\end{proof}

Note that any RAF-algebra is AF-embeddable, and so it is stably finite. Hence the ordered $K_0$-group of any RAF-algebra is a partially ordered abelian group. In combination with Lemma \ref{Lem3.1} {\rm (iii)}, the following results show that the ordered $K_0$-group of a $\mathcal{Z}$-absorbing RAF-algebra is an ordered abelian group. 

A typical example of a RAF-algebra is the Jiang-Su algebra $\mathcal{Z}$, which is constructed as a unital separable simple monotracial RAF-algebra whose ordered $K_0$-group is $(\Z, \Z^+)$. In \cite[Theorem 1]{GJS}, Gong, Jiang, and Su showed that the ordered $K_0$-group of a unital simple $\mathcal{Z}$-absorbing \Cs{} is weakly unperforated. Part {\rm (ii)} of the following Proposition  \ref{Prop4.3} is a variant of their argument for RAF-algebras in the absence of simplicity.

\begin{proposition}\label{Prop4.3}
{}\
\begin{itemize}
\item[\rm (i)] Let $A$ be a $C^*$-algebra and ${\frak n}$ a supernatural number. Then the pre-ordered abelian groups $(K_0(A\otimes M_{\frak n}),$ $K_0(A\otimes M_{\frak n})^+)$ and $(K_0(A)\otimes \D_{\frak n}, {K_0(A)\otimes \D_{\frak n}}^+)$ are isomorphic (in a natural way).
\item[\rm (ii)] Let $A$ be a $\mathcal{Z}$-absorbing RAF-algebra. Then the partially ordered abelian group $(K_0(A), K_0(A)^+)$ is unperforated, and, furthermore, is a rational dimension group. 
\end{itemize}
\end{proposition}

\begin{proof}
\noindent{\rm (i)}. 
It is enough to note that, for any natural number $n$, the Murray-von Neumann semigroup of $A\otimes M_n$ is the same (by definition---with respect to the map $A\rightarrow A\otimes e_{1 1}\subset A\otimes M_n$) as that of $A$ (the image of which  in $K_0(A)$ is the positive cone), and the semigroup  endomorphism corresponding to the map $A\rightarrow A\otimes 1_{M_n}\subset A\otimes M_n$ is just multiplication by $n$. (Recall that both the Murray-von Neumann semigroup functor and the $K_0$ functor preserve inductive limits.)

\noindent{\rm (ii)}. If $A$ is a RAF-algebra, then so also is $A\otimes\mathcal{K}$. Since the ordered $K_0$-groups $(K_0(A), K_0(A)^+)$ and $(K_0(A\otimes\mathcal{K}), K_0(A\otimes\mathcal{K})^+)$ are isomorphic, without loss of generality we may assume that $A$ is stable.

Let $\iota$ denote the canonical embedding of $A$ into $A\otimes\mathcal{Z}$ defined by $\iota(a)=a\otimes 1_{\mathcal Z}$ for $a\in A$, and $\iota_*$ the induced map by $\iota$ from $K_0(A)$ to $K_0(A\otimes\mathcal{Z})$. As in the proof of \cite[Theorem 1]{GJS}, for $g\in K_0(A)$ it is enough to show that $\iota_*(g)\in K_0(A\otimes\mathcal{Z})^+$ if and only if $ng\in K_0(A)^+$ for some $n\in\N$. Indeed, if $x\in K_0(A\otimes\mathcal{Z})$ satisfies $nx\in K_0(A\otimes\mathcal{Z})^+$ for some $n\in\N$, then, since (by \cite{Sch}) $\iota_*$ is a group isomorphism,
there exists $g_x\in K_0(A)$ such that $\iota_*(g_x)=x$. Then we obtain $m\in\N$ such that $mng_x\in K_0(A)^+$ which implies $x\in K_0(A\otimes\mathcal{Z})^+$.

Suppose that $g\in K_0(A)$ satisfies $\iota_*(g)\in K_0(A\otimes\mathcal{Z})^+$. Recall that the Jiang-Su algebra is constructed as the inductive limit \Cs{} $\lim\limits_{\longrightarrow}(\mathcal{Z}_{p_nq_n}, \varphi_n)$ of prime dimension drop algebras $\mathcal{Z}_{p_n q_n}$, $n\in\N$,  and the connecting maps $\varphi_n : \mathcal{Z}_{p_nq_n}$ $\rightarrow\mathcal{Z}_{p_{n+1} q_{n+1}}$, $n\in\N$, where $p_n$ and $q_n$ are relatively prime numbers; see \cite{JS}. Note that $A\otimes\mathcal{Z}_{p_n q_n}$ is stably finite for each $n\in\N$: then $(K_0(A\otimes \mathcal{Z}_{p_n q_n}), K_0(A\otimes \mathcal{Z}_{p_n q_n})^+)$ is a partially ordered abelian group. By the continuity of the functor $K_0$ (see \cite[Theorem 6.3.2]{RLL}, for example), there exist embeddings $\psi_n : A\otimes\mathcal{Z}_{p_n q_n}\rightarrow A\otimes \mathcal{Z}$ such that $K_0(A\otimes\mathcal{Z})^+=\bigcup_{n=1}^{\infty} \psi_{n*}(K_0(A\otimes\mathcal{Z}_{p_n q_n})^+)$. Then there exist $N\in\N$ and $x\in K_0(A\otimes\mathcal{Z}_{p_N q_N})^+$ such that $\iota_*(g)=\psi_{N*}(x)$. Let $\iota_N : A\rightarrow A\otimes\mathcal{Z}_{p_N q_N}$ denote the canonical embedding defined by $\iota_N(a)=a\otimes 1_{\mathcal{Z}_{p_N q_N}}$ for $a\in A$ and $\ev_0 : A\otimes\mathcal{Z}_{p_N q_N}\rightarrow A\otimes M_{p_N}$ the evaluation map determined by $\ev_0(f)\otimes 1_{M_{q_N}}=f(0)$ for $f\in A\otimes\mathcal{Z}_{p_N q_N}$. 
Since $\ev_0\circ\iota_N(a)=a\otimes 1_{M_{p_N}}$ for any $a\in A$, it follows that $\ev_{0 *}\circ \iota_{N *}(g)=p_N g$. 
Since $\iota_{N*}(g)=x\in K_0(A\otimes\mathcal{Z}_{p_N q_N})^+$, we have  $p_N g=\ev_{0*}\circ\iota_{N*}(g)\in K_0(A)^+$. 

Conversely, suppose that $g\in K_0(A)$ satisfies $ng\in K_0(A)^+$ for some $n\in\N$. Let $p$  and $q\in \N\setminus\{1\}$ be relatively prime natural numbers, and denote by $\iota_i : A\rightarrow A\otimes M_{i^{\infty}}$, $i=p, q$, the canonical embeddings defined by $\iota_i(a)=a\otimes 1_{M_{i^{\infty}}}$ for $a\in A$. Since $K_0(A\otimes M_{i^{\infty}})$, $i=p, q$, are ordered abelian groups, there exist $a_i$, $b_i\in K_0(A\otimes M_{i^{\infty}})^+$ such that $\iota_{i *}(g)=a_i-b_i$ for both $i=p, q$. Since $K_0(A\otimes M_{i^{\infty}})$ is unperforated and $n(a_i- b_i)=\iota_{i*}(ng)\in K_0(A\otimes M_{i^{\infty}})^+$ for $i= p, q$, it follows that $a_i-b_i\in K_0(A\otimes M_{i^{\infty}})^+$. Because $A\otimes M_{i^{\infty}}$, $i=p, q$ are stable AF-algebras, we obtain projections $y_i\in A\otimes M_{i^{\infty}}$, $i=p, q$ such that $[y_i]_0=a_i -b_i$. Identifying $A\otimes M_{p^{\infty}}\otimes M_{q^{\infty}}$ with $A\otimes M_{q^{\infty}}\otimes M_{p^{\infty}}$ canonically, we have $[y_p\otimes 1_{M_{q^{\infty}}}]_0=[y_q\otimes 1_{M_{p^{\infty}}}]_0$ in $K_0(A\otimes M_{p^{\infty}}\otimes M_{q^{\infty}})$. Since $A\otimes M_{p^{\infty}}\otimes M_{q^{\infty}}$ is an AF-algebra, there exists a projection $z\in C([0, 1])\otimes A\otimes M_{p^{\infty}}\otimes M_{q^{\infty}}$ such that $z(0)=y_p\otimes 1_{M_{q^{\infty}}}$ and $z(1)=y_q\otimes 1_{M_{p^{\infty}}}$. Considering the \Cs{}
\[\mathcal{Z}_{\infty}=\{f\in C([0, 1])\otimes M_{p^{\infty}}\otimes M_{q^{\infty}}\ : \ f(0)\in M_{p^{\infty}}\otimes 1_{M_{q^{\infty}}}, f(1)\in 1_{M_{p^{\infty}}}\otimes M_{q^{\infty}}\},\]
introduced in \cite{RW}, we can regard $z$ as a projection in $A\otimes\mathcal{Z}_{\infty}$. Denote by $\iota_{\infty} : A\rightarrow A\otimes \mathcal{Z}_{\infty}$ the canonical embedding defined by $\iota_{\infty}(a)= a\otimes 1_{\mathcal{Z}_{\infty}}$ for $a\in A$, and by $\Ev_i: A\otimes\mathcal{Z}_{\infty} \rightarrow A\otimes M_{i^{\infty}}$, $i=p, q$, the evaluation maps determined by $\Ev_p(f)\otimes 1_{M_{q^{\infty}}}=f(0)$ and $\Ev_q(f)\otimes 1_{M_{p^{\infty}}}=f(1)$ for $f\in A\otimes\mathcal{Z}_{\infty}$. Since $K_1(A\otimes M_{p^{\infty}}\otimes M_{q^{\infty}})=0$, the induced map $(\Ev_p\oplus \Ev_q)_*: K_0(A\otimes\mathcal{Z}_{\infty})\rightarrow K_0(A\otimes M_{p^{\infty}})\oplus K_0(A\otimes M_{q^{\infty}})$ is injective. Since 
\[ (\Ev_p\oplus \Ev_q)_*([z]_0-\iota_{\infty *}(g))=([y_p]_0\oplus [y_q]_0)-(\iota_{p*}(g)\oplus\iota_{q*}(g))=0,\]
considering the embedding $\iota_{\mathcal{Z}_{\infty}} : A\otimes\mathcal{Z}_{\infty}\rightarrow A\otimes \mathcal{Z}$ we conclude that $\iota_*(g)=\iota_{\mathcal{Z}_{\infty}*}\circ\iota_{\infty *}(g)\in K_0(A\otimes\mathcal{Z})^+$.
 
This shows that $(K_0(A), K_0(A)^+)$ is an unperforated partially ordered abelian group. On using once more that $A$ is RAF,  the second statement follows from {\rm (i)}.
 \end{proof}

Our goal in this section is to provide a construction of RAF-algebras which exhausts all countable rational dimension groups as $K_0$-groups. Here we emphasize that the non-simple cases of RAF-algebras are required in order to realize the rational dimension group $(G_{\Z}, G_{\Z}^+)$ of Section 3. The construction is somewhat analogous to the construction of simple \Cs s in \cite{Ell} and \cite{JS}.

As a corollary, we also show that a natural addition to the invariant in the non-stable case is exhausted (Corollary \ref{Cor4.8}). In Section 5, we shall show that the augmented invariant ($K_0$ alone in the stable case) is complete.

\begin{theorem}\label{Thm4.4}
For any countable rational dimension group $(G, G^+)$, there exists a separable stable RAF-algebra $A_G$ such that $A_G\otimes\mathcal{Z}\cong A_G$ and $(K_0(A_G), K_0(A_G)^+)$ is isomorphic to $(G, G^+)$ as an ordered abelian group.
\end{theorem}

Let $(G, G^+)$ be a countable rational dimension group. For $i=0, 1$, we obtain a stable AF-algebra $\AF_i$ such that  $(K_0(\AF_i), K_0(\AF_i)^+)$ is isomorphic to $(G\otimes \D_{(2+i)^{\infty}}, {G\otimes \D_{(2+i)^{\infty}}}^+)$ as an ordered abelian group. Because of the Elliott classification theorem \cite{Ell01} and Proposition \ref{Prop4.3} (i),  we can see that 
\[ \AF_0\otimes M_{3^{\infty}}\cong \AF_1\otimes M_{2^{\infty}}.\]
To simplify notation, we suppress the isomorphism between $\AF_0\otimes M_{3^{\infty}}$ and $\AF_1\otimes M_{2^{\infty}}$ in the rest of this section, and set $\AF=\AF_0\otimes M_{3^{\infty}}=\AF_1\otimes M_{2^{\infty}}$. We denote by $\iota_{\infty}^{(i)}$ the canonical embedding of $\AF_i$ into $\AF$ defined by $\iota_{\infty}^{(i)}(a)=a\otimes 1_{M_{(3-i)^{\infty}}}$ for $i=0, 1$ and $a\in \AF_i$. The desired RAF-algebra $A_G$ will be constructed as an inductive limit whose building blocks are type I $\mathrm{C}^*$-subalgebras of the following generalized dimension drop algebra $I_G$.

We define a \Cs{} $I_G$ by 
\[ I_G=\{f\in C([0, 1])\otimes \AF\ : \ f(i)\in \Image(\iota_{\infty}^{(i)})\text{ for both }i=0, 1\}.\]

\begin{proposition}\label{Prop4.5}
Let $A_i$, $i=0, 1$, be two AF-embeddable $C^*$-algebras.  Suppose that there exist an AF-algebra $A$ and embeddings $\iota^{(i)}$, $i=0, 1$, of $A_i$ into $A$. Then the $C^*$-algebra $J$ defined by 
\[ J=\{f\in C([0, 1])\otimes A\ : \ f(i)\in \Image(\iota^{(i)})\text{ for both }i=0, 1\}\]
has the following properties.
\begin{itemize}
\item[\rm (i)] Suppose that the induced maps $\iota_*^{(i)} : K_0(A_i)\rightarrow K_0(A)$, $i=0, 1$,  are injective. Let $\ev_i : J\rightarrow \Image(\iota^{(i)})$, $i=0, 1$, denote the evaluation maps  at these two points: $\ev_i (f)=f(i)$. Then, on identifying $K_0(\Image(\iota^{(i)}))$ with $\Image(\iota_*^{(i)})$, the induced maps $\ev_{i*} :K_0(J)\rightarrow\Image(\iota_*^{(i)})$, $i=0, 1$, satisfy $\ev_{0*}=\ev_{1*}$ and $\ev_{0*} : K_0(J)\rightarrow\Image(\iota_*^{(0)})\cap\Image(\iota_*^{(1)})$ is a group isomorphism. 
\item[\rm (ii)] If $\iota^{(i)}$, $i=0, 1$, satisfy the assumption of {\rm (i)} and $\iota_*^{(i)}(K_0(A_i)^+)=\iota_*^{(i)}(K_0(A_i))\cap K_0(A)^+$ for both $i=0, 1$, then $\ev_{0*}$ is an isomorphism of ordered abelian groups from $(K_0(J), K_0(J)^+)$ to $(\Image(\iota_*^{(0)})\cap \Image(\iota_*^{(1)}), \Image(\iota_*^{(0)})\cap \Image(\iota_*^{(1)})\cap K_0(A)^+)$. 
\item[\rm (iii)] In particular, for $A=\AF$, $A_i=\AF_i$, and $\iota^{(i)}=\iota_{\infty}^{(i)}$, $i=0, 1$, it follows that $(K_0(I_G), K_0(I_G)^+)$ is isomorphic to $(G, G^+)$ as an ordered abelian group.
\end{itemize}
\end{proposition}

\begin{proof}{}\
\noindent{\rm (i)}. Denote by $\iota$ the canonical embedding of $C_0((0, 1))\otimes A$ into $J$. From the exact sequence 
\[0\longrightarrow C_0((0,1))\otimes A\xrightarrow{\ \iota\quad} J\xrightarrow{\ev_0\oplus \ev_1} \Image(\iota^{(0)})\oplus \Image(\iota^{(1)})\longrightarrow0,\]
we obtain the six-term exact sequence of Bott periodicity which implies the following exact sequence:
\[0\longrightarrow K_0(J)\xrightarrow{(\ev_0\oplus\ev_1)_*} \Image(\iota_*^{(0)})\oplus\Image(\iota_*^{(1)})\xrightarrow{\ \partial\quad} K_0(A)\longrightarrow0,\]
where $\partial$ is the exponentioal map from $K_0(\Image(\iota^{(0)}))\oplus K_0(\Image(\iota^{(1)}))$ to $K_1( C_0((0, 1))\otimes A)\cong K_0(A)$. Then it is not  hard to check that $\partial (x\oplus y)= x-y$ for $x\in \Image(\iota_*^{(0)})$ and $y\in \Image(\iota_*^{(1)})$. Set $\Phi : \Image(\iota_*^{(0)})\cap \Image(\iota_*^{(1)})\rightarrow \Image(\iota_*^{(0)})\oplus\Image(\iota_*^{(1)})$ by $\Phi(x)=x\oplus x$ for $x\in \Image(\iota_*^{(0)})\cap \Image(\iota_*^{(1)})$. Then the group homomorphism $(\ev_0\oplus \ev_1)_*$ is an isomorphism onto $\Image(\Phi)$. Thus we see that $\ev_{0*}=\Phi^{-1}\circ(\ev_0\oplus\ev_1)_*=\ev_{1*}$ is a group isomorphism. 

\noindent{\rm (ii)}. From the assumption $\iota_*^{(i)}(K_0(A_i)^+)=\iota_*^{(i)}(K_0(A_i))\cap K_0(A)^+$, for $x\in \Image(\iota_*^{(0)})\cap\Image(\iota_*^{(1)})\cap K_0(A)^+$, there exist $N\in\N$ and projections $p_i\in A_i\otimes M_N$, $i=0, 1$, such that $\iota_*^{(i)}([p_i]_0)= x$ for both $i=0, 1$. Since $A\otimes M_N$ is an AF-algebra and $[\iota^{(0)}\otimes\id_{M_N}(p_0)]_0=[\iota^{(1)}\otimes\id_{M_N}(p_1)]_0$ in $K_0(A)$, there exists a projection $\tp$  $\in C([0, 1])\otimes A\otimes M_N$ such that $\tp(i)=\iota^{(i)}\otimes\id_{M_N}(p_i)$ for $i=0, 1$. Regarding $\tp$ as a projection in $J\otimes M_N$, we have that $(\ev_0\oplus\ev_1)_*([\tp]_0)=\iota_*^{(0)}([p]_0)\oplus\iota_*^{(1)}([p]_0)=x\oplus x$. Then it follows that $x=\Phi^{-1}\circ(\ev_0\oplus \ev_1)_*([\tp]_0)=\ev_{0*}([\tp]_0)\in \ev_{0*}(K_0(J)^+)$, which implies that $\ev_{0*}(K_0(J)^+)=\Image(\iota_*^{(0)})$ $\cap\Image(\iota_*^{(1)})\cap K_0(A)^+$. 

\noindent{\rm (iii)}. For $i=0, 1$, since $K_0(A_i)$ is torsion free, it follows that the induced map $\iota_*^{(i)} : K_0(A_i)\rightarrow K_0(A)\cong K_0(A_i)\otimes \D_{(3-i)^{\infty}}$ is injective. By Proposition \ref{Prop4.3} {\rm (i)}, it follows that 
\begin{align*}
(K_0(A), K_0(A)^+)&\cong(K_0(A_0)\otimes\D_{3^{\infty}}, {K_0(A_0)\otimes \D_{3^{\infty}}}^+)\\
&\cong (K_0(A_1)\otimes \D_{2^{\infty}}, {K_0(A_1)\otimes \D_{2^{\infty}}}^+),
\end{align*}
as ordered abelian groups. Then, for $i=0, 1$ and $x\in \Image (\iota_*^{(i)})\cap K_0(A)^+$, there exists $d\in\N$ such that $(3-i)^dx\in \iota_*^{(i)}(K_0(A_i)^+)$. Since $(K_0(A_i), K_0(A_i)^+)$ is unperforated, it follows that $x\in \iota_*^{(i)}(K_0(A_i)^+)$, which implies that $\Image(\iota_*^{(i)})\cap K_0(A)^+=\iota_*^{(i)}(K_0(A_i)^+)$ for both $i= 0, 1$. Applying {\rm (ii)}, we see that $(K_0(I_G), K_0(I_G)^+)$ is isomorphic to $(\Image(\iota_*^{(0)})\cap\Image(\iota_*^{(1)}), \Image(\iota_*^{(0)})\cap\Image(\iota_*^{(1)})\cap K_0(A)^+)$. 

By Lemma \ref{Lem3.1} {\rm (i)},  any rational dimension group is torsion free, and by Lemma \ref{Lem3.1} {\rm (ii)}, we have $\Image(\iota_*^{(0)})\cap\Image(\iota_*^{(1)})=G\otimes\D_{2^{\infty}}\otimes 1_{\D_{3^{\infty}}}\cap G\otimes 1_{\D_{2^{\infty}}}\otimes \D_{3^{\infty}}\cong G$. Let $\Psi : G\rightarrow G\otimes \D_{2^{\infty}}\otimes \D_{3^{\infty}}$ denote the positive group homomorphism defined by $\Psi(g)=g\otimes 1_{\D_{2^{\infty}}}\otimes 1_{\D_{3^{\infty}}}$ for $g\in G$. The group isomorphism of (the proof of) Lemma \ref{Lem3.1} {\rm (ii)} is just the map  $\Psi$. Furthermore, since $G$ is unperforated it is straightforward to check that $\Psi(G^+)=\Psi(G)\cap (G\otimes \D_{2^{\infty}}\otimes\D_{3^{\infty}})^+$. Therefore we conclude that 
\begin{align*}
&(\Image(\iota_*^{(0)})\cap\Image(\iota_*^{(1)}), \Image(\iota_*^{(0)})\cap\Image(\iota_*^{(1)})\cap K_0(A)^+) \\
&\cong(\Psi(G), \Psi(G)\cap (G\otimes \D_{2^{\infty}}\otimes \D_{3^{\infty}})^+)\cong(G, G^+).
\end{align*}\end{proof}

Before going into the proof, we collect some necessary well-known facts concerning  finite dimensional \Cs s.
\begin{lemma}\label{Lem4.6}
{}\
\begin{itemize}
\item[\rm (i)] If natural numbers $P$, $Q$, $R$ and a projection $p\in M_Q\otimes M_R$ satisfy $R\leq\rank(p)$, $R|\rank (p)$, $\frac{\rank(p)}{R}<Q$, and $P|(Q-\frac{\rank(p)}{R})$, then there exists a unitary $u\in M_P\otimes M_Q\otimes M_R$ such that: for any \Cs{} $A$, $a\in A$, and $b\in A\otimes M_P$,
\[\Ad(1_{A^{\sim}}\otimes u)((a\otimes 1_{M_P}\otimes p)\oplus(b\otimes(1_{M_Q\otimes M_R}-p)))\in A\otimes 1_{M_P}\otimes M_Q\otimes 1_{M_R}.\]
\item[\rm (ii)] For $\varepsilon>0$ and $N\in\N$, there exists $\delta>0$ such that: if $A$ and $B$ are $C^*$-subalgebras of a unital $C^*$-algebra $\cA$ such that $\dim(B)\leq N$ and $B^1\subsetdelta A$, then there exists a unitary $u\in\cA$ such that $u\approx_{\varepsilon}1_{\cA}$ and $\Ad u(B)\subset A$. 

\item[\rm (iii)] For $\varepsilon>0$ and $N\in\N$, there exists $\delta>0$ such that: if a finite dimensional $C^*$-algebra $A$ and two embeddings $\iota_0$ and $\iota_1$ of $A$ into a unital $C^*$-algebra $\cA$ satisfy $\dim (A)\leq N$ and $\|\iota_0(a)-\iota_1(a)\|<\delta$ for all $a\in A^1$, then there exists a unitary $u\in\cA$ such that $\Ad u\circ\iota_0=\iota_1$ and $\|u-1_{\cA}\|<\varepsilon$. 

\item[\rm (iv)] For $\varepsilon >0$ and $N\in\N$, there exists $\delta>0$ such that: if $A$, $B$, and $C$ are $C^*$-subalgebras of a unital $C^*$-algebra $\cA$ satisfying $C\subset A\cap B$, $B^1\subsetdelta A$, and $\dim(B), \dim(C)\leq N$, then there exists a unitary $y\in\cA$ such that $y\in C'\cap \cA$, $\Ad y(B)\subset A$, and $\|y-1_{\cA}\|<\varepsilon$. 

\item[\rm (v)] Let $A$ and $B$ be $C^*$-subalgebras of a unital $C^*$-algebra $\cA$. Suppose that $A$ and $B$ are finite dimensional $C^*$-algebras and unitaries $u$, $v\in\cA$ satisfy
$\{u, v\}\subset_{1/16} B+\C1_{\cA}$ and $\Ad u(A)\cup \Ad v(A)\subset B$. Then there exists a unitary $w\in B + \C1_{\cA}$ such that $\Ad wu(a)=\Ad v(a)$ for any $a\in A$.
\end{itemize}
\end{lemma} 

\begin{proof} {}

\noindent{\rm (i)}. Set a projection $e\in M_Q$ with $\rank (e)=\frac{\rank(p)}{R}$. By $P|(Q-\rank(e))$, we obtain a unital embedding $\iota : M_P\rightarrow (1_{M_Q}-e)M_Q(1_{M_Q}-e)$. Since $P\rank(p) =PR\rank(e)$, there exists a partial isometry $v\in M_P\otimes M_Q\otimes M_R$ such that $v^*v=1_{M_P}\otimes p$ and $vv^*=1_{M_P}\otimes e \otimes 1_{M_R}$. Denote by $\eta$  the canonical embedding of $M_P$ into $M_P\otimes(1_{M_Q\otimes M_R}-p)$ defined by $\eta(a)=a\otimes(1_{M_Q\otimes M_R}-p)$ for $a\in M_P$. Since the multiplicity of $\eta$ is $QR-\rank(p)$, the same as that of $1_{M_P}\otimes\iota\otimes 1_{M_R}$, there exists a partial isometry $w\in M_P\otimes M_Q\otimes M_R$ such that $w^*w=\eta(1_{M_P})$, $ww^*=1_{M_P}\otimes\iota(1_{M_P})\otimes 1_{M_R}$, and $\Ad w\circ\eta = 1_{M_P}\otimes \iota \otimes 1_{M_R}$. The unitary $u=v+ w$ in $M_P\otimes M_Q\otimes M_R$ then satisfies the desired conditions. Indeed, for any \Cs{} $A$, $a$, $b'\in A$, and $x\in M_P$, 
\begin{align*}
&\Ad(1_{A^{\sim}}\otimes u)((a\otimes 1_{M_P}\otimes p)\oplus (b'\otimes x\otimes(1_{M_Q\otimes M_R}-p))) \\
& = a\otimes 1_{M_P}\otimes e\otimes 1_{M_R} + b'\otimes 1_{M_P}\otimes \iota(x)\otimes 1_{M_R}\in A\otimes 1_{M_P}\otimes M_Q\otimes 1_{M_R}.
\end{align*}

\noindent{(ii)}. 
See \cite[Lemma III 3.2]{Dav}.

\noindent{(iii)}.
Since $A$ is finite dimensional, we may identify $A$ with $\bigoplus_{l=1}^L M_{k_l}$ for some $k_1, k_2, ..., k_L\in\N$. For given $\varepsilon >0$ and $N\in\N$ there exists $\delta\in (0, \varepsilon/4N)$ such that: if two projections $p$ and $q$ in $\cA$ satisfy $\|p-q\|<\delta$ then there exists a partial isometry $w\in \cA$ such that $w^*w=p$, $ww^*=q$, and $w\approx_{\varepsilon/4N} p$. Let us show that $\delta$ is as required. Suppose that $A$, $\iota_0$, and $\iota_1$ satisfy the hypotheses.
Then there are partial isometries $w_l\in\cA$, $l=1,2,...,L$, such that $w_l^*w_l=\iota_0(e_{11}^{(l)})$, $w_lw_l^*=\iota_1(e_{11}^{(l)})$, and $w_l\approx_{\varepsilon / 4N}\iota_0(e_{11}^{(l)})$. 
Consider the partial isometry $ w=\sum_{l=1}^L\sum_{j=1}^{k_l}\iota_1(e_{j,1}^{(l)})w_l \iota_0(e_{1,j}^{(l)})\in\cA$. 
We have  $\Ad w\circ \iota_0=\iota_1$ and $w\approx_{\varepsilon/2} \iota_0(1_A)$. Since $1_{\cA} -\iota_0(1_A)\approx_{\delta}1_{\cA}-\iota_1(1_A)$, there exists a partial isometry $v\in \cA$ such that $v^*v=1_{\cA}-\iota_0(1_A)$, $vv^*=1_{\cA}-\iota_1(1_A)$, and $v\approx_{\varepsilon/2}1_{\cA}-\iota_0(1_A)$. The unitary $u=v+w\in \cA$ satisfies the stipulated conditions.

\noindent{\rm (iv)}. Applying {\rm (iii)}, we obtain $\delta'\in (0, \varepsilon)$ satisfying the condition of {\rm (iii)} for $\varepsilon/2>0$ and $N\in\N$. Applying {\rm (ii)}, we also obtain $\delta>0$ satisfying the condition of {\rm (ii)} for $\delta'/2>0$ and $N\in\N$. This $\delta$ satisfies the desired condition. 
Indeed, if $A$, $B$, and $C$ satisfy the assumption of {\rm (iv)} for $\delta >0$ and $N\in\N$, then there exists a unitary $u\in \cA$ such that $u\approx_{\delta'/2} 1_{\cA}$ and $\Ad u(B)\subset A$, which implies that $\Ad u(c)\approx_{\delta'} c$ for any $c\in C^1$. Thus we obtain a unitary $w\in A+\C 1_{\cA}$ such that $\Ad wu(c)=c$ for any $c\in C$ and $w\approx_{\varepsilon/2} 1_{\cA}$. The unitary $y=wu\in\cA$ satisfies the conditions. 

\noindent{\rm (v)}. Let $A$ and $\bigoplus_{l=1}^LM_{k_l}$ be as in the proof of {\rm (iii)}. Define a unitary $w''\in\cA$ by $w''=vu^*$. By hypothesis, there exists a unitary $w'$ in $B+\C1_{\cA}$ such that $w'\approx_{1/2}w"$. It follows that $\Ad w'u(e_{11}^{(l)})\approx_{1} \Ad v(e_{11}^{(l)})$ for all $l=1,2,..., L$. Thus there exists a partial isometry $w_l\in B$ such that $w_l^*w_l=\Ad w' u(e_{11}^{(l)})$ and $w_lw_l^*=\Ad v(e_{11}^{(l)})$. Set $\barw=\left(\sum_{l=1}^{N}\sum_{j=1}^{k_l}\Ad v(e_{j,1}^{(l)})w_l\Ad w'u(e_{1j}^{(l)})\right)w'$; $\barw$ is a partial isometry in $B$ such that $\Ad \barw u (a) =\Ad v(a)$ for any $a\in A$. Since $B$ is finite dimensional, we can extend $\barw$ to a unitary $w$ in $B+\C 1_{\cA}$ such that $\Ad wu(a)=\Ad v(a)$ for any $a\in A$. 
\end{proof}

In the following proof, for a given \Cs{} $A$ and a Lipschitz continuous function $f\in C([0, 1])\otimes A$ we denote by $\Lip(f)$ the Lipschitz constant of $f$. We shall denote by $1_{AF}$ and $1_{AF_i}$ the units of the unitizations $\AF^{\sim}$ and $\AF_i^{\sim}$. 

\begin{proof}[Proof of Theorem \ref{Thm4.4}]
Let $F_n$, $n\in\N$, be an increasing sequence of finite subsets of $G^+$ whose union is $G^+$ with $F_1=\{0\}$, and $\varepsilon_n>0$, $n\in\N$, be a decreasing sequence such that $\sum\limits_{n\in\N}\varepsilon_n <1$. Let $F_{A n}$ (resp. $F_{{A}_i n}$), $n\in\N$, be an increasing sequence of finite subsets of $\AF^1$ (resp. $\AF_i^1$) whose union is dense in $\AF^1$ (resp. $\AF_i^1$) and $F_{A 1}=\{0\}$ (resp. $F_{A_i 1}=\{0\}$ for $i=0, 1$). For $n\in\N$, we shall inductively construct numbers $L_n$, $\Lambda_n$, $\tL_n\in \N\cup\{0\}$, finite dimensional $\mathrm{C}^*$-subalgebras $A_n^{(i)}, B_n^{(i)}\subset \AF_i$, $i=0, 1$, and $A_n\subset\AF$, embeddings $\epsilon_1: \{0\}\rightarrow \AF$, $\epsilon_n : M_{2^{L_n}}\otimes M_{3^{L_n}}\rightarrow A_{n-1}'\cap \AF$ (for $n\geq 2$), $\iota_n^{(i)} : A_n^{(i)}\rightarrow A_n$, $\kappa_n^{(i)} : B_n^{(i)}\otimes M_{(3-i)^{\Lambda_n}}\rightarrow \AF$ for $i=0, 1$, a $\mathrm{C}^*$-subalgebra $\cA_n$ of $C([0, 1])\otimes A_n$, and an injective $*$-homomorphism $\varphi_n : \cA_{n-1}\rightarrow \cA_n$ (where $\cA_0=\{0\}$) satisfying the following conditions:
\begin{itemize}
\item[\rm (1, $n$)]
$\Lambda_n>L_n+\Lambda_{n-1}$, $\tL_n>L_n\geq \Lambda_{n-1}$,
\item[\rm (2, $n$)]
for $i=0, 1$, on regarding $(K_0(\AF_i), K_0(\AF_i)^+)$ as $(G\otimes\D_{(2+i)^{\infty}}, {G\otimes \D_{(2+i)^{\infty}}}^+)$, one has $ F_n\otimes 1_{(2+i)^{\infty}}\subset K_0(A_n^{(i)})^+$, $A_{n-1}^{(i)}\subset A_n^{(i)}\subset B_n^{(i)}$, and $F_{A_i n}\subset_{\varepsilon_n} A_n^{(i)}$,
\item[\rm (3, $n$)]
if $n=2n'+ i_n$ for some $n'\in\N\cup\{0\}$ and $i_n\in \{0, 1\}$, then there exists a unitary $U_n$ in $\AF^{\sim}$ such that 
\[U_n\approx_{\varepsilon_n} 1_{AF}\text{ and }(A_{n-1} +\C1_{AF})\Image(\epsilon_n)\subset \Ad U_n(A_n^{(1-i_n)}\otimes M_{(2+i_n)^{\tL_n}})=A_n,\]
\item[\rm (4, $n$)]
$\epsilon_n(1_{(M_{2^{L_n}}\otimes M_{3^{L_n}})})a = a$ for any $a\in A_{n-1}$, and $F_{A n}\subset_{\varepsilon_n}A_n^1$,
\item[\rm (5, $n$)]
$\cA_n$ is defined by 
\[\cA_n=\{f\in C([0, 1])\otimes A_n\ : \ f(i)\in \Image (\iota_n^{(i)})\text{ for }i=0,1\},\]
\item[\rm (6, $n$)]
there are unitaries $V_n^{(i)}$, $W_n^{(i)}\in \AF^{\sim}$, $i=0, 1$, such that 
\[\iota_n^{(i)}=\Ad V_n^{(i)}\circ\iota_{\infty}^{(i)}|_{A_n^{(i)}},\quad \kappa_n^{(i)}(b\otimes 1_{M_{(3-i)^{\Lambda_n}}})=\Ad W_n^{(i)}(b\otimes 1_{M_{(3-i)^{\infty}}}),\]
for any $i=0, 1$ and $b\in B_n^{(i)}$, and 
\[ V_n^{(i)}\in_{\varepsilon_n} A_n + \C 1_{AF}\quad\text{ and }\quad W_n^{(i)}\in_{\frac{\varepsilon_n}{4}} \Image(\kappa_n^{(i)})+\C 1_{AF},\]
\item[\rm (7, $n$)] $A_n\subset\Image(\kappa_n^{(i)})$ for $i=0, 1$, and 
\[ \iota_n^{(i)}(a)=\kappa_n^{(i)}(a\otimes 1_{M_{(3-i)^{\Lambda_n}}})\quad\text{for any }i=0, 1,\text{ and } a\in A_n^{(i)},\] 
\item[\rm (8, $n$)] one has
\[\Lip(\varphi_n(f))\leq \Lip(f)/2 + 2\varepsilon_n\]
for any Lipschitz continuous function $f\in \cA_{n-1}$ with $\|f\|\leq 1$.
\end{itemize}

For $n=1$ and $i=0,1$, setting $\Lambda_0=0$, $A_0^{(i)}=\{0\}\subset\AF_i$, and $A_0=\{0\}\subset \AF$,  define $L_1=0$, $\Lambda_1=\tL_1=1$, $A_1=\{0\}\subset\AF$, $A_1^{(i)}=B_1^{(i)}=\{0\}\subset \AF_i$, $U_1=V_1^{(i)}=W_1^{(i)}=1_{AF}$, and $\cA_0=\cA_1=\{0\}$. These choices satisfy {\rm (1, 1)}--{\rm (8, 1)}.

For the induction,  assume that $L_n$, $\Lambda_n$, $\tL_n$, $A^{(i)}_n$, $B_n^{(i)}$, $A_n$, $\epsilon_n$, $\iota_n^{(i)}$, $\kappa_n^{(i)}$, $\cA_n$, and $\varphi_n$ satisfying {\rm (1, $n$)}--{\rm (8, $n$)} have been constructed. Choose $L_{n+1}\in\N$ large enough that $2^{L_{n+1}}>3^{\Lambda_n}$. Then there are $k_{n+1}'$, $l_{n+1}'\in\N$ such that $k_{n+1}'< 3^{\Lambda_n}$, $l_{n+1}'< 2^{\Lambda_n}$, $3^{\Lambda_n}|2^{L_{n+1}}-k_{n+1}'$, and $2^{\Lambda_n}|3^{L_{n+1}}-l_{n+1}'$. Set $N_{n+1}=L_{n+1}+ \Lambda_n$, $k_{n+1}=3^{L_{n+1}}k_{n+1}'$, $l_{n+1}=2^{L_{n+1}}l_{n+1}'$, and $m_{n+1}=6^{L_{n+1}}$. Note that from $k_{n+1}'< 3^{\Lambda_n}\leq 2^{L_{n+1}}-k_{n+1}'$ and $l_{n+1}'< 2^{\Lambda_n}\leq 3^{L_{n+1}}-l_{n+1}'$ it follows that
\[ m_{n+1}-(k_{n+1}+l_{n+1})=(3^{L_{n+1}}-l_{n+1}')(2^{L_{n+1}}-k_{n+1}')-l_{n+1}'k_{n+1}'>0.\]
Set $N=(\dim(B_n^{(0)})+\dim(B_n^{(1)})+\dim(A_n))6^{2L_{n+1}}\in\N$, and $\delta_{n+1}^{(0)}= \varepsilon_{n+1}/8$. Applying Lemma \ref{Lem4.6} {\rm (ii), (iii), (iv)} inductively to $\delta_{n+1}^{(k)}$ $(=\varepsilon)>0$ and $N$, we obtain $\delta_{n+1}^{(k+1)}$ $(=\delta)>0$ verifying the estimates of {\rm (ii), (iii), (iv)} of Lemma \ref{Lem4.6}. 
Refining the choice of $\delta_{n+1}^{(k)}>0$, $k\in\N$, making each one smaller in turn,  we may suppose that $\delta_{n+1}^{(k+1)}<\delta_{n+1}^{(k)}$ for $k\in\N$ and 
\[ N\left(\sum_{j=1}^{\infty} 8^j \delta_{n+1}^{(k+j)}\right)<\delta_{n+1}^{(k)}\quad\text{for any }k\in\N.\]
Since $\AF_i\otimes M_{(2+i)^{\infty}}\cong \AF_i$ for $i=0, 1$, there are embeddings $\epsilon_{n+1}^{(i)} : M_{(2+i)^{L_{n+1}}}\rightarrow ({B_n^{(i)}})'\cap \AF_i$, $i=0, 1$, such that
\begin{align*}
(\Image(\epsilon_{n+1}^{(1)})\otimes 1_{M_{2^{\infty}}})^1&\subset_{\delta_{n+1}^{(10)}}\left(\Image(\epsilon_{n+1}^{(0)})\otimes 1_{M_{3^{\infty}}}\right)'\cap \AF, \\
b\epsilon_{n+1}^{(i)}(1_{M_{(2+i)^{L_{n+1}}}})&\approx_{\frac{\varepsilon_{n+1}}{4}} b\quad\text{for any }i=0, 1 \text{ and }b\in (B_n^{(i)})^1, \text{ and }\\
a(\epsilon_{n+1}^{(i)}(1_{M_{(2+i)^{L_{n+1}}}})\otimes 1_{M_{(3-i)^{\infty}}})&\approx_{\delta_{n+1}^{(10)}} a \quad\text{ for any }a\in A_n^1\cup\bigcup\limits_{j=0,1}\left(B_n^{(j)}\otimes 1_{M_{(3-j)^{\infty}}}\right)^1.
\end{align*}
By Lemma \ref{Lem4.6} {\rm (ii)}, we obtain a unitary $E_1$ in $\AF^{\sim}$ such that $E_1\approx_{\delta_{n+1}^{(9)}}1_{AF}$ and 
\[\Ad E_1(\Image(\epsilon_{n+1}^{(1)})\otimes 1_{M_{2^{\infty}}})\subset \left(\Image(\epsilon_{n+1}^{(0)})\otimes 1_{M_{3^{\infty}}}\right)'\cap \AF.\]
Set $\barepsilon_{n+1}(a\otimes b)=\epsilon_{n+1}^{(0)}(a)\otimes 1_{M_{3^{\infty}}}\Ad E_1(\epsilon_{n+1}^{(1)}(b)\otimes 1_{M_{2^{\infty}}})$ for $a\in M_{2^{L_{n+1}}}$, $b\in M_{3^{L_{n+1}}}$. 
Choosing $\epsilon_{n+1}^{(0)}$ and $\epsilon_{n+1}^{(1)}$ almost commuting with $\Image(\kappa_n^{(i)})^1$ and $W_n^{(i)}\!\!,$ we may  further assume that
\[\Image(\kappa_n^{(i)})^1\cup \{W_n^{(i)}\}\subset_{\delta_{n+1}^{(8)}}\Image(\barepsilon_{n+1})'\cap \AF,\]
for both $i=0, 1$.  Because of the inclusion $A_n\subset \Image(\kappa_n^{(i)})$ for $i=0, 1$ in {\rm (7, $n$)}, we have an embedding $\epsilon_{n+1} : M_{2^{L_{n+1}}}\otimes M_{3^{L_{n+1}}}\rightarrow A_n'\cap \AF$ such that $\epsilon_{n+1}(x)\approx_{\delta_{n+1}^{(6)}} \barepsilon_{n+1}(x)$ for all $x\in (M_{2^{L_{n+1}}}\otimes M_{3^{L_{n+1}}})^1$. Then it follows that $\Image(\epsilon_{n+1})^1\subset_{\delta_{n+1}^{(5)}}\Image(\kappa_n^{(i)})'\cap\AF$ for $i=0, 1$. Note that, since $1_{A_n}\epsilon_{n+1}(\oneLnplus)\approx_{\varepsilon_{n+1}} 1_{A_n}$ and $\epsilon_{n+1}(\oneLnplus)\in A_n'$, we have  $\epsilon_{n+1}(\oneLnplus) 1_{A_n}=1_{A_n}$ which implies 
$\epsilon_{n+1}(\oneLnplus)a=a$ for any $a\in A_n$ in {\rm (4, $n+1$)}. For the same reason, from 
$1_{A_n^{(i)}}\epsilon_{n+1}^{(i)}(1_{M_{(2+i)^{L_{n+1}}}})\approx_{\varepsilon_{n+1}} 1_{A_n^{(i)}}$ and $\epsilon_{n+1}^{(i)}(1_{M_{(2+i)^{L_{n+1}}}})\in {A_n^{(i)}}'$, it follows that $\epsilon_{n+1}^{(i)}(1_{M_{(2+i)^{L_{n+1}}}})a=a$ for all $i=0, 1$ and $a\in A_n^{(i)}$.

For $i=0, 1$, since $\AF_i\otimes M_{(3-i)^{\infty}}=\AF$, there exist $\tL_{n+1}\in\N$ and a finite dimensional $\mathrm{C}^*$-subalgebra $A_{n+1}^{(i)}$ of $\AF_i$ satisfying $F_{n+1}\otimes 1_{(2+i)^{\infty}}\subset K_0(A_{n+1}^{(i)})^+$ in {\rm (2, $n+1$)}, $\tL_{n+1}> L_{n+1}$, $A_n^{(i)}\subset A_{n+1}^{(i)}$, $F_{A_i n+1}\subset_{\varepsilon_{n+1}} A_{n+1}^{(i)}$, $A_n\subset_{\delta_{n+1}^{(10)}}A_{n+1}^{(i)}\otimes M_{(3-i)^{\tL_{n+1}}}$, and $B_n^{(i)}\Image(\epsilon_{n+1}^{(i)})\subset A_{n+1}^{(i)}$. 
When $n$ is an even number (resp.~an odd number), we have $i_{n+1}=1$ (resp.~$i_{n+1}=0$). By Lemma \ref{Lem4.6} (ii), there exists $\tdelta >0$ satisfying the condition for $\delta_{n+1}^{(10)}$ $(=\varepsilon)>0$ and $\dim(A_{n+1}^{(i_{n+1})})6^{2L_{n+1}}\in\N$. Enlarging the choice of $\tL_{n+1}\in\N$ and $A_{n+1}^{(1-i_{n+1})}$ further, we may choose $\tL_{n+1}$ and $A_{n+1}^{(1-i_{n+1})}$ such that 
\setcounter{equation}{0}
\begin{align}
&(A_{n+1}^{(i_{n+1})}\otimes M_{(3-i_{n+1})^{L_{n+1}}})^1\cup (A_n+\C1_{AF})\Image(\epsilon_{n+1})^1\cup \Image(\kappa_n^{(0)})^1\cup\Image(\kappa_n^{(1)})^1\cup F_{A n+1}\nonumber \\
&\subset_{\min\{\delta_{n+1}^{(10)}, \tdelta\}} A_{n+1}^{(1-i_{n+1})}\otimes M_{(2+i_{n+1})^{\tL_{n+1}}}, \text{ and} \\
&\quad \{W_n^{(0)}, W_n^{(1)}\}\subset_{\delta_{n+1}^{(10)}}A_{n+1}^{(1-i_{n+1})}\otimes M_{(2+i_{n+1})^{\tL_{n+1}}}+ \C1_{AF}.
\end{align}
By Lemma \ref{Lem4.6} (ii) for $\delta_{n+1}^{(9)}>0$ and $N\in\N$, there exists a unitary $U_{n+1}\in \AF^{\sim}$ satisfying the condition of (3, $n+1$) and $U_{n+1}\approx_{\delta_{n+1}^{(9)}}1_{AF}$. Define 
\[ A_{n+1}=\Ad U_{n+1} (A_{n+1}^{(1-i_{n+1})}\otimes M_{(2+i_{n+1})^{\tL_{n+1}}})\subset \AF.\] 
Note that from $F_{A n+1}\subset_{\delta_{n+1}^{(10)}} A_{n+1}^{(1-i_{n+1})}\otimes M_{(2+i_{n+1})^{\tL_{n+1}}}$, we see that $F_{A n+1}\subset_{\varepsilon_{n+1}}A_{n+1}^1$ 
as required for (4, $n+1$). By the conditions of (3, $n+1$) and (4, $n+1$), it follows that $A_n\subset A_{n+1}$. 
Since $ \bigcup_{i=0, 1}\Image(\kappa_n^{(i)})^1\subset_{\delta_{n+1}^{(8)}}A_{n+1}$, applying Lemma \ref{Lem4.6} (iv) to $A=A_{n+1}$, $B=\Image(\kappa_n^{(i)})$, and $C=A_n$, we obtain unitaries $y_n^{(i)}\in \AF^{\sim}\cap A_n'$, $i=0,1$, such that $\Ad y_n^{(i)}(\Image(\kappa_n^{(i)}))\subset A_{n+1}$ and $y_n^{(i)}\approx_{\delta_{n+1}^{(7)}} 1_{AF}$. Define embeddings $\barkappa_n^{(i)} : B_n^{(i)}\otimes M_{(3-i)^{\Lambda_n}}\rightarrow A_{n+1}$, $i=0, 1$, by $\barkappa_n^{(i)} =\Ad y_n^{(i)}\circ\kappa_n^{(i)}$. Note that, by (7,$\ n$), we also see that $\barkappa_n^{(i)}(a\otimes 1_{M_{(3-i)^{\Lambda_n}}})=\iota_n^{(i)}(a)$ for any $i=0, 1$ and $a\in A_n^{(i)}$. Now we have $A_n\subset\Image(\barkappa_n^{(i)})\subset A_{n+1}$ for both $i=0, 1$, $\Image(\epsilon_{n+1})\subset A_{n+1}$, and $\Image(\barkappa_n^{(i)})^1\subset_{\delta_{n+1}^{(4)}}A_{n+1}\cap \Image(\epsilon_{n+1})'$ for $i=0, 1$. Applying Lemma \ref{Lem4.6} (iv) to $A=A_{n+1}\cap\Image(\epsilon_{n+1})'$, $B=\Image(\barkappa_n^{(i)})$, and $C=A_n$, we obtain unitaries $\ty_n^{(i)}\in A_n'\cap \AF^{\sim}$, $i=0, 1$, such that $\ty_n^{(i)}\approx_{\delta_{n+1}^{(3)}}1_{AF}$ and $\Ad \ty_n^{(i)}(\Image(\barkappa_n^{(i)}))\subset  A_{n+1}\cap\Image(\epsilon_{n+1})'$. Define embeddings $\tkappa_n^{(i)} : B_n^{(i)} \otimes M_{(3-i)^{\Lambda_n}}\rightarrow A_{n+1}$, $i=0, 1$, by $\tkappa_n^{(i)} =\Ad \ty_n^{(i)}\circ\barkappa_n^{(i)}$, which satisfy $\tkappa_n^{(i)}(a\otimes 1)=\iota_n^{(i)}(a)$ for any $i=0, 1$ and $a\in A_n^{(i)}$, and also $A_n\subset\Image(\tkappa_n^{(i)})\subset A_{n+1}\cap\Image(\epsilon_{n+1})'$ for both $i=0,1$. 

To simplify notation,  let $\{e_{i j}\}_{i, j=1}^{m_{n+1}}$ be a system of matrix units for $M_{m_{n+1}}\cong M_{2^{L_{n+1}}}\otimes M_{3^{L_{n+1}}}$. Note that $m_{n+1}> k_{n+1}+ l_{n+1}$. Define an embedding $\xi$ from $C([0, 1])\otimes A_n$ into $C([0, 1])\otimes A_{n+1}$ by
\[
\xi(f)(t)=\sum_{j=1}^{k_{n+1}}f\left(\frac{t}{2}\right)\epsilon_{n+1}(e_{jj})+\sum_{j=1+k_{n+1}}^{m_{n+1}-l_{n+1}}f\left(\frac{1}{2}\right)\epsilon_{n+1}(e_{jj})+\sum_{j=1+m_{n+1}-l_{n+1}}^{m_{n+1}} f\left(\frac{1+t}{2}\right)\epsilon_{n+1}(e_{jj}),\]
for $f\in C([0,1])\otimes A_n$ and $t\in [0, 1]$. 
Set $p=\epsilon_{n+1}\left(\sum_{j=1}^{k_{n+1}}e_{jj}\right)$. Since $\tkappa_n^{(0)}(a\otimes 1_{M_{3^{\Lambda_n}}})=\iota_n^{(0)}(a)$ for any $a\in A_n^{(0)}$ and $A_n\subset\Image(\tkappa_n^{(0)})$, for any $f\in\cA_n$ there exist $a_f\in A_n^{(0)}\subset B_n^{(0)}$ and $b_f\in B_n^{(0)}\otimes M_{3^{\Lambda_n}}$ such that 
\begin{align*}
\xi(f)(0)&=\sum_{j=1}^{k_{n+1}}f(0)\epsilon_{n+1}(e_{jj})+\sum_{j=1+k_{n+1}}^{m_{n+1}}f\left(\frac{1}{2}\right)\epsilon_{n+1}(e_{jj}) \\
&=\tkappa_n^{(0)}(a_f\otimes 1)p +\tkappa_n^{(0)}(b_f)(\epsilon_{n+1}(1)-p).
\end{align*}
Since $3^{L_{n+1}}|3^{L_{n+1}}k_{n+1}'=k_{n+1}$ and $3^{\Lambda_n}|2^{L_{n+1}}-k_{n+1}'$, applying Lemma \ref{Lem4.6} (i) to $P=3^{\Lambda_n}$, $Q=2^{L_{n+1}}$, $R=3^{L_{n+1}}$, and $p$, we obtain a unitary $U_{n+1}^{(0)}$ in $\tkappa_n^{(0)}(1_{B_n^{(0)}}\otimes M_{3^{\Lambda_n}})\epsilon_{n+1}(M_{2^{L_{n+1}}}\otimes M_{3^{L_{n+1}}})+\C 1_{AF}\subset A_{n+1}+\C 1_{AF}$ such that 
\[ \Ad U_{n+1}^{(0)}(\xi(f)(0))\in \tkappa_n^{(0)}(B_n^{(0)}\otimes 1_{M_{3^{\Lambda_n}}})\epsilon_{n+1}(M_{2^{L_{n+1}}}\otimes 1_{M_{3^{L_{n+1}}}}),\]
for any $f\in \cA_n$.
By (6, $n$), now we have $\Ad W_n^{(0)}(b\otimes 1_{M_{3^{\infty}}})=\kappa_n^{(0)}(b\otimes 1_{M_{3^{\Lambda_n}}})$ for any $\in B_n^{(0)}$ and $W_n^{(0)}\in_{\varepsilon_n/4}\Image(\kappa_n^{(0)})+\C1_{AF}$. Since $W_n^{(0)}\in_{\delta_{n+1}^{(8)}}\Image(\barepsilon_{n+1})'\cap \AF^{\sim}$ (from (2)), and $
\barepsilon_{n+1}(x)\approx_{\delta_{n+1}^{(6)}}\epsilon_{n+1}(x)$ for all $x\in (M_{2^{L_{n+1}}}\otimes M_{3^{L_{n+1}}})^1$, we obtain a unitary $\tW_n^{(0)}\in \Image(\epsilon_{n+1})'\cap\AF^{\sim}$ such that $\tW_n^{(0)}\approx_{\delta_{n+1}^{(5)}}W_n^{(0)}$. 
Then it follows that 
\begin{align*}
&\Ad {W_n^{(0)}}^*\left(\tkappa_n^{(0)}(b\otimes 1_{M_{3^{\Lambda_n}}})\epsilon_{n+1}(a\otimes1_{M_{3^{L_{n+1}}}})\right)\\
&\approx_{2\delta_{n+1}^{(5)}}\Ad {W_n^{(0)}}^*\left(\tkappa_n^{(0)}(b\otimes 1_{M_{3^{\Lambda_n}}})\right)\epsilon_{n+1}(a\otimes  1_{M_{3^{L_{n+1}}}})\\
&=(b\otimes 1_{M_{3^{\infty}}})\epsilon_{n+1}(a\otimes 1_{M_{3^{L_{n+1}}}})\approx_{\delta_{n+1}^{(6)}+2\delta_{n+1}^{(9)}}((b\epsilon_{n+1}^{(0)}(a))\otimes 1_{M_{3^{\infty}}})(\epsilon_{n+1}^{(1)}(1)\otimes 1_{M_{2^{\infty}}})\\ 
&\approx_{\delta_{n+1}^{(10)}} b\epsilon_{n+1}^{(0)}(a)\otimes 1_{M_{3^{\infty}}},
\end{align*}
for any $b\in {B_n^{(0)}}^1$ and $a\in M_{2^{L_{n+1}}}^1$. By Lemma \ref{Lem4.6} (iii), we obtain a unitary $u^{(0)}$ in $\AF^{\sim}$ such that $u^{(0)}\approx_{\delta_{n+1}^{(3)}}1_{AF}$ and 
\[\Ad u^{(0)}{W_n^{(0)}}^*\left(\tkappa_n^{(0)}(b\otimes 1_{M_{3^{\Lambda_n}}})\epsilon_{n+1}(a\otimes 1_{M_{3^{L_{n+1}}}})\right)=(b\epsilon_{n+1}^{(0)}(a))\otimes 1_{M_{3^{\infty}}},\]
for any $b\in B_n^{(0)}$ and $a\in M_{2^{L_{n+1}}}$. Consider the unitary $z_n^{(0)}=W_n^{(0)}{u^{(0)}}^*$ in $\AF^{\sim}$, which satisfies $z_n^{(0)}\approx_{\delta_{n+1}^{(3)}}W_n^{(0)}$ and $\Ad z_n^{(0)}(b\epsilon_{n+1}^{(0)}(a)\otimes1_{M_{3^{\infty}}}) =\tkappa_n^{(0)}(b\otimes 1_{M_{3^{\Lambda_n}}})\epsilon_{n+1}(a\otimes 1_{M_{3^{L_{n+1}}}})$ for any $b\in B_n^{(0)}$ and $a\in M_{2^{L_{n+1}}}$.  
Define an embedding $\bariota_{n+1}^{(0)}$ of $A_{n+1}^{(0)}$ into $\AF$ by 
\[ \bariota_{n+1}^{(0)}(a) =\Ad {U_{n+1}^{(0)}}^*z_n^{(0)}(a\otimes 1_{M_{3^{\infty}}})\quad\text{for }a\in A_{n+1}^{(0)}.\]
From $B_n^{(0)}\Image(\epsilon_{n+1}^{(0)})\subset A_{n+1}^{(0)}$, it follows that $\xi(f)(0)\in \Image(\bariota_{n+1}^{(0)})$ for any $f\in \cA_n$. 

In the same way as for $i=0$, we also obtain unitaries $u^{(1)}$, $U_{n+1}^{(1)}$, $z_n^{(1)}$ in $\AF^{\sim}$, and an embedding $\bariota_{n+1}^{(1)}$ of $A_{n+1}^{(1)}$ into $\AF$ such that 
$U_{n+1}^{(1)}\in\tkappa_n^{(1)}(1_{B_n^{(1)}}\otimes M_{2^{\Lambda_n}})\epsilon_{n+1}(M_{2^{L_{n+1}}}\otimes M_{3^{L_{n+1}}})+\C1_{AF}$, $z_n^{(1)}=W_n^{(1)}{u^{(1)}}^*\approx_{\delta_{n+1}^{(3)}} W_n^{(1)}$, $\bariota_{n+1}^{(1)}(a) =\Ad {U_{n+1}^{(1)}}^*z_n^{(1)}(a\otimes 1_{M_{2^{\infty}}})$ for $a\in A_{n+1}^{(1)}$, and $\xi(f)(1)\in \Image(\bariota_{n+1}^{(1)})$ for any $f\in \cA_n$.

Note that, since $(A_{n+1}^{(i)}\otimes 1_{M_{(3-i)^{L_{n+1}}}})^1\subset_{\delta_{n+1}^{(8)}} A_{n+1}$ for both $i=0, 1$, and since $z_n^{(i)}\approx_{\delta_{n+1}^{(3)}}W_n^{(i)}\in_{\delta_{n+1}^{(8)}} A_{n+1} +\C 1_{AF}$ (from (2)), we have  
\[\Image(\bariota_{n+1}^{(i)})^1=\Ad {U_{n+1}^{(i)}}^*z_n^{(i)}((A_{n+1}^{(i)})^1\otimes 1_{M_{(3-i)^{L_{n+1}}}})\subset_{\delta_{n+1}^{(2)}}A_{n+1}.\]
Although $\dim(A_{n+1}^{(i)})$ can be bigger than $N$, the following argument allows us to obtain unitaries $\tU_{n+1}^{(i)}\in \AF^{\sim}$, $i=0,1$, such that $\tU_{n+1}^{(i)}\approx_{\delta_{n+1}^{(2)}} 1_{AF}$ and $\Ad \tU_{n+1}^{(i)}\circ\bariota_{n+1}^{(i)}(A_{n+1}^{(i)})\subset A_{n+1}$. In the case $i=1-i_{n+1}$, we fix a unitary $\barU_{n+1}^{(i)}=$ $U_{n+1}u^{(i)} {W_n^{(i)}}^*$ $U_{n+1}^{(i)}\in \AF^{\sim}$. Since $W_n^{(i)}\in_{\delta_{n+1}^{(8)}} A_{n+1}+\C 1_{AF}$ and $U_{n+1}^{(i)}\in A_{n+1}+\C 1_{AF}$, there exists a unitary $\bary^{(i)}\in A_{n+1} +\C 1_{AF}$ such that $\bary^{(i)}\approx_{\delta_{n+1}^{(2)}}\barU_{n+1}^{(i)}$. Consider the unitary $\tU_{n+1}^{(i)}={\bary^{(i)}}^* \barU_{n+1}^{(i)}$, which satisfies $\tU_{n+1}^{(i)}\approx_{\delta_{n+1}^{(2)}} 1_{AF}$ and $\Ad\tU_{n+1}^{(i)}\circ \bariota_{n+1}^{(i)}(a)= \Ad {\bary^{(i)}}^*U_{n+1}(a\otimes 1_{M_{(3-i)^{\infty}}})\in A_{n+1}$ for any $a\in A_{n+1}^{(i)}$. 
In the case  $i=i_{n+1}$, from (1) it follows that $(A_{n+1}^{(i)}\otimes M_{(3-i)^{L_{n+1}}})^1\subset_{\tdelta} A_{n+1}^{(1-i)}\otimes M_{(2+i)^{\tL_{n+1}}}$ and $\tdelta>0$ is chosen for $\dim(A_{n+1}^{(i)})6^{2L_{n+1}}$. By Lemma \ref{Lem4.6} (ii), there exists a unitary $V\in\AF^{\sim}$ such that $V\approx_{\delta_{n+1}^{(10)}} 1_{AF}$ and $\Ad V(A_{n+1}^{(i)}\otimes M_{(3-i)^{L_{n+1}}})\subset A_{n+1}^{(1-i)}\otimes M_{(2+i)^{\tL_{n+1}}}$.
It follows that $\Ad U_{n+1} V(A_{n+1}^{(i)}\otimes M_{(3-i)^{L_{n+1}}})\subset A_{n+1}$. Set $\barU_{n+1}^{(i)}=U_{n+1}Vu^{(i)} {W_n^{(i)}}^* U_{n+1}^{(i)}$ which is a unitary in $\AF^{\sim}$ such that $\barU_{n+1}^{(i)}\in_{2\delta_{n+1}^{(3)}}A_{n+1}+\C 1_{AF}$. Then there exists a unitary $\bary^{(i)}\in A_{n+1} +\C 1_{AF}$ such that $\bary^{(i)}\approx_{\delta_{n+1}^{(2)}}\barU_{n+1}^{(i)}$. Setting $\tU_{n+1}^{(i)} = {\bary^{(i)}}^*\barU_{n+1}^{(i)}$, we also see that $\Ad\tU_{n+1}^{(i)}\circ \bariota_{n+1}^{(i)}(a)= \Ad {\bary^{(i)}}^*U_{n+1}V(a\otimes 1_{M_{(3-i)^{\infty}}})\in A_{n+1}$ for any $a\in A_{n+1}^{(i)}$.

Define embeddings $\iota_{n+1}^{(i)}$ of $A_{n+1}^{(i)}$ into $A_{n+1}$, $i=0,1$, by $\iota_{n+1}^{(i)}=\Ad \tU_{n+1}^{(i)}\circ\bariota_{n+1}^{(i)}$. Then it follows that $\Ad\tU_{n+1}^{(i)}(\xi(f)(i))\in \Image(\iota_{n+1}^{(i)})$ for any $i=0, 1$ and $f\in\cA_{n}$. 
Define unitaries $V_{n+1}^{(i)}=\tU_{n+1}^{(i)}{U_{n+1}^{(i)}}^*z_n^{(i)}$, $i=0, 1$, in $\AF^{\sim}$. Since $U_{n+1}^{(i)}\in A_{n+1} + \C1_{AF}$ for both $i=0, 1$, we see that $V_{n+1}^{(i)}$, $i=0, 1$, satisfy the conditions of (6, $n+1$). 
We define a \Css{} $\cA_{n+1}$ of $C([0, 1])\otimes A_{n+1}$ as required for (5, $n+1$). Define 
finite dimensional \Css s $E^{(i)}$, $i=0, 1$, of $A_{n+1}$ by
$E^{(i)}=\{\xi(f)(i)\ : \ f\in \cA_n\}$. Note that $\dim(E^{(i)})\leq \dim(A_n^{(i)})+\dim(B_n^{(i)})(3-i)^{2\Lambda_n}< N$ for both $i=0, 1$ and that $\xi(f) (i) \approx_{2\delta_{n+1}^{(2)}}\Ad \tU_{n+1}^{(i)}(\xi(f)(i))$ for any $f\in \cA_{n}^1$. By Lemma \ref{Lem4.6} (iii) with (3, $n+1$), we obtain unitaries 
$\cU^{(i)}$ in $ A_{n+1} +\C 1_{AF}$, $i=0, 1$, such that $\cU^{(i)}\approx_{\delta_{n+1}^{(0)}} 1_{AF}$ and $\Ad\cU^{(i)}(\xi (f)(i))=\Ad \tU_{n+1}^{(i)}(\xi (f)(i))$ for any $i=0, 1$ and $f\in \cA_n$.
 Thus we obtain a unitary $\tcU\in C([0, 1])\otimes (A_{n+1}+\C 1_{AF})$ which is a Lipschitz continuous function such that $\tcU(i)=\cU^{(i)}$ for $i=0, 1$, $\tcU(t)\approx_{\varepsilon_{n+1}} 1_{AF}$ for all $t\in[0, 1]$ and $\Lip (\tcU)\leq \varepsilon_{n+1}$. 
We define an embedding $\varphi_{n+1}$ of $\cA_{n}$ into $\cA_{n+1}$ by $\varphi_{n+1}(f)=\Ad \tcU\circ\xi(f)$ for $f\in \cA_n$. Because of the construction of $\xi$ and the inequality $\Lip(\tcU)\leq \varepsilon_{n+1}$, it follows that $\varphi_{n+1}$ satisfies (8, $n+1$). 

In the rest of the induction, we construct $\Lambda_{n+1}$, $B_{n+1}^{(i)}$ and $\kappa_{n+1}^{(i)}$ satisfying the conditions (1, $n+1$), (2, $n+1$), (6, $n+1$), (7, $n+1$). Set $\bardelta_{n+1}^{(0)}=\varepsilon_{n+1}/8$. By applying Lemma \ref{Lem4.6} (ii), (iii), (iv) to $\bardelta_{n+1}^{(k)}$\ $ ( =\varepsilon)$ and $\barN=\dim(A_{n+1})$ inductively, there exists $\bardelta_{n+1}^{(k+1)}$\   $( =\delta)>0$ satisfying the conditions of (ii), (iii), and (iv). Taking smaller $\bardelta_{n+1}^{(k)}$, $k\in\N$, as above, we may assume that $\bardelta_{n+1}^{(k+1)}<\bardelta_{n+1}^{(k)}$ for $k\in\N$ and 
 $N\sum_{j=1}^{\infty} 8^j \bardelta_{n+1}^{(k+j)}< \bardelta_{n+1}^{(k)}$. 
Since $\AF=\AF_i\otimes M_{(3-i)^{\infty}}$ for both $i=0, 1$, there exist $\Lambda_{n+1}\in \N$ with the condition of (1, $n+1$) and finite dimensional \Css s $B_{n+1}^{(i)}\subset \AF_i$, $i=0, 1$, such that $A_{n+1}^{(i)}\subset B_{n+1}^{(i)}$, $A_{n+1}^1\subset_{\bardelta_{n+1}^{(4)}} B_{n+1}^{(i)}\otimes M_{(3-i)^{\Lambda_{n+1}}}$, and $\{V_{n+1}^{(0)}, V_{n+1}^{(1)}\}\subset_{\bardelta_{n+1}^{(4)}} B_{n+1}^{(i)}\otimes M_{(3-i)^{\Lambda_{n+1}}} +\C 1_{AF}$.  

By Lemma \ref{Lem4.6} (ii), there exist unitaries $\barW_{n+1}^{(i)}$, $i=0, 1$, in $\AF^{\sim}$ such that $\barW_{n+1}^{(i)}\approx_{\bardelta_{n+1}^{(3)}} 1_{AF}$ and $\Ad\barW_{n+1}^{(i)}(B_{n+1}^{(i)}\otimes M_{(3-i)^{\Lambda_{n+1}}})\supset A_{n+1}$ for $i=0, 1$.  Define embeddings $\hatkappa_{n+1}^{(i)} : B_{n+1}^{(i)}\otimes M_{(3-i)^{\Lambda_{n+1}}}\rightarrow \AF$, $i=0, 1$, by $\hatkappa_{n+1}^{(i)}=\Ad\barW_{n+1}^{(i)}$. By Lemma 4.6 (v) and $V_{n+1}^{(i)}\in_{\bardelta_{n+1}^{(2)}} \Image(\hatkappa_{n+1}^{(i)})+\C 1_{AF}$, there exist unitaries $\hatW_{n+1}^{(i)}\in \Image(\hatkappa_{n+1}^{(i)}) +\C 1_{AF}$ such that  $\Ad\hatW_{n+1}^{(i)} V_{n+1}^{(i)}(a\otimes 1_{M_{(3-i)^{\infty}}})=\Ad \barW_{n+1}^{(i)}(a\otimes 1_{M_{(3-i)^{\infty}}})$ for any $a\in A_{n+1}^{(i)}$. Define unitaries $W_{n+1}^{(i)}$, $i=0, 1$, in $\AF^{\sim}$ and embeddings $\kappa_{n+1}^{(i)} : B_{n+1}^{(i)}\otimes M_{(3-i)^{\Lambda_{n+1}}}\rightarrow \AF$ by 
\[W_{n+1}^{(i)}=(\hatW_{n+1}^{(i)})^* \barW_{n+1}^{(i)},\quad \kappa_{n+1}^{(i)}(b\otimes a) =\Ad W^{(i)}(b\otimes a),\]
for $b\in B_{n+1}^{(i)}$ and $a\in M_{(3-i)^{\Lambda_{n+1}}}\subset M_{(3-i)^{\infty}}$. These choices satisfy  the conditions of (6, $n+1$) and (7, $n+1$). 

Define $A_G$ as the inductive limit \Cs{} $\displaystyle\lim_{\longrightarrow}(\cA_n, \varphi_{n+1})$. In the proof of Proposition \ref{Prop4.5} (iii), we have seen that $(G, G^+)\cong (\Image(\iota_{\infty *}^{(0)})\cap\Image(\iota_{\infty *}^{(1)}), \Image(\iota_{\infty *}^{(0)})\cap\Image(\iota_{\infty *}^{(1)})\cap K_0(\AF)^+)$, as ordered abelian groups. To show that the right hand side is isomorphic to $(K_0(A_G), K_0(A_G)^+)$, we need the following observation. Let $\eta_n$ denote the canonical embedding of $A_n$ into $A_{n+1}$ which is obtained from (3, $n+1$) for $n\in\N$.
Write $\ev_n^{(i)} : \cA_n\rightarrow\Image(\iota_n^{(i)})$, $i=0, 1$, to denote the evaluation maps and $\bareta_n^{(i)} : \Image(\iota_n^{(i)})\rightarrow A_n$ for the canonical embeddings.

\begin{observation}\label{Obs4.8}
For $i=0, 1$ and $n\in\N$, the two $*$-homomorphisms $\ev_{n+1}^{(i)}\circ\varphi_{n+1}$ and $\Ad \tcU(i)\circ\eta_n\circ\bareta_n^{(i)}\circ\ev_n^{(i)}$ $:\cA_n \rightarrow \Image(\iota_{n+1}^{(i)})$ are homotopic.
\end{observation}
\begin{proof}
Let $\Phi_t^{(i)}$, $t\in[0, 1]$, $i=0, 1$, be pointwise continuous paths of automorphisms on $\cA_n$ such that $\Phi_0^{(0)}=\Phi_0^{(1)}=\id_{\cA_n}$, $\Phi_1^{(0)}(f) (t)=f(0)$, $t\in [0, 1/2]$, $\Phi_1^{(0)}(f)(t)=f(2t-1)$, $t\in [1/2, 1]$, $\Phi_1^{(1)}(f)(t)=f(2t)$, $t\in [0, 1/2]$,  and $\Phi_1^{(1)}(f)(t)=f(1)$, $t\in[1/2, 1]$, for any $f\in\cA_n$. Then it follows that $\ev_{n+1}^{(i)}\circ\varphi_{n+1}$ and $\ev_{n+1}^{(i)}\circ\varphi_{n+1}\circ\Phi_1^{(i)}$ are homotopic for both $i=0, 1$. From the construction of $\varphi_{n+1}=\Ad\tcU\circ \xi$ and $\Phi_1^{(i)}(f)(1/2)=f(i)$, we have 
\[\ev_{n+1}^{(i)}\circ\varphi_{n+1}\circ\Phi_1^{(i)}(f)=\Ad \tcU(i)( f(i)\epsilon_{n+1}(1_{M_{m_{n+1}}}))=\Ad \tcU (i)\circ\eta_n\circ \bareta_n^{(i)}\circ\ev_n^{(i)}(f),\]
for any $i=0, 1$ and $f\in\cA_n$.
\end{proof}

Denote by $\eta_{\infty n}$ the canonical embedding of $A_n$ into $\AF$. For $i= 0, 1$,  denote by $\Ev_n^{(i)} : \cA_n\rightarrow A_n$ the evaluation map at $i$. Note that  $\bareta_n^{(i)}\circ\ev_{n}^{(i)}=\Ev_n^{(i)}$.  Let $\eta_n^{(i)} :A_n^{(i)}\rightarrow A_{n+1}^{(i)}$ denote the canonical embeddings of (2, $n+1$) and $\eta_{\infty n}^{(i)} : A_n^{(i)}\rightarrow \AF_i$ the canonical embeddings. For $m$, $n\in\N$ with $m>n$,  set $\varphi_{m, n}=\varphi_m\circ\varphi_{m-1}\circ\cdots \circ\varphi_{n+1}: \cA_n\rightarrow \cA_m$, $\eta_{m, n}=\eta_{m-1}\circ\eta_{m-2}\circ\cdots\circ\eta_n : A_n\rightarrow A_m$, and $\eta_{m, n}^{(i)}=\eta_{m-1}^{(i)}\circ\eta_{m-2}^{(i)}\circ\cdots \circ\eta_n^{(i)} : A_n^{(i)}\rightarrow A_m^{(i)}$. 
 By Observation \ref{Obs4.8}, $\Ev_{n+1}^{(i)}\circ\varphi_{n+1}$ and $\Ad \tcU(i)\circ\eta_n\circ\Ev_n^{(i)}$ are homotopic as $*$-homomorphisms from $\cA_n$ into $A_{n+1}$.
Since $\tcU(i)\in A_{n+1}+\C 1_{AF}$, it follows that $\Ev_{n+1 *}^{(i)}\circ \varphi_{n+1 *}=\eta_{n*}\circ\Ev_{n*}^{(i)}$. From (6, $n$) it follows that 
\[\iota_{n+1}^{(i)}\circ\eta_n^{(i)}= \Ad V_{n+1}^{(i)} {V_n^{(i)}}^*\circ\iota_n^{(i)}, \text{ 
and } V_{n+1}^{(i)}{V_n^{(i)}}^*\in_{2\varepsilon_n} A_{n+1}+\C1_{AF}.\] 
Then we have  $\Ad V_{n+1}^{(i)}{V_n^{(i)}}^*|_{\Image(\iota_n^{(i)}) *}=\eta_n|_{\Image(\iota_n^{(i)}) *}$ as group homomorphisms from $K_0(\Image(\iota_n^{(i)}))$ to $K_0(A_{n+1})$, which implies that $\iota_{n+1 *}^{(i)}\circ\eta_{n *}^{(i)}=\eta_{n *}\circ\iota_{n*}^{(i)}$. It follows that 
\[\eta_{n*}(\Image(\iota_{n*}^{(0)})\cap \Image(\iota_{n*}^{(1)}))\subset\Image(\iota_{n+1*}^{(0)})\cap\Image(\iota_{n+1*}^{(1)}).\]
 Denote by $(I, I^+)$ the inductive limit of ordered groups $\displaystyle\lim_{\longrightarrow}(\Image(\iota_{n*}^{(0)})\cap\Image(\iota_{n*}^{(1)}), \eta_{n*})$, where the order in  $\Image(\iota_{n*}^{(0)})\cap\Image(\iota_{n*}^{(1)})$ is determined by $\Image(\iota_{n*}^{(0)})\cap\Image(\iota_{n*}^{(1)})\cap K_0(A_n)^+$. By (6, $n$), we also see that $\eta_{\infty n *}\circ\iota_{n*}^{(i)}=\iota_{\infty *}^{(i)}\circ\eta_{\infty n *}^{(i)}$. 
 Since $K_0(\AF)=\bigcup_{n=1}^{\infty}\eta_{\infty n *}(K_0(A_n))$ (from $F_{A n}\subset_{\varepsilon_n} A_n^1$ in (4, $n$)), we have 
 \[ I\cong \bigcup_{n=1}^{\infty}\eta_{\infty, n *}(\Image (\iota_{n *}^{(0)})\cap \Image(\iota_{n *}^{(1)})) = \bigcup_{n=1}^{\infty} \Image(\iota_{\infty *}^{(0)}\circ \eta_{\infty n *}^{(0)})\cap \Image(\iota_{\infty *}^{(1)}\circ \eta_{\infty n *}^{(1)}).\]
 As $K_0(\AF_i)=\bigcup_{n=1}^{\infty} \eta_{\infty n *}^{(i)}(K_0(A_n^{(i)}))$ (since $F_{A_i n}\subset A_n^{(i)}$ by (2, $n$)), it follows that $\bigcup_{n=1}^{\infty}\Image(\iota_{\infty *}^{(0)}\circ\eta_{n *}^{(0)})\cap \Image(\iota_{\infty *}^{(1)}\circ\eta_{n *}^{(1)}) =\Image (\iota_{\infty *}^{(0)})\cap\Image(\iota_{\infty *}^{(1)})$. For the same reason, it is straightforward to check that $I^+$ corresponds to $\Image(\iota_{\infty *}^{(0)})\cap\Image(\iota_{\infty *}^{(1)})\cap K_0(\AF)^+$.
  Then it follows that $(I, I^+)$ is isomorphic to  $(\Image(\iota_{\infty *}^{(0)})\cap\Image(\iota_{\infty *}^{(1)}), \Image(\iota_{\infty *}^{(0)})\cap\Image(\iota_{\infty *}^{(1)})\cap K_0(\AF)^+)$. 
 
  Since $\Ev_{n+1 *}^{(0)}\circ\varphi_{n+1*} =\eta_{n*}\circ\Ev_{n *}^{(0)}$ and $\Image(\Ev_{n *}^{(0)})\subset\Image(\iota_{n, *}^{(0)})\cap\Image(\iota_{n, *}^{(1)})$, we obtain a positive group homomorphism $\Ev_*^{(0)}$ from $\displaystyle K_0(A_G)=\lim_{\longrightarrow}(K_0(\cA_n), \varphi_{n+1 *})$ into $(I, I^+)$. We may regard $\Ev_*^{(0)}$ as a positive group homomorphism from $K_0(A_G)$ to $\Image(\iota_{\infty *}^{(0)})\cap\Image(\iota_{\infty *}^{(1)})$. 
  To show the injectivity of $\Ev_*^{(0)}$,  let $x\in \varphi_{\infty n*}(K_0(\cA_n))\subset K_0(A_G)$ be such that $\Ev_*^{(0)}(x)=0$ and $\barx\in K_0(\cA_n)$ be such that $x=\varphi_{\infty n *}(\barx)$. Since $0=\Ev_*^{(0)}(x)=\eta_{\infty n *}\circ\Ev_{n *}^{(0)}(\barx)=\eta_{\infty n *}\circ\Ev_{n *}^{(1)}(\barx)$, there exists $m> n+1$ such that $\eta_{m, n *}\circ\Ev_{n*}^{(0)}(\barx)=\eta_{m, n*}\circ \Ev_{n *}^{(1)}(\barx)=0$. Then it follows that $\Ev_{m *}^{(0)}\circ\varphi_{m, n *}(\barx)=\Ev_{m *}^{(1)}\circ\varphi_{m, n *}(\barx)=0$. 
 Fix $i=1-i_m$. From the construction of $\iota_m^{(i)}$, there exists a unitary $\bary_m^{(i)}\in A_m +\C 1_{AF}$ such that $\iota_m^{(i)}(a)=\Ad \bary_m^{(i)}U_m(a\otimes 1_{(3-i)^{\infty}})$ for $a\in A_m^{(i)}$. From (3, $m$) for $i=1-i_m$, it follows that 
 $\bareta_{m*}^{(i)} : K_0(\Image(\iota_m^{(i)}))\rightarrow K_0(A_m)$ is injective. Then we have $\ev_{m*}^{(i)}\circ\varphi_{m n *}(\barx)=0$. By Observation \ref{Obs4.8}, we see that $\ev_{m+1 *}^{(i)}\circ\varphi_{m+1, n*}(\barx)=0$. By the same argument as for $1-i_{m+1}$ $(=i_m)$ in (3, $m+1$), we also have $\ev_{m+1 *}^{(1-i)}\circ\varphi_{m+1, n *}(\barx)=0$. Because $\ev_{ m+1*}^{(0)}\oplus \ev_{m+1 *}^{(1)} : K_0(\cA_{m+1})\rightarrow K_0(\Image(\iota_{m+1}^{(0)}))\oplus K_0(\Image(\iota_{m+1}^{(1)}))$ is injective, it follows that $\varphi_{m+1, n *}(\barx)=0$, which implies that $x=\varphi_{\infty, n *}(\barx)=0$. 
 
 Let $y\in \Image(\iota_{\infty *}^{(0)})\cap \Image(\iota_{\infty *}^{(1)})\cap K_0(\AF)^+\cong G^+$. To show the surjectivity of $\Ev_*^{(0)}$, since $(\Image(\iota_{\infty *}^{(0)})\cap \Image(\iota_{\infty *}^{(1)}), \Image(\iota_{\infty *}^{(0)})\cap \Image(\iota_{\infty *}^{(1)})\cap K_0(\AF)^+)$ is an ordered abelian group, it suffices to show that $y\in \Ev_*^{(0)}(K_0(A_G)^+)$.
 Since $G^+=\bigcup_{n\in\N}F_n$, by (2, $n$), there exists $m\in\N$ such that $y\in \iota_{\infty *}^{(i)}\circ\eta_{\infty m *}^{(i)}(K_0(A_m^{(i)})^+)$ for both $i=0, 1$. 
 Thus we have projections $p_m^{(i)}$ in $A_m^{(i)}\otimes M_N$, $i=0, 1$, for some $N\in \N$ such that $y=\iota_{\infty *}^{(i)}\circ \eta_{\infty m *}^{(i)}([p_m^{(i)}]_0)$. 
 Since $y= \eta_{\infty, m *}\circ \iota_{m *}^{(i)}([p_m^{(i)}]_0)$ for $i=0, 1$, there exists $l>m$ such that $\eta_{l, m *}\circ\iota_{m *}^{(0)}([p_m^{(0)}]_0)=\eta_{l, m *}\circ\iota_{m *}^{(1)}([p_m^{(1)}]_0)$. Choose projections $p_l^{(i)}=\eta_{l, m}^{(i)}\otimes\id_{M_N}(p_m^{(i)})\in A_l^{(i)}\otimes M_N$, $i=0, 1$, which satisfy $\iota_{l *}^{(0)}([p_l^{(0)}]_0)=\iota_{l *}^{(1)}([p_l^{(1)}]_0)$ in $K_0(A_l)$. Then there exists a projection $\tp$ in $C([0, 1])\otimes A_l\otimes M_N$ such that $\tp(i)=\iota_l^{(i)}\otimes\id_{M_N}(p_l^{(i)})$ for $i=0, 1$. Regarding $\tp$ as a projection in $\cA_l\otimes M_N$, we have $\Ev_{l *}^{(0)}([\tp]_0)=\iota_{l *}^{(0)}([p_l^{(0)}]_0)$ 
 and 
 \[ y=\eta_{\infty l *}\circ\iota_{l *}^{(0)}([p_l^{(0)}]_0)=\eta_{\infty l *}\circ \Ev_{l *}^{(0)}([\tp]_0)=\Ev_*^{(0)}(\varphi_{\infty l *}([\tp]_0))\in \Ev_*^{(0)}(K_0(A_G)^+).\]
 Thus, $\Ev_*^{(0)}$ is surjective and ${\Ev_*^{(0)}}^{-1}$ is also a positive group homomorphism.

In the rest of the proof, we show that $A_G$ is a rationally AF algebra. We only show that $A_G\otimes M_{2^{\infty}}$ is approximately finite dimensional, because by replacing even numbers by odd numbers the same argument allows us to see that $A_G\otimes M_{3^{\infty}}$ is approximately finite dimensional. Since any separable local AF-algebra is exactly approximately finite dimensional (\cite[Theorem 2.2]{Br}), it suffices to show that for a given finite subset $F$ of $\cA_n^1$ and $\varepsilon>0$, there exist $N\in\N$ and a finite dimensional \Css{} $E$ of $\cA_{n+N}\otimes M_{2^{\infty}}$ such that $\varphi_{n+N, n}(F)\otimes 1_{M_{2^{\infty}}}\subset_{\varepsilon} E$.  Because the set of Lipschitz continuous functions is dense in $\cA_n^1$, we may assume that $F$ consists of Lipschitz continuous functions. Set $L=\max_{f\in F} \Lip(f)$, and let $m\in\N$ be such that $\frac{L+2}{2^m} + 4\sum_{j=m+1}^{\infty} \varepsilon_j < \varepsilon$. Because of (8, $n$) and the inequality $\sum_{n\in \N}\varepsilon_n <1$, for any $f\in F$, $s, t\in [0, 1]$, and $l >m $, it follows that 
\begin{align*}
\|\varphi_{n+m+l, n}(f)(s)- \varphi_{n+m+l, n}(f)(t)\|&\leq\frac{\Lip(f)}{2^{l+m}}+2\left(\sum_{j=n+1}^{n+m+l}\frac{\varepsilon_j}{2^{n+m+l-j}}\right)\\
&<\frac{L+2}{2^{m+1}} + 2\sum_{j=m+1}^{\infty}\varepsilon_j < \frac{\varepsilon}{2}.
\end{align*}
Choose $l\in\N$ such that, in addition to $l > m$, the number $n+m+l$ is even  (odd  for the case  $A_G\otimes M_{3^{\infty}}$). Set $N=m+l$, $k=n+N$, and 
\[\barE=1_{C([0, 1])}\otimes\iota_k^{(0)}(A_k^{(0)})\subset C([0, 1])\otimes A_k.\]
From the 
%above calculation, 
calculation above,
it follows that $\varphi_{k, n}(F)\subset_{\varepsilon/2} \barE$. 
Let us now modify $\barE$ to a finite dimensional \Cs{} $E$ in $\cA_k\otimes M_{2^{\infty}}$. By (3, $k$) with $i_k=0$, we obtain a unitary $U_k$ in $\AF^{\sim}$ such that $U_k\approx_{\varepsilon_k} 1_{AF}$ and $\Ad U_k(A_k^{(1)}\otimes M_{2^{\tL_k}})=A_k$. 
By (6, $k$) there exists a unitary $V_k^{(1)}\in \AF^{\sim}$ such that $\iota_k^{(1)}(a)=\Ad V_k^{(1)}(a\otimes 1_{M_{2^{\infty}}})\in A_k$ for any $a\in A_k^{(1)}$, and $V_k^{(1)}\in_{\varepsilon_k}A_k+\C1_{AF}$. Applying Lemma \ref{Lem4.6} (v), we obtain a unitary $W\in A_k +\C 1_{AF}$ such that $\Ad WU_k(a\otimes 1_{M_{2^{\infty}}})=\Ad V_k^{(1)}(a\otimes 1)$ for $a\in A_k^{(1)}$. Define an isomorphism $\tepsilon : M_{2^{\tL_k}}\rightarrow A_k\cap \Image(\iota_k^{(1)})'$ by $\tepsilon(a)=\Ad WU_k(1_{A_k^{(1)}}\otimes a)$ for $a\in M_{2^{\tL_k}}$, and note that $A_k=\Image(\iota_k^{(1)})\Image(\tepsilon)$. 
Denote by $s$ the self-adjoint unitary in $\Image(\tepsilon)\otimes M_{2^{\tL_k}}$ such that $\Ad s(\tepsilon(x)\otimes y) =\tepsilon(y)\otimes x$ for any $x, y\in M_{2^{\tL_k}}$. Choose a path of unitaries $\ts\in C([0, 1])\otimes A_k\otimes M_{2^{\tL_k}}$ such that $\ts(t)\in \Image(\tepsilon)\otimes M_{2^{\tL_k}}$ for all $t\in [0, 1]$, $\ts(0)=1$, and $\ts(1)=s$. Define a finite dimensional \Cs{} $E$ by 
\[E=\Ad\ts(\barE\otimes 1_{M_{2^{\tL_k}}})\subset C([0, 1])\otimes A_k\otimes M_{2^{\tL_k}}.\]
We can regard $E$ as a finite dimensional \Css{} of $\cA_k\otimes M_{2^{\tL_k}}$. 
Indeed, for $e\in E$, there exists $a_e\in \Image(\iota_k^{(0)})$ such that $e(0)=a_e\otimes 1_{M_{2^{\tL_k}}}$. 
Since $a_e\in A_k=\Image(\iota_k^{(1)})\Image(\tepsilon)$, there are $x_l\in A_k^{(1)}$, $y_l\in M_{2^{\tL_k}}$, $l=1,2,..., L$, such that $a_e=\sum\limits_{l=1}^L\iota_k^{(1)}(x_l)\tepsilon(y_l)$. Then it follows that 
\[e(1)=\Ad s(a_e\otimes 1_{M_{2^{\tL_k}}})=\sum_{l=1}^L\iota_k^{(1)}(x_l)\otimes y_l\in \Image(\iota_k^{(1)})\otimes M_{2^{\tL_k}}.\]
We conclude that for $f\in F$ there exists $x_f\in (A_k^{(1)})^1$ and $y_f\in (A_k^{(0)})^1$ such that
\begin{align*}
\varphi_{k, n}(f)\otimes 1_{M_{2^{\tL_k}}}&\approx_{\varepsilon /2}\left(1_{C([0, 1])}\otimes \iota_k^{(1)}(x_f)\right)\otimes 1_{M_{2^{\tL_k}}}\\
&=\Ad \ts(1_{C([0, 1])}\otimes \iota_k^{(1)}(x_f)\otimes 1_{M_{2^{\tL_k}}})\\
&\approx_{\varepsilon/2} \Ad \ts (1_{C([0, 1])}\otimes \iota_k^{(0)}(y_f)\otimes 1_{M_{2^{\tL_k}}})\in E.
\end{align*}

In order to adjoin the property of $\cZ$-absorption to $A_G$, we only need to consider $A_G\otimes\cZ$ instead, which is also a RAF-algebra (as a UHF algebra is $\cZ$-absorbing). Indeed, the ordered $K_0$-group $(K_0(A_G\otimes \cZ), K_0(A_G\otimes\cZ)^+)$ is isomorphic to $(G, G^+)$, because, in the proof of Proposition \ref{Prop4.3} (ii), for an RAF-algebra $A$ and the  map $\iota_* : K_0(A)\rightarrow K_0(A\otimes \cZ)$ induced by the canonical embedding $\iota$ of $A$ into $A\otimes\cZ$, we have seen that for $g\in K_0(A)$,  $\iota_*(g)\in K_0(A\otimes\cZ)^+$ if and only if $ng\in K_0(A)^+$ for some $n\in\N$. 
Applying this fact to $A_G$, since $(K_0(A_G), K_0(A_G)^+)\cong (G, G^+)$ is an unperforated ordered abelian group, we see that the induced map from $K_0(A_G)$ to $K_0(A_G\otimes\cZ)$ is an isomorphism of ordered abelian groups. 

 Finally, to make $A_G\otimes\cZ$ stable, we just need to replace it by the tensor product  $A_G\otimes\cZ\otimes \cK$ by the algebra $\cK$ of compact operators on a separable infinite dimensional Hilbert space. 
 \end{proof}

Note that the invariant considered in Theorem \ref{Thm4.4}, the ordered $K_0$-group, is shown in Theorem \ref{Thm5.3} below to be complete, for stable, separable, $\cZ$-absorbing RAF-algebras. Theorem \ref{Thm4.4} is therefore a range of the invariant theorem for the class in question. 

As a consequence of Theorem \ref{Thm4.4}, together with its proof, and also the isomorphism result Theorem \ref{Thm5.3}, we can determine the range of a generalization of the invariant of \cite{Ell01}
 for non-stable AF-algebras---referred to in \cite{Ell01} as the \emph{dimension range}. In Corollary \ref{Cor5.4}, below, we shall show (using Theorem \ref{Thm5.3}) that this analogue of the non-stable AF invariant, 
 what might now be called the matrix dimension range, is complete---for not necessarily stable, separable, $\cZ$-absorbing RAF-algebras. 
 
 Recall that in \cite{Ell01} the invariant for general (separable) AF-algebras, the range of the Murray-von Neumann dimension---the local abelian semigroup of equivalence classes of projections in the algebra, or, equivalentlly, the subset of the $K_0$-group consisting of those equivalence classes---was characterized as an upward directed, hereditary, generating subset of the positive cone of the ordered $K_0$-group---and, more abstractly, as an arbitrary such subset of a dimension group (by \cite{EHS}, an unperforated countable ordered abelian group with the interpolation property of \cite{Bir}, equivalent to the decomposition property of \cite{Rie}).

 To extend this invariant  to RAF-algebras, since there are fewer projections, we must keep track of the increasing sequence of dimension ranges of matrix algebras over the the algebra; let us consider these as subsets of the ordered $K_0$-group. Of course, in the stable case, each of these will be the whole positive cone of the $K_0$-group. Let us call this structure the \emph{matrix dimension range}. 

\begin{corollary}\label{Cor4.8}
The matrix dimension range of a $\cZ$-absorbing separable RAF-algebra can be described in terms of the order-unit $K_0$-group of the algebra with unit adjoined, tensored with $\cZ$, as follows. The $\mathit{n}$-th level of the matrix dimension range, $n\in\N,$ is the set of positive elements of the $K_0$-group of the given algebra which, with respect to the embedding of this in the $K_0$-group of the algebra with unit adjoined, are majorized by $n$ times the class of the unit. (A $\cZ$-absorbing RAF-algebra has cancellation, so this is the same as comparison in the algebra.) The matrix dimension range in fact determines the larger ordered abelian group. Hence for any countable rational dimension group with specified order unit for which there exists a positive map onto $\Z$ taking the order unit into $1\in\Z$, there exists a $\cZ$-absorbing separable RAF-algebra with ordered $K_0$-group the kernel of the map onto $\Z$, and with matrix dimension range as described above with respect to the specified order unit.
\end{corollary}
\begin{proof}
First, let us check that a $\cZ$-absorbing separable RAF-algebra $A$ has cancellation of projections---the Murray-von Neumann semigroup of $A$ is mapped 
injectively into $K_0(A)$. It is enough to consider the stable case.  By Proposition \ref{Prop4.3} (ii), $K_0(A)$ is a rational dimension group. By Theorem \ref{Thm4.4} together with its proof there exists a stable, $\cZ$-absorbing, RAF-algebra $B$ with $K_0(B)$  isomorphic to $K_0(A)$ as an ordered abelian group,  and such that $B$ is the inductive limit of a sequence of point-line algebras (called one-dimensional non-commutative CW complexes in \cite{ELP}), each tensored with $\cZ$. It is straightforward to show (using that the Murray-von Neumann
semigroup of $\cZ$ is the same as that of the complex numbers---recall that $\cZ$ has stable rank one and therefore cancellation)
that the Murray-von Neumann semigroup of a point-line algebra after tensoring with $\cZ$ is the same as before this operation,
namely, a certain subsemigroup of the finite direct sum of
copies of the (cancellative) semigroup of natural numbers
(including zero), indexed by the points at infinity in the
spectrum---see the discussion of point-line algebras in \cite{Ell}.
Thus, $B$ has cancellation.
By Theorem \ref{Thm5.3}, below, $A$ is isomorphic to $B$.

Now, given a countable rational dimension group $G$, by Theorem \ref{Thm4.4} (now just the statement), there exists a stable, separable, $\cZ$-absorbing RAF-algebra $B$ with $K_0(B)$ isomorphic to $G$ as an ordered abelian group. In particular, consider the case that, as an ordered abelian group, $G=H+ H'$ where $H$ is an order ideal and $H'$ is isomorphic in the relative order to the ordered group $\Z$. In this case, since (see proof of Corollary \ref{Cor5.4}, below) $B$ has the ideal property (closed two-sided ideals generated as such by projections), and also (as shown above) has cancellation (so that the equivalence classes of projections are the same as their $K_0$-classes), closed two-sided ideals of $B$ are in exact correspondence with the order ideals of $K_0(B)$. With $A_0$ the ideal of $B$ corresponding to the order ideal $H\subset H+ H'=K_0(B)$, denote by $A$ the hereditary \Css{} $eA_0e\subset A_0\subset B$ where $e\in B$ is a projection with $K_0(e)=1\in H'\cong \Z$. $A$ is as desired, i.e., $K_0(A^{\sim}\otimes \cZ)\cong G$. 
\end{proof}

\section{Classification of $\cZ$-absorbing RAF-algebras}

In this section, we prove that the invariant considered in Section 4, the ordered $K_0$-group together with what we propose to call the matrix dimension range (see Corollary \ref{Cor4.8}), is complete (Theorem \ref{Thm5.3} in the stable case and Corollary \ref{Cor5.4} in the non-stable case).

We begin with the observation that the hypothesis of a unit in the otherwise completely general  deformation isomorphism theorem for $\cZ$-absorbing separable \Cs s due to Winter---Proposition 4.5 of \cite{W1}---is not used, as is seen on replacing all unitaries appearing in the statement and in the proof by the corresponding quasiunitaries, with the convention that the inner automorphism determined by a quasiunitary is just that determined by the corresponding unitary in the unitization, obtained (by definition) by adding the unit to it.

\begin{proposition}[essentially Proposition 4.5 of \cite{W1}, cf.~proof of Theorem 14.3 of \cite{GL1}]\label{Prop5.1}
Let ${\frak p}$ and ${\frak q}$ be relatively prime supernatural numbers. Suppose that $A$ and $B$ are separable $\cZ$-absorbing $C^*$-algebras and let $\varphi : A\otimes \cZ_{\frak p, \frak q}$ $\rightarrow$ $B\otimes \cZ_{\frak p, \frak q}$ be a quasiunitarily  suspended $C([0, 1])$-isomorphism (as in \cite[Definition 4.2]{W1} with unitaries replaced by quasiunitaries). Then, there is an isomorphism $\tildevarphi : A\rightarrow B\otimes \cZ$ such that 
\[\tildevarphi \approx_{\rm au} (\id_B\otimes\barsigma_{\frak p, \frak q})\circ\varphi\circ(\id_A\otimes 1_{\cZ_{\frak p, \frak q}}), \]
where $\barsigma_{p, q}$ is the standard embedding $\cZ_{\frak p, \frak q}\hookrightarrow \cZ$ of \cite[Proposition 3.4]{W1}.
\end{proposition}  

\begin{proof}
The proof is exactly the same as the proof of Proposition 4.5 of \cite{W1}, with unitaries replaced by the quasiunitaries which in the present context (quasiunitarily suspended $C([0, 1])$-isomorphism) they correspond to.   Note that, when in the proof of \cite[Proposition 4.5]{W1} (spread over Sections 4.3, 4.4, and 4.5 of \cite{W1}), the product of two (or three) unitaries appears, these unitaries and therefore also the product correspond to quasiunitaries, which they should be replaced by. 
\end{proof}

\begin{lemma}[essentially Theorem 2.3 of \cite{Bl1}]\label{Lem5.2}
Let $A$ and $B$ be (separable) AF-algebras, and let $\varphi_0$ and $\varphi_1$ be $C^*$-algebra  homomorphisms from $A$ to $B$ that agree on the ordered $K_0$-group. It follows that $\varphi_0$ and $\varphi_1$ are (one-parameter) asymptotically quasiunitarily equivalent: there exists a one-parameter family of unitaries $u_t$, $t\in[0,1]$, in $B^{\sim}$ such that $u_t -1_{B^{\sim}}\in B$, $0\leq t\leq 1$ (i.e., $u_t -1_{B^{\sim}}$ is a quasiunitary in $B$), $u_0=1_{B^{\sim}}$ and $\lim\limits_{t\to 1} \Ad u_t\circ\varphi_0 = \varphi_1$. 
\end{lemma}

\begin{proof}
This holds by the proof of Theorem 2.3 of \cite{Bl1}.
\end{proof}
\begin{theorem}[cf.~Proposition 4.6 of \cite{W1}]\label{Thm5.3}
Suppose that $A$ and $B$ are stable separable RAF-algebras absorbing the Jiang-Su algebra tensorially. If there is an ordered group isomorphism $\gamma$ from $(K_0(A), K_0(A)_+)$ to $(K_0(B), K_0(B)_+)$, then $A$ is isomorphic to $B$ and there exists an isomorphism $\alpha$ from $A$ to $B$ such that $\alpha_*=\gamma$ at the level of $K_0$-groups.
\end{theorem}

\begin{proof}
By Definition \ref{Def4.1}, there are relatively prime supernatural numbers $\frak p$ and $\frak q$ such that $A\otimes M_{\frak p}$ and $A\otimes M_{\frak q}$ are AF. We may suppose (enlarging them if necessary) that $\frak p$ and $\frak q$ are infinite.  Hence by Lemma \ref{Lem4.2} (i), $B\otimes M_{\frak p}$ and $B\otimes M_{\frak q}$ are AF. It is immediate that the maps
\[ \gamma\otimes \id_{\D_{\frak p}} : K_0(A)\otimes \D_{\frak p} \rightarrow K_0(B)\otimes \D_{\frak p}\quad\text{ and }\quad \gamma \otimes \id_{\D_{\frak q}} : K_0(A)\otimes \D_{\frak q} \rightarrow K_0(B)\otimes \D_{\frak q}\]
are isomorphisms of ordered groups, and by Proposition \ref{Prop4.3} (i) these maps may be viewed as ordered group isomorphisms
\[ \gamma_{\frak p} : K_0(A\otimes M_{\frak p})\rightarrow K_0(B\otimes M_{\frak p})\quad\text{ and } 
\quad\gamma_{\frak q} : K_0(A\otimes M_{\frak q}) \rightarrow K_0(B\otimes M_{\frak q}).\]
By \cite{Ell01}, there are isomorphisms of AF-algebras 
\[ \varphi_{\frak p} : A\otimes M_{\frak p}\rightarrow  B\otimes M_{\frak p} \text{ with } \varphi_{\frak p *}= \gamma_{\frak p}\quad\text{ and } 
\quad\varphi_{\frak q} : A\otimes M_{\frak q} \rightarrow B\otimes M_{\frak q} \text{ with } \varphi_{\frak q *}= \gamma_{\frak q}.\]
Since $\varphi_{\frak p}\otimes \id_{M_{\frak q}}$ and $\varphi_{\frak q}\otimes \id_{M_{\frak p}}$ give rise to the same $K_0$-map (in the obvious sense), $\gamma\otimes \id_{\D_{\frak p}}\otimes\id_{\D_{\frak q}}$ $: K_0(A\otimes M_{\frak p}\otimes M_{\frak q})\rightarrow K_0(B\otimes M_{\frak p}\otimes M_{\frak q})$, by Lemma \ref{Lem5.2} there is a quasiunitarily suspended $C([0, 1])$-isomorphism $A\otimes \cZ_{\frak p, \frak q}$ $\rightarrow $ $B\otimes \cZ_{\frak p, \frak q}$ agreeing with $\varphi_{\frak p}$ and $\varphi_{\frak q}$ at the endpoints of the interval $[0, 1]$.

Hence by Proposition \ref{Prop5.1}, there exists an isomorphism $\tildevarphi : A\rightarrow B\otimes \cZ$ such that $\tildevarphi \approx_{\rm au} (\id_B\otimes\barsigma_{\frak p, \frak q})\circ\varphi\circ(\id_A\otimes 1_{\cZ_{\frak p, \frak q}})$. Hence (as in Proposition 4.6 of \cite{W1}), $\tildevarphi_*=\gamma$.
\end{proof}

\begin{corollary}\label{Cor5.4}
Let $A$ and $B$ be $\cZ$-absorbing, separable RAF-algebras (not necessarily stable). Suppose that the invariants of $A$ and $B$ described in Corollary \ref{Cor4.8} are isomorphic. Then this isomorphism is induced by an isomorphism of the $C^*$-algebras $A$ and $B$.
\end{corollary}
\begin{proof}
Note first that by $A^{\sim}\otimes \cZ$ and $B^{\sim}\otimes \cZ$ are RAF, as (see \cite{BE}) an extension of one AF-algebra by another is AF. By definition---see Corollary \ref{Cor4.8}---, there is an isomorphism of order-unit groups $K_0(A^{\sim}\otimes \cZ)$ and $K_0(B^{\sim}\otimes \cZ)$, respecting the subgroups $K_0(A)$ and $K_0(B)$. By Theorem \ref{Thm5.3}, there is an isomorphism of the stabilizations $A^{\sim}\otimes \cZ\otimes \cK$ and  $B^{\sim}\otimes \cZ\otimes \cK$ giving rise to the given isomorphism of $K_0$-groups. In particular, it takes the cutdown of $A^{\sim}\otimes\cZ\otimes\cK$ by $1_{A^{\sim}\otimes \cZ}\otimes e_{1 1}$ into the cutdown of $B^{\sim}\otimes\cZ\otimes \cK$ by a projection with $K_0$-class equal to the class of $1_{B^{\sim}\otimes\cZ}\otimes e_{ 1 1}$, and therefore (by cancellation---Corollary \ref{Cor4.8}) Murray-von Neumann equivalent to it---in fact unitarily equivalent to it as the \Cs{} is stable. So we may assume that the isomorphism takes $A^{\sim}\otimes \cZ= A^{\sim}\otimes \cZ\otimes e_{1 1}$ onto $B^{\sim}\otimes \cZ=B^{\sim}\otimes \cZ\otimes e_{1 1}$, and reproduces the given isomorphism of $K_0(A^{\sim}\otimes \cZ)$ with $K_0(B^{\sim}\otimes\cZ)$. 
 
 Now note that any stable RAF algebra has the ideal property---any closed two-sided ideal (let us just say ``ideal'') is generated (as an ideal) by its projections. This holds because it holds in an AF algebra. More precisely, by \cite{Br} (see also \cite{Dav}), every  ideal in an inductive limit \Cs{} is the inductive limit 
 of its finite-stage inverse limits, which shows that the ideals of the tensor product of any \Cs{} with a  UHF algebra are just the tensor products of the ideals of the given algebra with the UHF algebra. (Given that this is true for a finite matrix algebra in place of the UHF.) If the given algebra is RAF, so that the tensor product is AF, the ideal in the tensor product corresponding to a given ideal is generated (as an ideal) by projections in the AF algebra. 
 By Proposition \ref{Prop4.3} (i), a multiple of any $K_0$-element of the tensor product belongs to the $K_0$-group of the canonical image of the given algebra.  Since the tensor product is AF, and the given algebra is stable, 
 this says that any projection in the tensor product is equivalent to a projection in the given algebra, and of course in the same ideal. Such projections therefore generate the given ideal. 
 
 (In fact, this argument shows that, in the stable case, every ideal of an RAF algebra has an approximate 
 unit consisting of projections. Indeed, in the case of an ideal with compact spectrum, the corresponding ideal in the (AF) tensor product with a UHF algebra (which has the same spectrum) 
 is generated by a single projection---as is seen by looking at the finite-dimensional finite stages in an inductive limit decomposition---, and therefore the given ideal of the RAF algebra is generated by a single projection. Hence by Brown's theorem (\cite{Bro}) the given ideal (assumed to be separable, as well as stable) is isomorphic to the stabilization of  
 the cutdown by this projection, which has an approximate unit consisting of projections. Since the spectrum of a (separable) AF algebra, and therefore of an RAF-algebra, is an increasing union of compact open sets (the spectra of ideals generated by a single projection), and so the algebra is the closure of the corresponding increasing sequence of ideals, it follows that the whole (stable) 
 RAF-algebra has an approximate unit consisting of projections.
  This property could also be used to prove Theorem \ref{Thm5.3} above, using Proposition 4.5 of \cite{W1} directly (for the unital case), provided that one also established a uniqueness theorem.)

 The proof of the present non-stable isomorphism theorem is now in hand. The isomorphism of  the stabilized algebras $A^{\sim}\otimes \cZ\otimes\cK$ and  $B^{\sim}\otimes\cZ\otimes\cK$, giving rise to the given isomorphism of the order-unit groups $K_0(A^{\sim}\otimes \cZ)$ and $K_0(B^{\sim}\otimes\cZ)$, since this respects the canonical order ideals $K_0(A)=K_0(A\otimes \cZ)$ and $K_0(B)=K_0(B\otimes\cZ)$, and since the ideals $A\otimes\cZ\otimes\cK$ and $B\otimes\cZ\otimes\cK$ are generated by their projections (and because of cancellation---see proof of Corollary \ref{Cor4.8}---which implies that  equivalence classes of projections are the same as their $K_0$-classes), restricts to an isomorphism of $A=A\otimes\cZ$ with $B=B\otimes \cZ$ giving rise to the given isomorphism of $K_0(A)$ with $K_0(B)$. 
\end{proof}

\section{KMS states of $\mathcal{Z}$-absorbing \Cs s}
In order to show the main result, Theorem \ref{ThmMain}, let us prepare some facts for ideals and traces of RAF-algebras.
By an ideal of a \Cs, we shall mean a closed  two-sided ideal. For a \Cs{} $A$, we shall denote by $\cI_A$ the set of all ideals in $A$. For a supernatural number ${\frak n}$, we shall consider the map $\Phi_{\frak n}$ from $\cI_A$ to $\cI_{A\otimes M_{\frak n}}$ defined by $\Phi_{\frak n}(I)=I\otimes M_{\frak n}\subset A\otimes M_{\frak n}$ for $I\in \cI_A$. 

Let $(G, G^+)$ be an ordered abelian group and let $\gamma_j$, $j\in J$, be automorphisms of $G$ as an ordered abelian group. If no order ideal of $(G, G^+)$ other than $0$ or $G$ is invariant under all $\gamma_j$, $j\in J$, we shall say that $(G, G^+)$ is \emph{ $\{\gamma_j\}_{j\in J}$-simple}. For a \Cs{} $A$ and automorphisms $\alpha_j$, $j\in J$, of $A$, if no ideal of $A$ other than $0$ or $A$ is invariant under all $\alpha_j$, $j\in J$, we shall say that $A$ is \emph{ $\{\alpha_j\}_{j\in J}$-simple}. 

\begin{lemma}\label{Lem6.1}
{}\

\begin{itemize}
\item[]{\rm (i)} For any $C^*$-algebra $A$ and  supernatural number ${\frak n}$, the map $\Phi_{\frak n} : \cI_A\rightarrow \cI_{A\otimes M_{\frak n}}$ is bijective. 
\item[]{\rm (ii)} Suppose that $A$ is a $\cZ$-absorbing RAF-algebra and $\alpha_j$, $j\in J$, are automorphisms of $A$. If the ordered abelian group $(K_0(A), K_0(A)^+)$ is $\{\alpha_{j *}\}_{j\in J}$-simple, then $A$ is $\{\alpha_j\}_{j\in J}$-simple.
\end{itemize}
\end{lemma}

\begin{proof}
{}

\noindent{\rm (i)}.  Given an ideal of $I$ of $A\otimes M_{\frak n}$,  set $I_A=\{a\in A\ : \ a\otimes 1_{M_{\frak n}}\in I\}$, which is an ideal of $A$. Then $I_A\otimes M_{N}=I\cap (A\otimes M_{N})$ for any $N\in \N$ with $N | {\frak n}$. 
Indeed, if $x\in I\cap (A\otimes M_N)$ is written as $x=\sum\limits_{i, j=1}^N a_{i j}\otimes e_{i j}^{(N)}$ for some $a_{i j}\in A$ and system of matrix units $\{e_{ij}^{(N)}\}_{i, j=1}^N$ of $M_N$, then, using an approximate unit $h_{\lambda}$, $\lambda\in \Lambda$, of $A$ we have 
\[ a_{i j}\otimes e_{i j}^{(N)}=\lim_{{\lambda}\to \infty} (h_{\lambda}\otimes e_{i i}^{(N)}) x (h_{\lambda}\otimes e_{j j}^{(N)})\in I\cap (A\otimes M_N).\]
It follows that $a_{i j}\otimes 1_{M_N}\in I$ which means $a_{i j}\in I_A$ for all $i, j=1, 2, ..., N$. The converse inclusion $I_A\otimes M_N\subset I\cap(A\otimes M_N)$ is trivial.
Thus it follows that
\[ I= \overline{\bigcup_{N |{\frak n}} I\cap(A\otimes M_{N})}=\overline{\bigcup_{N |{\frak n}} I_A\otimes M_N} = \Phi_{\frak n}(I_A)\] 
(see \cite[Lemma 3.1]{Br}, and see also \cite[Lemma III 4.1]{Dav}).
It is straightforward to show the injectivity of $\Phi_{\frak n}$. Indeed, for $I_A$, $J_A\in \cI_A$ with $\Phi_{\frak n}(I_A)=\Phi_{\frak n}(J_A)$, choosing an approximate unit $k_{\lambda}$, $\lambda\in \Lambda$, of $I_A$, we have 
\[\lim_{\lambda\to \infty}(a k_{\lambda})\otimes 1_{M_{\frak n}}=\lim_{\lambda\to\infty} (a\otimes 1_{M_{\frak n}})(k_{\lambda}\otimes 1_{M_{\frak n}})=a\otimes 1_{M_{\frak n}},\]
 for any $a\in J_A$. Then it follows that $a=\lim\limits_{\lambda\to\infty} a k_{\lambda}\in I_A$. This shows that $J_A\subset I_A$.
 
 \noindent{\rm (ii)}. Let $\frak n$ be an infinite supernatural number and let $I\in \cI_A$ be invariant under all $\alpha_j$, $j\in J$. Then the ideal $\Phi_{\frak n}(I)\in \cI_{A\otimes M_{\frak n}}$ is also invariant under all $\alpha_j\otimes \id_{M_{\frak n}}$, $j\in J$. By {\rm (i)}, it suffices to show that $A\otimes M_{\frak n}$ is $\{\alpha_j\otimes \id_{M_{\frak n}}\}_{j\in J}$-simple. In the case that $A\otimes M_{\frak n}$ is an AF-algebra, it is well known that $\cI_{A\otimes M_{\frak n}}$ corresponds to the set of all order ideals of $(K_0(A\otimes M_{\frak n}), K_0(A\otimes M_{\frak n})^+)$. Hence it is enough to show that $(K_0(A\otimes M_{\frak n}), K_0(A\otimes M_{\frak n})^+)$ is $\{(\alpha_j\otimes \id_{M_{\frak n}})_*\}_{j\in J}$-simple. By Proposition \ref{Prop4.3} (i),  $(K_0(A\otimes M_{\frak n}), K_0(A\otimes M_{\frak n})^+)$ is isomorphic to $(K_0(A)\otimes \D_{\frak n}, {K_0(A)\otimes \D_{\frak n}}^+)$ as an ordered abelian group, and in such a way that $(\alpha_j\otimes \id_{M_{\frak n}})_*$, $j\in J$, corresponds to $\alpha_{j *}\otimes \id_{\D_{\frak n}}$, $j\in J$. Hence we only need to show that $(K_0(A)\otimes\D_{\frak n}, {K_0(A)\otimes \D_{\frak n}}^+)$ is $\{\alpha_{j *}\otimes \id_{\D_n}\}_{j\in J}$-simple.
 
 Let $H$ be an order ideal of $(K_0(A)\otimes \D_{\frak n}, {K_0(A)\otimes\D_{\frak n}}^+)$ which is invariant under all $\alpha_{j *}\otimes \id_{D_{\frak n}}$, $j\in J$. Set $H_A=\{g\in K_0(A)\ : \ g\otimes 1_{\D_{\frak n}}\in H\}$, and $H_A^+=H_A\cap K_0(A)^+$. Since $H^+=H\cap(K_0(A)\otimes \D_{\frak n})^+$ is hereditary (i.e., if $g\in {K_0(A)\otimes\D_{\ n}}^+$ and $h\in H^+$ satisfies $g\leq h$, then $g\in H^+$), it is trivial to see that $H_A^+$ is also hereditary. To show that $(H_A, H_A^+)$ is an ideal of $(K_0(A), K_0(A)^+)$, we must show that $H_A=H_A^+-H_A^+$. We shall use the fact that the ordered abelian group $(K_0(A), K_0(A)^+)$ is unperforated, by Proposition \ref{Prop4.3} (ii).  For $x\in H_A$, since $H=H^+-H^+$ there exist $y$, $z\in H^+$ such that $x\otimes 1_{\D_{\frak n}}= y-z$. From $y$, $z\in {K_0(A)\otimes \D_{\frak n}}^+$, we obtain $N\in \N$ with $N |{\frak n}$ and $y_A, z_A\in K_0(A)^+$ such that $Ny= y_A\otimes 1_{\D_{\frak n}}$ and $Nz=z_A\otimes 1_{\D_{\frak n}}$. Since $Nx=y_A-z_A \leq N y_A$, it follows that $x\leq y_A$, which implies that $x= y_A-(y_A-x)\in H_A^+-H_A^+$. Since $H$ is invariant under all $\alpha_{j *}\otimes \id_{\D_{\frak n}}$, $j\in J$, we see that $H_A$ is also invariant under all $\alpha_{j *}$, $j\in J$, which implies that $H_A=0$ or $K_0(A)$. Then it follows that $H=0$ or $K_0(A)\otimes \D_{\frak n}$, as required. 
 \end{proof}

\begin{proposition}\label{Prop6.2}
Let $A$ be a $\cZ$-absorbing RAF-algebra, $\alpha$ an automorphism of $A$, and $\sigma$ an automorphism of $\cZ$. Suppose that the ordered abelian group $(K_0(A), K_0(A)^+)$ is $\alpha_*$-simple and $\sigma$ has the weak Rohlin property (see \cite[Definition 1.1]{Sat0}).
Then the crossed product $C^*$-algebra $(A\otimes\cZ)\rtimes_{\alpha\otimes\sigma}\Z$ is simple.
\end{proposition}
\begin{proof}
Since $(K_0(A), K_0(A)^+)$ is $\alpha_*$-simple, it follows that $(K_0(A\otimes \cZ), K_0(A\otimes \cZ)^+)$ $(\cong (K_0(A),$ $ K_0(A)^+))$ is $(\alpha\otimes\sigma)_*$-simple. By Lemma \ref{Lem6.1} (ii),  therefore $A\otimes\cZ$ is $\alpha\otimes\sigma$-simple. Denote by $u$ the implementing unitary of $\alpha\otimes \sigma$. As in a similar argument in the proof of \cite[Theorem 3.2]{Ell0} (see also \cite{Kis1}), it suffices to show that for any $x\in A\otimes\cZ$, any finite subset $F$ of $\Z\setminus \{0\}$, and any finite family $\{y_i\}_{i\in F}\subset A\otimes \cZ$, 
\[ \|x\| \leq \| x +\sum_{i\in F} y_i u^i\|.\]

Set $k=\max\{ |i|\ : \ i\in F\}$, $k_0=k_1=k$, and $k_2=k+1$. Applying \cite[Theorem 6.4]{Lia} 
to $\sigma\in \Aut(\cZ)$, we obtain positive contractions $f_{j, n}^{(l)}\in \cZ$, $l=0, 1, 2$, $j=0,1,...,k_l$, $n\in\N$ such that 
\[\sum_{l=0}^2 \sum_{j=0}^{k_l} f_{j, n}^{(l)} =1_{\cZ},\]
each $( f_{j, n}^{(l)})_{n\in\N}$ is a central sequence in $\cZ$ for all $l=0, 1, 2$ and $j=0,1,..., k_l$,  
$\displaystyle \lim_{n\to \infty}\|\sigma(f_{j, n}^{(l)})-f_{j+1, n}^{(l)}\| =0$ for all $l=0, 1, 2$ and $0\leq j \leq k_l$ mod $(k_l+1)$, and $\displaystyle \lim_{n\to\infty} \| f_{i, n}^{(l)}f_{j, n}^{(l)}\|=0$ for all $l=0, 1, 2$ and  $0\leq i \neq j \leq k_l$. 
 Let $h_n$, $n\in\N$, be an increasing approximate unit of $A$. For any $n\in\N$,  define a completely positive map $\Phi_n$ from $(A\otimes\cZ)\rtimes_{\alpha\otimes\sigma}\Z$ to $(A\otimes\cZ)\rtimes_{\alpha\otimes\sigma}\Z$ by 
\[\Phi_n(a) =\sum_{l=0}^2 \sum_{j=0}^{k_l}(h_n\otimes f_{j, n}^{(l)})^{1/2} a(h_n\otimes f_{j, n}^{(l)})^{1/2},\]
for $a\in (A\otimes\cZ)\rtimes_{\alpha\otimes\sigma}\Z$. 
Note that $\Phi_n$ is a contraction for each $n\in\N$. 
Since $(h_n\otimes f_{j, n}^{(l)})_{n\in\N}$ is a central sequence in $A\otimes \cZ$, it follows that $\displaystyle \lim_{n\to\infty}\|\Phi_n(a)-a\|=0$ for any $a\in A\otimes \cZ$. Since 
\[ \lim_{n\to\infty} \|\sigma^{i}(f_{j, n}^{(l)})f_{j, n}^{(l)}\| =\lim_{n\to\infty}\| f_{i+j, n}^{(l)}f_{j, n}^{(l)}\|=0,\]
where $i+j$ is considered mod $(k_l+1)$ for $i\in F$ and $j\in \{0,1,..., k_l\}$, we have  $\displaystyle \lim_{n\to\infty} \|\Phi_n(u^i)\|=0$ for any $i\in F$ (where $\Phi_n(u^i)$ is defined in a natural way). Then we have
\begin{align*}
\|x\|&\leq \limsup_{n\to\infty} \|\Phi_n(x +\sum_{i\in F} y_i u^i)\|+ \|\Phi_n(\sum_{i\in F} y_i u^i)\| \\
&\leq \|x +\sum_{i\in F} y_iu^i\|+\limsup_{n} \sum_{i\in F}\| y_i\Phi_n(u^i)\|\\
&=\|x + \sum_{i\in F} y_i u^i\|.
\end{align*}

\end{proof}

For a \Cs{} $A$, we denote by $\cT(A)$ the cone of densely defined lower semicontinuous  traces of $A$. With $\Ped(A)$ the Pedersen ideal of $A$,  it is well known that $\tau(p)<\infty$ and $p\in\Ped(A)$ for any $\tau\in \cT(A)$ and  projection $p$ in $A$. Set $\Ped(A)^+=\Ped(A)\cap A^+$. For $n\in\N$, taking the normalized trace $\tr_n$ of $M_n$ we set $\tau\otimes\tr_n(a\otimes b)=\tau(a)\tr_n(b)$ for $\tau\in\cT(A)$, and $a\otimes b\in\Ped(A)\otimes M_n$. By the following Lemma \ref{Lem6.4} (i), we will see that $\Ped(A)\otimes M_n=\Ped(A\otimes M_n)$, so that $\tau\otimes \tr_n$ can be regarded as an element of $\cT(A\otimes M_n)$. To simplify, we use the same symbol $\tau$ for $\tau\otimes\tr_n$. 

For an ordered abelian group $(G, G^+)$, we define $S_0(G)$, as the set of positive group homomorphisms from $G$ to $\R$. In the case that $(K_0(A), K_0(A)^+)$ is an ordered abelian group, and $A$ has cancellation, we define the standard affine map $\PhiT : \cT(A)\rightarrow S_0(K_0(A))$ by $\PhiT(\tau) ([p]_0)=\tau(p)$ for $[p]_0\in K_0(A)^+$. 
Note that, by the cancellation property of $A$, projections $p$ and $q$ in $A\otimes M_n$, $n\in\N$, are Murray-von Neumann equivalent if and only if $[p]_0=[q]_0$ in $K_0(A)$. Then $\PhiT(\tau)$ is well defined for $\tau\in \cT(A)$. 
We shall often use the notation $\tau_*$ for $\PhiT(\tau)$.

Consider the subset  $A_0$ of the \Cs{} $A$ defined by $A_0=\{a\in A : a\in pAp \text{ for\- some}$ $\text{\- pro\-je\-ction }$ $p\in A \}$. Equip $\cT(A)$ with the smallest topology such that $\cT(A)\ni \tau\mapsto \tau(a)$ is continuous for any $n\in\N$ and $a\in (A\otimes M_n)_0$, and equip $S_0(G)$ with the topology of pointwise convergence. 
 
 \begin{lemma}\label{Lem6.4}
 {}\
 \begin{itemize}
 \item[]{\rm (i)} For a $C^*$-algebra $A$ and $n\in\N$, it follows that $\Ped(A)\otimes M_n=\Ped(A\otimes M_n)$.
 \item[]{\rm (ii)} If $A$ is a real rank zero $C^*$-algebra with cancellation, and $(K_0(A), K_0(A)^+)$ is an ordered abelian group, then a net $\tau_{\lambda}\in \cT(A)$, $\lambda\in\Lambda$, converges to $\tau\in \cT(A)$ if and only if the net $\PhiT(\tau_{\lambda})$, $\lambda\in\Lambda$, converges to $\PhiT(\tau)$ in $S_0(K_0(A))$. 
 \end{itemize}
 \end{lemma}
 \begin{proof}
 {}\
  \noindent{\rm (i)}. Let us first check the inclusion $\Ped(A)\otimes M_n\subset \Ped(A\otimes M_n)$. From the construction of the Pedersen ideal, one sees that $\Ped(A)\otimes e_{i i}^{(n)}\subset \Ped (A\otimes M_n)$, $i=1,2,...,n$. Since the Pedersen ideal is equal to its square, it is an ideal of the multiplier algebra of the \Cs. Applying this to $\Ped(A\otimes M_n)$, we see that also $\Ped(A)\otimes e_{i j}^{(n)}\subset\Ped(A\otimes M_n)$, $i, j=1,2,...,n$.
  
 Since the Pedersen ideal is the smallest dense two-sided  ideal, the two ideals are equal.

 \noindent{\rm (ii)}. The ``only if'' part of the statement is obvious. Assume that $\PhiT(\tau_{\lambda})$, $\lambda\in\Lambda$, converges to $\PhiT(\tau)$ in $S_0(K_0(A))$. Let $x$ be a self-adjoint element in $p(A\otimes M_n)p$ for some $n\in\N$ and a projection $p\in A\otimes M_n$. Without loss of generality, we may assume that $\tau(p)>0$. Since $p(A\otimes M_n)p$ has real rank zero, for $\varepsilon >0$ there exists a self-adjoint element $\barx\in p(A\otimes M_n)p$ with finite spectrum such that $\|x-\barx\|<\frac{\varepsilon}{4\tau(p)}$. By 
 \[\barx -\frac{\varepsilon}{4\tau(p)}p\leq x \leq \barx +\frac{\varepsilon}{4\tau(p)}p,\]
 it follows that $|\tau_{\lambda}(x) -\tau_{\lambda}(\barx)|\leq\varepsilon\tau_{\lambda}(p)/(4\tau(p))$ for any $\lambda\in\Lambda$. By $\lim\limits_{\lambda\to\infty} \tau_{\lambda}(\barx) =\tau(\barx)$ and $\lim\limits_{\lambda\to\infty}\tau_{\lambda}(p)=\tau(p)$, we have that $\limsup\limits_{\lambda\to\infty}|\tau_{\lambda}(x) -\tau(x)|< \varepsilon$. Since $\varepsilon>0$ is arbitrary, a net $\tau_{\lambda}$, $\lambda\in\Lambda$ converges to $\tau$ in the topology of $\cT(A)$.
 \end{proof}

 In the case of a RAF-algebra $A$, we note that for $\tau\in\cT(A)$  the map $\tau_*$ on $K_0(A)$ is also well defined. Indeed, for $N\in\N\setminus\{1\}$ and two projections $p$, $q\in A\otimes M_N$, if $[p]_0=[q]_0$ in $K_0(A)$, then $[p\otimes 1_{M_{N^{\infty}}}]_0=[q\otimes 1_{M_{N^{\infty}}}]_0$ in $K_0(A\otimes M_{N^{\infty}})$. Since $A\otimes M_{N^{\infty}}$ is an AF algebra, there exists a partial isometry $r\in A\otimes M_{N^{\infty}}$ such that $r^*r=p\otimes 1_{M_{N^{\infty}}}$ and $rr^*=q\otimes 1_{M_{N^{\infty}}}$. Taking a large $k\in\N$ we can obtain such an $r$ in $A\otimes M_{N^k}$. Then it follows that $\tau(p)=\tau\otimes\tr_{N^k}(p\otimes 1_{M_{N^k}})=\tau\otimes\tr_{N^k}(q\otimes 1_{M_{N^k}})=\tau(q)$. Thus, for a RAF-algebra $A$, we can define $\PhiT : \cT(A)\rightarrow S_0(K_0(A))$.

\begin{proposition}\label{Prop6.4}
If $A$ is a $\cZ$-absorbing $RAF$-algebra, then the map $\PhiT$ $:\cT(A)\rightarrow S_0(K_0(A))$ is an affine homeomorphism.
\end{proposition}
\begin{proof}
Note that, by Proposition \ref{Prop4.3} (ii),  $(K_0(A), K_0(A)^+)$ is an ordered abelian group.
Let ${\frak n}$ be an infinite supernatural number and $\iota$ the embedding of $A$ into $A\otimes M_{\frak n}$ defined by $\iota(a) = a\otimes 1_{M_{\frak n}}$ for $a\in A$. Since $A\otimes M_{\frak n}$ is an AF-algebra, we see that the canonical map $\Phi_{\cT, A\otimes M_{\frak n}} : \cT(A\otimes M_{\frak n}) \rightarrow S_0(K_0(A\otimes M_{\frak n}))$ is bijective; see the last paragraph of \cite{Cu} and \cite[Lemma 3.5]{Thomsen1}, for example.

Denote by $\tau_{\frak n}$ the unique tracial state of $M_{\frak n}$. Regarding $(K_0(A\otimes M_{\frak n}), K_0(A\otimes M_{\frak n})^+)$ as $(K_0(A)\otimes K_0(M_{\frak n}), K_0(A)\otimes K_0(M_{\frak n})^+)$ by Proposition \ref{Prop4.3} (i), define $\tau_*\otimes\tau_{\frak n *}\in S_0(K_0(A\otimes M_{\frak n}))$ for $\tau \in \cT(A)$ by $\tau_*\otimes \tau_{{\frak n} *}([p]_0\otimes [q]_0)=\tau(p)\tau_{\frak n}(q)$ for $l$, $m\in\N$ and projections $p\in A\otimes M_{l}$, $q\in A\otimes M_m$. The expression $\tau\otimes \tau_{\frak n}$ denotes the densely defined lower semicontinuous trace on $A\otimes M_{\frak n}$ such that $\Phi_{\cT, A\otimes M_{\frak n}}(\tau\otimes\tau_{\frak n})=\tau_*\otimes\tau_{{\frak n} *}$. In the following argument, we shall show that $\tau\otimes\tau_{\frak n}(a\otimes 1_{M_{\frak n}})=\tau(a)$ for any $a\in\Ped(A)$, where note that $a\otimes 1_{M_{\frak n}}\in \Ped(A\otimes M_{\frak n})$ for $a\in \Ped(A)$.

For $\varepsilon >0$,  consider the continuous function $f_{\varepsilon}$ defined by \[f_{\varepsilon}(t)=\max\{0,\- \min\{1,\- (t-\varepsilon)\}\}\quad\text{ for } t\in\R,\]  and set $a_{\varepsilon}=f_{\varepsilon}(a)$ for $a\in \Ped(A)^+$. Since $\tau\otimes\tau_{\frak n}$ and $\tau$ are lower semicontinuous, it suffices to see that $\tau\otimes\tau_{\frak n}(a_{\varepsilon}\otimes 1_{M_{\frak n}})=\tau(a_{\varepsilon})$ for 
any $\varepsilon >0$ and a contraction $a\in\Ped(A)^+$. 
Fix $\varepsilon >0$ and a contraction $a\in\Ped(A)^+$. Then there exists $\bara\in\Ped(A)^+$ such that $\bara a_{\varepsilon} = a_{\varepsilon}$. 
Since $\overline{a_{\varepsilon} A a_{\varepsilon}}\otimes M_{\frak n}=\overline{(a_{\varepsilon}\otimes 1_{M_{\frak n}})A\otimes M_{\frak n}(a_{\varepsilon}\otimes 1_{M_{\frak n}})}$ is an AF-algebra, 
for $n\in\N$ there exist $N_n\in\N$ with $N_n| \frak{n}$ and a positive element $a_n\in \overline{a_{\varepsilon}Aa_{\varepsilon}}\otimes M_{N_n}$ with finite spectrum such that 
$\| a_n -a_{\varepsilon}\otimes 1_{M_{\frak n}}\|< 1/n$. Since $a_n-\frac{1}{n}(\bara\otimes 1_{M_{\frak n}})\leq a_{\varepsilon}\otimes 1_{M_{\frak n}}\leq a_n +\frac{1}{n}(\bara\otimes 1_{M_{\frak n}})$, we have $\tau\otimes\tau_{\frak n}(a_{\varepsilon}\otimes 1_{M_{\frak n}})=\lim\limits_{n\to\infty}\tau\otimes \tau_{\frak n}(a_n)$. On the other hand, since $\tau\otimes\tau_{\frak n}\in\cT(A\otimes B)$ corresponds to $\tau_*\otimes\tau_{{\frak n} *}\in S_0(K_0(A))$, it follows that $\tau\otimes\tau_{\frak n}(p)=\tau\otimes \tr_{N_n}(p)$ for any projection $p\in \overline{a_{\varepsilon} A a_{\varepsilon}}\otimes M_{N_n}$. 
Thus we have 
\[ \tau(a_{\varepsilon})=\lim_{n\to\infty}\tau\otimes\tr_{N_n}(a_n)=\lim_{n\to \infty} \tau\otimes\tau_{\frak n}(a_n)=\tau\otimes\tau_{\frak n}(a_{\varepsilon}\otimes 1_{M_{\frak n}}).\]

Define a map $\Psi :\cT(A)\rightarrow \cT(A\otimes M_{\frak n})$ by $\Psi(\tau)=\tau\otimes \tau_{\frak n}$ for $\tau\in\cT(A)$. From $\tau(a)=\Psi(\tau)(a\otimes 1_{M_{\frak n}})$ for any $a\in\Ped(A)$, it follows that $\Psi$ is injective and affine. For $\varphi \in \cT(A\otimes M_{\frak n})$, setting $\varphi_A(a)=\varphi(a\otimes 1_{M_{\frak n}})$ for $a\in \Ped(A)$ we obtain $\varphi_A\in \cT(A)$. Since $\tau_{\frak n}$ is a unique tracial state, it follows that $\varphi(p\otimes x)=\varphi_A(p)\tau_{\frak n}(x)$ for any projection $p\in A$ and $x\in M_{\frak n}$. Thus we see that $\Psi(\varphi_A)=\varphi$.

Let $\iota_* : K_0(A) \rightarrow K_0(A\otimes M_{\frak n})$ denote the induced map defined by $\iota : A\rightarrow A\otimes M_{\frak n}$. Let us show that the map $\hatiota_* : S_0(K_0(A\otimes M_{\frak n}))\rightarrow S_0(K_0(A))$ defined by $\hatiota_*(\varphi) = \varphi\circ \iota_*$ is bijective.  
Indeed, for any $x\in K_0(A\otimes M_{\frak n})\cong K_0(A)\otimes K_0(M_{\frak n})$ there exists $N\in\N$ such that $Nx\in \iota_*(K_0(A))$. Thus we see that $\hatiota_*$ is injective. Since $K_0(A)$ is torsion free, $\iota_*$ is injective, which implies the surjectivity of $\hatiota_*$.  Since $\PhiT$ can be decomposed as $\PhiT =\hatiota_*\circ\Phi_{\cT, A\otimes M_{\frak n}}\circ\Psi$, it follows that $\PhiT$ is a bijective affine map.

It is straightforward to check that $\PhiT$ is continuous. It remains to show that $\PhiT^{-1}$
 is also continuous. Suppose that $\tau_{\lambda}$, $\lambda\in \Lambda$ is a net in $\cT(A)$ such that $\PhiT(\tau_{\lambda})$ converges to $\PhiT(\tau)$ for some $\tau\in\cT(A)$.  
 Then it follows that the net $\tau_{\lambda *}\otimes\tau_{{\frak n} *}$, $\lambda\in\Lambda$ in $S_0(K_0(A\otimes M_{\frak n}))$ converges to $\tau_*\otimes \tau_{{\frak n} *}$ in the topology of pointwise convergence. Applying Lemma \ref{Lem6.4} (ii) to $A\otimes M_{\frak n}$, we have that $\tau_{\lambda}\otimes\tau_{\frak n}$, $\lambda\in\Lambda$ converges to $\tau\otimes\tau_{\frak n}$ in $\cT(A\otimes M_{\frak n})$. Since $(A\otimes M_N)_0\otimes 1_{M_{\frak n}}\subset(A\otimes M_N\otimes M_{\frak n})_0$ for any $N\in\N$, we conclude that $\displaystyle \lim_{\lambda\to \infty}\tau_{\lambda}(a)=\tau(a)$ for any $a\in (A\otimes M_N)_0$.
 
 \end{proof}

 We are now ready to give our application of the RAF-algebra classification. 
  
 \begin{theorem}\label{ThmMain}
Let $(S, \pi)$ be a proper simplex bundle such that $\pi^{-1}(0)$ is a singleton. Then there exists a $2\pi$-periodic flow $\theta$ on the Jiang-Su algebra whose KMS-bundle is isomorphic to $(S, \pi)$.  
 \end{theorem}
 
For a given proper simplex bundle $(S, \pi)$ with $\pi^{-1}(0)=\{\tau_S\}$, in Section 3 we have constructed the rational dimension group $(G_{\Z}, G_{\Z}^+)$ and its shift automorphism $\sigma$ such that $(S_{\sigma}(G_{\Z}), \pi_{\Z})$ is isomorphic to $(S, \pi)$ as a simplex bundle (Proposition \ref{Prop3.5} {\rm (iii)}). 

We define a positive group homomorphism $\Sigma_0 : G_{\Z}\rightarrow \Z$ by $\Sigma_0((z_n)_{n\in\Z}\oplus g) =\sum_{n\in\Z} z_n=g(\tau_S)$ for $(z_n)_n\oplus g\in G_{\Z}$. By the same arguments as in the proofs of \cite[Lemma 4.10]{ETh}, \cite[Lemma 3.7]{EST},  and \cite[Lemma 4.14]{ETh}, the rational dimension group $(G_{\Z}, G_{\Z}^+)$ has the following properties.

\begin{lemma}\label{Lem6.6}
{}\ 

\begin{itemize}
\item[]{\rm (i)} $(G_{\Z}, G_{\Z}^+)$ is $\sigma$-simple, and
\item[]{\rm (ii)} $\displaystyle(\id-\sigma)(G_{\Z})=\ker (\Sigma_0)$.
\end{itemize}
\end{lemma}

From Theorem \ref{Thm4.4} and Theorem \ref{Thm5.3}, we obtain a stable $\cZ$-absorbing RAF-algebra $A$ and an automorphism $\alpha_{\Z}$ on $A$ such that 
\[ (K_0(A), K_0(A)^+)\cong (G_{\Z}, G_{\Z}^+)\quad\text{and}\quad \alpha_{\Z *}=\sigma\text{ on } K_0(A).\]
By Proposition \ref{Prop6.2} and the same argument in \cite[Remark 2.8]{MS} (see also \cite[Lemma 3.6]{EST}), $A$ and  $\alpha_{\Z}$ can be chosen to have the following properties.

\begin{lemma}\label{Lem6.7}
{}\

\begin{itemize}
\item[]{\rm (i)} $A\rtimes_{\alpha_{\Z}} \Z$ is a simple \Cs, 
\item[]{\rm (ii)}The restriction map $\tau\mapsto \tau|_A$ from $\cT(A\rtimes_{\alpha_{\Z}}\Z)$ onto the $\alpha_{\Z}$-invariant traces in $\cT(A)$ is bijective, and
\item[]{\rm (iii)} $A\rtimes_{\alpha_{\Z}}\Z$ absorbs the Jiang-Su algebra tensorially.
\end{itemize}
\end{lemma}

Set $C=A\rtimes_{\alpha_{\Z}} \Z$, and denote by $P$ the conditional expectation from $C$ onto $A$. Since $A$ is a stable \Cs{}, for $u\in G_{\Z}^+\cong K_0(A)^+$ defined as $u=(1_0, 1)$ in Section 3, there exists a projection $p$ in $A$ such that $[p]_0=u$. Then the unital \Cs{} $pCp$ has the following properties.

\begin{proposition}\label{Prop6.8}
{}\
\begin{itemize}
\item[]{\rm (i)} $pCp$ has a unique tracial state,
\item[]{\rm (ii)} $(K_0(pCp), K_0(pCp)^+)\cong (\Z, \Z^+)$,\quad $[p]_0=1\in\Z$,
\item[]{\rm (iii)} $K_1(pCp)\cong 0$, and
\item[]{\rm (iv)} therefore $pCp$ is isomorphic to $\cZ$.
\end{itemize}
\end{proposition}
\begin{proof}
{}\
\noindent{\rm (i)}. Set $\cT_{\alpha_{\Z}}^0(A)=\{\tau\in\cT(A)\ : \ \tau\circ\alpha_{\Z}=\tau, \text{ and } \tau(p)=1\}$. By Proposition \ref{Prop6.4}, $\cT_{\alpha_{\Z}}^0(A)$ corresponds to $S_{\sigma}^1(G_\Z)=\{\tau_S\}$ via $\PhiT$. Because of \cite[Proposition 4.7]{CP}, the set of tracial states $T(pCp)$ of $pCp$ corresponds to $\{\tau\in\cT(C)\ :\ \tau(p)=1\}$ via the restriction on $pCp$. By Lemma \ref{Lem6.7} (ii), 
we know that $\{\tau\in\cT(C)\ : \ \tau(p)=1\}$ is a singleton, which implies that $T(pCp)$ is also singleton.

\noindent{\rm (ii)}. By Lemma \ref{Lem4.2} (ii), we see that $K_1(A)=0$. Then the Pimsner-Voiculescu six-term exact sequence allows us to see that
\[K_0(C)\cong G_{\Z}/(\id-\sigma)(G_{\Z})= G_{\Z}/\ker(\Sigma_0) \cong\Z.\]
Consider  the quotient map $q: G_{\Z}\rightarrow G_{\Z}/(\id-\sigma)(G_{\Z})$ and denote by $\iota_* : K_0(A)\rightarrow K_0(C)$ the map induced  by the inclusion $\iota : A\rightarrow C$. In the exact sequence, precisely, $\iota_*$ corresponds to $q$ and the above group isomorphism $\Sigma : K_0(C)\rightarrow \Z$ is determined by $\Sigma\circ q=\Sigma_0$; see also the paragraph after \cite[Lemma 3.7]{EST}. Set $\tau^0\in\cT_{\alpha_{\Z}}^0(A)$ which corresponds to $\tau_S\in S_{\sigma}^1(G_{\Z})$, (determined by $\tau^0_*=\widehat{\tau_S}$). For $x\in K_0(C)^+$, now we obtain $(y_n)_{n\in\Z}\oplus f\in G_{\Z}$ such that $q((y_n)_n \oplus f))=x$. Then it follows that 
\[ 0\leq (\tau^0\circ P)_*(x) =\tau_*^0((y_n)_n\oplus f)=f(\tau_S)=\sum_{n\in\Z} y_n\in\Z,\]
which implies that $\Sigma (x) \in \Z^+$. The converse inclusion $\Sigma(K_0(C)^+)\supset \Z^+$ follows from the properties $nu\in K_0(A)^+$ and $n=\Sigma\circ q(nu)\in \Sigma(K_0(C)^+)$ for any $n\in\Z^+$. Then we have that $(K_0(pCp), K_0(pCp)^+)\cong(K_0(C), K_0(C)^+)\cong(\Z, \Z^+)$. In particular, since $1=\Sigma\circ q(u)$, $[p]_0$ corresponds to $1\in\Z$.

\noindent{\rm (iii)}. 
For the same reason as in \cite[Lemma 3.11]{EST}, we see that the map $\id-\sigma$ is injective on $K_0(A)\cong G_{\Z}$. From the Pimsner-Voiculescu exact sequence, we see that $K_1(C)=0$. 

\noindent{\rm (iv)}.
By Lemma \ref{Lem6.7},  we have now shown that $pCp$ is a unital separable simple nuclear \Cs{} which absorbs the Jiang-Su algebra tensorially and satisfies the universal coefficient theorem. By \rm{(i)}, \rm{(ii)}, \rm{(iii)} above, we see that $pCp$ is monotracial and the Elliott invariant of $pCp$ is same as for $\cZ$. From \cite[Corollary 4.6]{Ror}, we also see that $pCp$ has strict comparison. Then by \cite[Corollary 6.2]{MS2}, we conclude that $pCp\cong \cZ$.

\end{proof}

\begin{proof}[Proof of Theorem \ref{ThmMain}]
Denote by $\widetilde{\theta}$  the dual action of $\alpha_{\Z}$ on $C$ and by $\theta$ the restriction of $\widetilde{\theta}$ to $pCp\cong\cZ$. Denote by $(S_{\theta}, \pi_{\theta})$ the KMS-bundle of the dynamical system $(\cZ, \theta)$.
By Proposition \ref{Prop3.5} (iii),  $(S_{\sigma}(G_{\Z}), \pi_{\Z})$ is a proper simplex bundle isomorphic to $(S, \pi)$. It remains to show that $(S_{\sigma}(G_{\Z}), \pi_{\Z})$ is isomorphic to $(S_{\theta}, \pi_{\theta})$. Set 
\[\cT_{\alpha_{\Z}}(A)=\{ (\tau, \beta) \in \cT(A)\times\R \ : \ \tau\circ\alpha_{\Z}= e^{-\beta}\tau,\quad \tau(p)=1\},\]
equipped with the product topology on $\cT(A)\times\R$, and denote by $\pi_{\alpha_{\Z}} : \cT_{\alpha_{\Z}}(A)\rightarrow\R$  the projection. Define a continuous map $\Phi_{\alpha_{\Z}} : \cT_{\alpha_{\Z}}(A)\rightarrow S_{\sigma}(G_{\Z})$ by $\Phi_{\alpha_{\Z}}(\tau, \beta)=(\PhiT(\tau), \beta)$ for $(\tau, \beta)\in \cT_{\alpha_{\Z}}(A)$. By Proposition \ref{Prop6.4}, $\Phi_{\alpha_{\Z}}$ is a homeomorphism, and $\Phi_{\alpha_{\Z}}^{-1}|_{\pi_{\Z}^{-1}(\beta)}$ is affine for each $\beta\in\R$. 

It follows that the projection $p\in A$ (of Proposition \ref{Prop6.8}) is full  in $C$, i.e., $\overline{CpC}=C$, by the simplicity of $C$. 
Define a map
$\Psi : \cT_{\alpha_{\Z}}(A)\rightarrow S_{\theta}$ by $\Psi((\tau, \beta))=(\tau\circ P|_{pCp}, \beta)$
for $(\tau, \beta)\in\cT_{\alpha_{\Z}}(A)$. By \cite[Lemma 4.1]{ETh}, we see that $\Psi$ is bijective and $\Psi|_{\pi_{\alpha_{\Z}}^{-1}(\beta)}$ is an affine homeomorphism from $\pi_{\alpha_{\Z}}^{-1}(\beta)$ onto $\pi_{\theta}^{-1}(\beta)$. Since $pAp\subset A_0$, if a sequence $\tau_n\in \cT(A)$, $n\in\N$, converges to $\tau\in\cT(A)$ then $\tau_n\circ P|_{pCp}\rightarrow \tau\circ P|_{pCp}$ in the topology of pointwise convergence. Thus, $\Psi$ is continuous. We define a bijective continuous map $\Phi: S_{\sigma}(G_{\Z})\rightarrow S_{\theta}$ by $\Phi=\Psi\circ\Phi_{\alpha_{\Z}}^{-1}$. By Lemma \ref{Lem3.3}, we conclude that $(S_{\sigma}(G_{\Z}), \pi_{\Z})$ is isomorphic to $(S_{\theta}, \pi_{\theta})$, as desired.
\end{proof}

\begin{theorem}
Let $A$ be a unital separable \Cs{} with a unique tracial state. Suppose that 
$A$ absorbs the Jiang-Su algebra tensorially. 
Then for any proper simplex bundle $(S, \pi)$ such that $\pi^{-1}(0)$ is singleton,
there exists a $2\pi$-periodic flow on $A$ whose KMS-bundle is isomorphic to $(S, \pi)$.
\end{theorem}

\begin{proof}
Let $\widetilde{\iota}_t=\id_A$, $t\in\R$, denote the trivial flow on $A$ and $\tau$ the unique tracial state of $A$. For a given $(S, \pi)$, by Theorem \ref{ThmMain} we obtain a flow $\theta$ on $\mathcal{Z}$ whose KMS-bundle is isomorphic to $(S, \pi)$. We denote by  $A\otimes \mathcal{Z}$ the \Cs{}  tensor product and define a flow $\alpha$ on $A\otimes \mathcal{Z}$ as the tensor product action which is defined by $\alpha_t(a\otimes b) =\widetilde{\iota}_t(
a)\otimes\theta_t(b)$ for $a\in A$ and $b\in\mathcal{Z}$. 
Let $(S_{\alpha}, \pi_{\alpha})$ denote the KMS-bundle of $\alpha$ and $S_{\theta}^{\beta}(\cZ)$ (resp.~$S_{\alpha}^{\beta}(A\otimes \cZ)$) the  set of $\beta$-KMS states for $\theta$ (resp.~$\alpha$). 
Define an affine map $\Phi$ from $S(\mathcal{Z})$ to $S(A\otimes\mathcal{Z})$ by $\Phi(\varphi) =\tau\otimes\varphi$. Then it is straightforward to check that $\Phi(\varphi)$ is a $\beta$-KMS state for $\alpha$ for any $\varphi\in S_{\theta}^{\beta}(\mathcal{Z})$, and $\Phi |_{S_{\theta}^{\beta}(\cZ)}$ $: S_{\theta}^{\beta}(\cZ)\rightarrow S_{\alpha}^{\beta}(A\otimes\cZ)$ is affine continuous and injective for any $\beta\in\R$. Define a continuous map $\Phi_A$ $: S_{\theta}\rightarrow S_{\alpha}$ by $\Phi_A((\varphi, \beta))=(\Phi(\varphi), \beta)$ for $(\varphi, \beta)\in S_{\theta}$.

In order to apply Lemma \ref{Lem3.3} to $\Phi_A$, it remains to show surjectivity of $\Phi |_{S_{\theta}^{\beta}(\cZ)}$. Given $\psi\in S_{\alpha}^{\beta}(A\otimes\cZ)$,  $x\in \mathcal{Z}$, and $a, b\in A$, note that $a\otimes 1_{\mathcal{Z}}$ and $b\otimes 1_{\mathcal{Z}}$ are analytic  elements for $\alpha$ and $\psi(ab\otimes x)=\psi((b\otimes x)\alpha_{i\beta}(a\otimes 1_{\mathcal{Z}})) = \psi(ba\otimes x)$. 
Since $\tau$ is a unique tracial state, we have $\tau(a)\psi(1_A\otimes x)=\psi(a\otimes x)$, for any $a\in A$ and $x\in \mathcal{Z}$. Set $\psi_{\mathcal{Z}}(x)=\psi(1_A\otimes x)$, for $x\in \mathcal{Z}$. Then $\psi_{\mathcal{Z}}$ is contained in $S_{\theta}^{\beta}(\mathcal{Z})$,  as $\psi(1_A\otimes xy)=\psi((1_A\otimes y)\alpha_{i\beta}(1_A\otimes x))$ $=\psi(1_A\otimes y(\theta_{i\beta}(x)))$ for any analytic elements $x, y\in\mathcal{Z}$ for $\theta$. Thus we see that $\Phi(\psi_{\mathcal{Z}})(a\otimes x)=\psi(a\otimes x)$ for any $a\in A$ and $x\in \mathcal{Z}$. 
\end{proof}

From \cite[Theorem 3.2]{PS}, \cite[Proposition 2.1]{Kis}, and \cite[Theorem 1.1]{MS1}, we obtain uncountably many flows which are not approximately inner on the following classifiable class of \Cs s.
\begin{corollary}[cf.~Corollary 4.2 of \cite{EST}]
Any unital separable simple amenable monotracial \Cs{} with strict comparison has uncountably many flows which are not approximately inner, up to cocycle conjugacy and the trivial scaling equivalence.
\end{corollary}

\noindent{\bf Acknowledgements.}\quad  The authors would like to thank Professor Klaus Thomsen for helpful comments on this research. They also wish to express their gratitude to Professor Huaxin Lin and the organizers of Special Week on Operator Algebras 2021.

\noindent{Department of Mathematics, University of Toronto, Toronto, 
Canada ~\  M5S 2E4}\\
e-mail: { elliott@math.toronto.edu}

\bigskip

\noindent{Graduate School of Mathematics, Kyushu University, 744 Motoka, Nishi-ku, Fukuoka, Japan}\\
e-mail: {ysato@math.kyushu-u.ac.jp }

\bigskip

\end{document}